\documentclass[12pt]{article}

\usepackage{amsthm}
\usepackage{amscd}
 \usepackage{txfonts}

\usepackage{amsmath}
\usepackage{mathrsfs}
\usepackage{enumerate}

\usepackage{mathtools}
\mathtoolsset{showonlyrefs}
\allowdisplaybreaks[1]
\usepackage[english]{babel}
\usepackage{color}
\usepackage{hyperref}
\usepackage{lastpage}
\usepackage{fancyhdr}
\usepackage{geometry}
\usepackage{graphicx}
\allowdisplaybreaks[1]
\hypersetup{
    colorlinks,
    citecolor=black,
    filecolor=black,
    linkcolor=black,
    urlcolor=black
}
\usepackage{fancyhdr}
	\usepackage{tocloft}

\numberwithin{equation}{section}
\theoremstyle{plain}
	\newtheorem{theorem}{Theorem}[section]           % Theorem environment
	\newtheorem{lemma}[theorem]{Lemma}                 % Lemma environment
	\newtheorem{assumption}{Assumption}[section]  
\theoremstyle{definition}
	\newtheorem{definition}[theorem]{Definition}                 % Definition environment
\theoremstyle{remark}
	\newtheorem{remark}[theorem]{Remark}                % Remark environment
\newcommand{\func}[1]{\operatorname{#1}}

\title{\vspace{-1.2in}On Degenerate  Linear Stochastic Evolution Equations Driven by Jump Processes \vspace{-0.5cm}}
\date{\today}

\begin{document}

\maketitle

% Create table with author name, affiliation, and abstract
\begin{tabular}{l}
{\large\textbf{James-Michael Leahy}} \vspace{0.2cm}\\
 The University of Edinburgh, E-mail: \href{J.Leahy-2@sms.ed.ac.uk}{J.Leahy-2@sms.ed.ac.uk}\vspace{0.4cm} \\
{\large \textbf{Remigijus Mikulevi\v{c}ius}}\vspace{0.2cm}\\
 The University of Southern California, E-mail: \href{Mikulvcs@math.usc.edu }{Mikulvcs@math.usc.edu}\vspace{0.4cm}\\
 {\large\textbf{Abstract}} \vspace{0.2cm} \\
 \begin{minipage}[t]{0.9\columnwidth}%
We prove the existence and uniqueness  of solutions of degenerate linear stochastic evolution equations driven by jump processes in a Hilbert scale using the variational framework of stochastic evolution equations and  the method of vanishing viscosity. As an application of this result, we derive the existence and uniqueness of solutions of degenerate parabolic linear stochastic integro-differential equations (SIDEs) in the  Sobolev scale.  The SIDEs that we consider arise in the theory of non-linear filtering as the equations governing the conditional density of a degenerate jump-diffusion signal given a  jump-diffusion observation, possibly with correlated noise.
 \end{minipage}
\end{tabular}

% Table of contents
\tableofcontents

\thispagestyle{empty}

\section{Introduction}

Let $\left( \Omega ,\mathcal{F},\mathbf{P}\right) $ be a probability space
with the filtration $\mathbf{F}=\left( \mathcal{F}_{t}\right) _{0\leq t\leq
T}$ of sigma-algebras satisfying usual conditions. In a triple of Hilbert
spaces $(H^{\alpha +\mu },H^{\alpha },H^{\alpha -\mu })$ with parameters $%
\mu \in (0,1]$ and $\alpha \geq \mu $, we consider a linear stochastic
evolution equation given by 
\begin{align}
du_{t}& =\left( \mathcal{L}_{t}u_{t}+f_{t}\right) dV_{t}+\left( \mathcal{M}%
_{t}u_{t-}+g_{t}\right) dM_{t},\;\;t\leq T,  \label{eq:SESintro} \\
u_{0}& =\varphi ,  \notag
\end{align}%
where $V_{t}$ is a continuous non-decreasing process, $M_{t}$ is a
cylindrical square integrable martingale, $\mathcal{L}$ and $\mathcal{M}$
are linear adapted operators, and $\phi,f,$ and $g$ are adapted input
functions.

By virtue of Theorems 2.9 and 2.10 in \cite{Gy82}, under some suitable
conditions on the data $\varphi ,f$ and $g,$ if $\mathcal{L}$ satisfies a
growth assumption and $\mathcal{L}$ and $\mathcal{M}$ satisfy a coercivity
condition in the triple $(H^{\alpha +\mu },H^{\alpha },H^{\alpha -\mu })$,
then there exists a unique solution $(u_{t})_{t\leq T}$ of (\ref{eq:SESintro}%
) that is strongly c\`{a}dl\`{a}g in $H^{\alpha }$ and belongs to $%
L^{2}(\Omega \times [ 0,T],\mathcal{O}_{T},dV_{t}d\mathbf{P};H^{\alpha +\mu})
$, where $\mathcal{O}_{T}$ is the optional sigma-algebra on $\Omega \times
\lbrack 0,T]$. In this paper, under a weaker assumption than coercivity (see
Assumption \ref{asm:mainSES}$(\alpha ,\mu )$ below) and using the method of
vanishing viscosity, we prove that there exists a unique solution $%
(u_{t})_{t\leq T}$ of (\ref{eq:SESintro}) that is strongly c\`{a}dl\`{a}g in 
$H^{\alpha^{\prime }}$ for all $\alpha^{\prime }<\alpha$ and belongs to $%
L^{2}(\Omega \times \lbrack 0,T],dV_{t}d\mathbf{P};H^{\alpha })$.
Furthermore, under some additional assumptions on the operators $\mathcal{L}$
and $\mathcal{M}$ we can show that the solution $u$ is weakly c\`{a}dl\`{a}g
in $H^{\alpha }$.

The variational theory of deterministic degenerate linear elliptic and
parabolic PDEs was established by O.A. Oleinik and E.V. Radkevich in \cite%
{Ol65} and \cite{OlRa71}. In \cite{Pa75}, {\'E}. Pardoux developed the
variational theory of monotone stochastic evolution equations, which was
extended in \cite{KrRo77}, \cite{KrRo79}, \cite{GyKr80}, and \cite{Gy82} by
N.V. Krylov, B.L Rozovski\u{\i}, and I. Gy{\"o}ngy. Degenerate parabolic
stochastic partial differential equations (SPDEs) driven by continuous noise
were first investigated by N.V. Krylov and B.L. Rozovski\u{\i} in \cite%
{KrRo82a}. These types of equations arise in the theory of non-linear
filtering of continuous diffusion processes as the Zakai equation and as
equations governing the inverse flow of continuous diffusions. In \cite%
{GeGyKr14}, the solvability of systems of linear SPDEs in Sobolev spaces was
proved by M. Gerencs{\'e}r, I. Gy{\"o}ngy, and N.V. Krylov, and a small gap
in the proof of the main result of \cite{KrRo82a} was fixed. In Chapters 2,
3, and 4 of \cite{Ro90}, B.L. Rozovski\u{\i} offers a unified presentation
and extension of earlier results on the variational framework of linear
stochastic evolution systems and SPDEs driven by continuous martingales
(e.g.\ \cite{Pa75}, \cite{KrRo77}, \cite{KrRo79}, and \cite{KrRo82a}). Our
existence and uniqueness result on degenerate linear stochastic evolution
equations driven by jump processes (Theorem \ref{thm:degeneratespideexist}
below) extends Theorem 2 in Chapter-3-Section 2.2 of \cite{Ro90} to include
the important case of equations driven by jump processes. It is also worth
mentioning that the semigroup approach for non-degenerate SPDEs driven by
L\'evy processes is well-studied (see, e.g. \cite{PeZa07} and \cite{PeZa13}).

As a special case of (\ref{eq:SESintro}), we will consider a system of
stochastic integro-differential equations. Before introducing the equation,
let us describe our driving processes. Let $\mathcal{P}_{T}$ and $\mathcal{R}%
_{T}$, be the predictable and progressive sigma-algebras on $\Omega \times
\lbrack 0,T]$, respectively. Let $\eta (dt,dz)$ be an integer-valued random
measure on $(\mathbf{R}_{+}\times Z,\allowbreak \mathcal{B}(\mathbf{R}%
_{+})\otimes \mathcal{Z})$ with predictable compensator $\pi _{t}(dz)dV_{t}$%
. Let $\tilde{\eta}(dt,dz)=\eta (dt,dz)-\pi _{t}(dz)dt$ be the martingale
measure corresponding to $\eta (dt,dz)$. Let $(Z^{2},\mathcal{Z}^{2})$ be a
measurable space with $\mathcal{R}_{T}$-measurable family $\pi _{t}^{2}(dz)$
of sigma-finite random measures on $Z$. Let $w_{t}=(w_{t}^{\varrho })_{\rho
\in \mathbf{N}},$ $t\geq 0$, be a sequence of continuous local uncorrelated
martingales such that $d\langle w^{\varrho }\rangle _{t}=dV_{t}$, for all $%
\rho \in \mathbf{N}.$ Let $d_{1},d_{2}\in \mathbf{N}$. For convenience, we
set $(Z^{1},\mathcal{Z}^{1})=(Z,\mathcal{Z})$ and $\pi _{t}^{1}=\pi _{t}$.
We consider the $d_{2}$-dimensional system of SIDEs on $[0,T]\times \mathbf{R%
}^{d_{1}}$ given by 
\begin{align}
du_{t}^{l}& =\left( (\mathcal{L}_{t}^{1;l}+\mathcal{L}%
_{t}^{2;l})u_{t}+b_{t}^{i}\partial _{i}u_{t}^{l}+c_{t}^{l\bar{l}}u_{t}^{\bar{%
l}}(x)+f_{t}^{l}\right) dV_{t}+(\mathcal{N}_{t}^{l\varrho
}u_{t}+g_{t}^{l\varrho })dw_{t}^{\varrho }  \label{eq:SIDEintro} \\
& \quad +\int_{Z^{1}}\left( \mathcal{I}_{t,z}^{l}u_{t-}^{\bar{l}%
}+h_{t}^{l}(z)\right) \tilde{\eta}(dt,dz),  \notag \\
u_{0}^{l}& =\varphi ^{l},\;\;l\in \{1,\ldots ,d_{2}\},  \notag
\end{align}%
where for $k\in \{1,2\}$, $l\in \{1,\ldots ,d_{2}\}$, and $\phi \in
C_{c}^{\infty }(\mathbf{R}^{d_{1}};\mathbf{R}^{d_{2}})$, 
\begin{align*}
\mathcal{L}_{t}^{k;l}\phi (x)& :=\frac{1}{2}\sigma _{t}^{k;i\varrho
}(x)\sigma _{t}^{k;j\varrho }(x)\partial _{ij}\phi ^{l}(x)+\sigma
_{t}^{k;i\varrho }(x)\upsilon _{t}^{k;l\bar{l}\varrho }(x)\partial _{i}\phi
^{\bar{l}}(x) \\
& \quad \int_{Z^{k}}\left( \left( \delta _{l\bar{l}}+\rho _{\omega ,t}^{k;l%
\bar{l}}(x,z)\right) \left( \phi ^{\bar{l}}(x+\zeta ^{k}(x,z))-\phi ^{\bar{l}%
}(x)\right) -\zeta _{t}^{k;i}(x,z)\partial _{i}\phi ^{l}(x)\right) \pi
_{t}^{k}(dz) \\
\mathcal{N}_{t}^{l\varrho }\phi (x)& :=\sigma _{t}^{1;i\varrho }(x)\partial
_{i}\phi ^{l}(x)+\upsilon _{t}^{1;l\bar{l}\varrho }(x)\phi ^{\bar{l}%
}(x),\;\varrho \in \mathbf{N}, \\
\mathcal{I}_{t,z}^{l}\phi (x)& :=(\delta _{l\bar{l}}+\rho _{t}^{1;l\bar{l}%
}(x,z))\phi ^{\bar{l}}(x+\zeta _{t}^{1}(x,z))-\phi ^{l}(x),
\end{align*}%
and where $\delta _{l\bar{l}}$ is the Kronecker delta (i.e. $\delta _{l\bar{l%
}}=1$ if $l=\bar{l}$ and $\delta _{l\bar{l}}=0$ otherwise). The summation
convention with respect to repeated indices is used here and below;
summation over $i$ is performed over the set $\{1,\ldots ,d_{1}\}$ and the
summation over $l,\bar{l}$ is performed over the set $\{1,\ldots ,d_{2}\}.$
Without the noise term $\tilde{\eta}(dt,dz)$ and integro-differential
operators in $\mathcal{L}^{1}$ and $\mathcal{L}^{2}$, equation %
\eqref{eq:SIDEintro} has been well-studied (see, e.g. \cite{KrRo82a}, \cite%
{Ro90} (Chapter 3), and the recent paper \cite{GeGyKr14}).

Let $(H^{\alpha }(\mathbf{R}^{d_{1}},\mathbf{R}^{d_{2}}))_{\alpha \in 
\mathbf{R}}$ be the $L^{2}$-Sobolev-scale (i.e. the Bessel potential spaces
with $p=2$). For each $m\in \mathbf{N}$, using our theorem on degenerate
stochastic evolution equations discussed above, under suitable measurability
and regularity conditions on the coefficients, initial condition, and free
terms, we derive the existence of a unique solution $(u_{t})_{t\le T}$ of %
\eqref{eq:SIDEintro} that is weakly c\`{a}dl\`{a}g in $H^{m}(\mathbf{R}%
^{d_{1}},\mathbf{R}^{d_{2}})$, strongly c\`{a}dl\`{a}g in $H^{\alpha}(%
\mathbf{R}^{d_{1}},\mathbf{R}^{d_{2}})$ for all $\alpha<m$, and belongs to $%
L^{2}(\Omega \times [ 0,T],\mathcal{O}_{T},dV_{t}d\mathbf{P};\allowbreak
H^{m}(\mathbf{R}^{d_{1}},\mathbf{R}^{d_{2}}))$.

Degenerate stochastic integro-differential equations of type %
\eqref{eq:SIDEintro} arise in the theory of non-linear filtering of
semimartingales as the Zakai equation and as the equations governing the
inverse flow of jump diffusion processes. We constructed solutions of the
above equation (with $\pi_t(dz)$ deterministic and independent of time)
using the method of stochastic characteristics in \cite{LeMi14b} and \cite%
{LeMi14}. In \cite{DaGy14}, I. Gy{\"o}ngy and K. Dareiotis proved the
existence, uniqueness, and the positivity of solutions of non-linear
stochastic integro-differential equations with non-degenerate stochastic
parabolicity using a comparison principle. It is worth mentioning that the
main estimate used in the proof of uniqueness for Theorem 2.2 in \cite%
{LeMi14} (which is done in a weighted $L^2$-norm) is essentially the same as
the main estimate used in the proof of the degenerate coercivity property of
the operators $\mathcal{L}$, $\mathcal{N}$, and $\mathcal{I}$ in %
\eqref{eq:SIDEintro}.

This chapter is organized as follows. We derive our existence and uniqueness
result for \eqref{eq:SESintro} in Section \ref%
{sec:DegenerateLinearStochasticEv} and for \eqref{eq:SIDEintro} in Section %
\ref{sec:DegenerateLinearSIDESobolev}.

\section{Degenerate linear stochastic evolution equations}

\label{sec:DegenerateLinearStochasticEv}

\subsection{Basic notation and definitions}

\label{subsec:NotandDefSES}

\label{sec:outlineofmainresultsevolution} Let $\mathbf{N}=\{1,2,\ldots,\}$
be the set of natural numbers, $\mathbf{R}$ be the set of real numbers, and $%
\mathbf{R}_{+} $ be the set of non-negative real numbers. All vector spaces
considered in this paper are assumed to have base field $\mathbf{R}$. We
also assume that all Hilbert spaces are separable. For a Hilbert space $H$,
we denote by $H^{\ast }$ the dual of $H$ and by $\mathcal{B}(H)$ the Borel
sigma-algebra of $H$. Unless otherwise stated, the norm and inner product of
a Hilbert space $H$ is denoted by $|\cdot |_{H}$ and $(\cdot ,\cdot )_{H}$,
respectively. For Hilbert spaces $H$ and $U$ and a bounded linear map $%
L:H\rightarrow U$, we denote by $L^{\ast}$ the Hilbert adjoint of $L$.
Whenever we say that a map $F$ from a sigma-finite measure space $(S,%
\mathcal{S},\mu )$ to a Hilbert space $H$ is $\mathcal{S}$-measurable
without specifying the sigma-algebra on $H$, we always mean that $F$ is $%
\mathcal{S}/\mathcal{B}(H)$-measurable. For any Hilbert space $H$ and
sigma-finite measure space $(S,\mathcal{S},\nu )$, we denote by $L^{2}(S,%
\mathcal{S},\mu ;H)$ the linear space of all $\mathcal{S} $-measurable
functions $F:S\rightarrow H$ such that 
\begin{equation*}
|F|_{L^{2}(S,\mathcal{S},\nu ;H)}=\int_{S}|F(s)|_{H}^{2}\nu (ds)<\infty ,
\end{equation*}%
where we identify functions $F,G:S\rightarrow H$ that are equal $\mu $%
-almost-everywhere ($\nu$-a.e.). The linear space $L^{2}(S,\mathcal{S},\nu
;H)$ is a Hilbert space when endowed with the inner product 
\begin{equation*}
(F,G)_{L^{2}(S,\mathcal{S},\nu ;H)}:=\int_{S}(F(s),G(s))_{H}\nu (ds).
\end{equation*}
We use the notation $N=N(\cdot ,\cdots ,\cdot )$ below to denote a positive
constant depending only on the quantities appearing in the parentheses. In a
given context, the same letter is often used to denote different constants
depending on the same parameter. All the stochastic processes considered
below are (at least) $\mathbf{F}$-adapted unless explicitly stated
otherwise. Furthermore, we will often drop the dependence on $\omega\in
\Omega$ for random quantities.

In this section, we consider a scale of Hilbert spaces $(H^{\alpha
})_{\alpha \in \mathbf{R}}$ and a family of operators $(\Lambda ^{\alpha
})_{\alpha \in \mathbf{R}}$ satisfying the following properties:

\begin{itemize}
\item for all $\alpha,\beta\in\mathbf{R}$ with $\beta>\alpha$, $H^{\beta }$
is densely embedded in $H^{\alpha }$;

\item for all $\alpha ,\beta ,\mu \in \mathbf{R}$ with $\alpha <\beta <\mu $
and all $\varepsilon >0$, there is a constant $N=N(\alpha ,\beta ,\mu
,\varepsilon )$ such that 
\begin{equation}
|v|_{\beta }\le \varepsilon |v|_{\mu }+N|v|_{\alpha },\;\;\forall v\in
H^{\mu };  \label{ineq:interpolationinequalityscale}
\end{equation}

\item $\Lambda ^{0}=I;$ for all $\alpha ,\mu \in \mathbf{R}$, $\Lambda
^{\alpha }:H^{\mu }\rightarrow H^{\mu -\alpha }$ is an isomorphism; for all $%
\alpha ,\beta \in \mathbf{R}$, $\Lambda ^{\alpha +\beta }=\Lambda ^{\alpha
}\Lambda ^{\beta }$;

\item for all $\alpha\in \mathbf{R}$, the inner product in $H^{\alpha }$ is
given by $\left( \cdot,\cdot \right) _{\alpha }=\left( \Lambda ^{\alpha
}\cdot, \Lambda ^{\alpha }\cdot\right) _{0}$;

\item for all $\alpha >0$, the dual $\left( H^{\alpha }\right) ^{\ast }$ can
be identified with $H^{-\alpha }$ through the duality product given by 
\begin{equation*}
\langle u,v\rangle _{\alpha }=\langle u,v\rangle _{H^{\alpha },H^{-\alpha
}}=\left( \Lambda ^{\alpha }u,\Lambda ^{-\alpha }v\right) _{0},\;\;u\in
H^{\alpha },\;v\in H^{-\alpha };
\end{equation*}

\item We assume that for every $\alpha \ge 0$, $\Lambda ^{\alpha }$ is
selfadjoint as an unbounded operator in $H^{0}$ with domain $H^{\alpha
}\subseteq H^{0}$: i.e. $\left( \Lambda ^{\alpha }u,v\right) _{0}=\left(
u,\Lambda ^{\alpha }v\right) _{0}$ for all $u,v\in H^{\alpha }.$
\end{itemize}

\begin{remark}
It follows from the above properties that for all $\alpha\in \mathbf{R}$,
the $H^{\alpha }$ norm is given by $\left| v\right| _{\alpha }=\left|
\Lambda ^{\alpha }v\right| _{0}$, $\Lambda ^{\alpha }$ is defined and linear
on $\cup _{\beta \in \mathbf{R}}H^{\beta }$, $\Lambda ^{-\alpha }=\left(
\Lambda ^{\alpha }\right) ^{-1}$, and $\Lambda ^{\alpha }\Lambda ^{\beta
}=\Lambda ^{\beta }\Lambda ^{\alpha }$, for all $\beta \in \mathbf{R.}$
Moreover, for each $\alpha\ge 0$, if $u\in H^{\alpha}$ and $v\in H^0$, then $%
\langle u,v\rangle _{\alpha }=(u,v)_0$.
\end{remark}

\indent We will now describe our driving cylindrical martingale $%
(M_{t})_{t\ge 0}$ in \eqref{eq:SESintro} ) and the associated stochastic
integral. For a more thorough exposition, we refer to \cite{MiRo99}. Let $E$
be a locally convex quasi-complete topological vector space; all bounded
closed subsets of $E$ are complete. Let $E^{\ast }$ be its topological dual.
Denote by $\langle \cdot ,\cdot \rangle_{E^{\ast},E}$ the canonical bilinear
form (duality product) on $E^{\ast}\times E.$ Assume that $E^{\ast }$ is
weakly separable. Denote by $\mathcal{L}^{+}(E)$ the space of symmetric
non-negative definite forms $Q$ from $E^{\ast }$ to $E$; that is, for all $%
Q\in \mathcal{L}^{+}(E),$ we have 
\begin{equation*}
\langle x,Qy\rangle_{E^{\ast},E} =\langle y,Qx \rangle_{E^{\ast},E},\text{
and }\langle x,Qx\rangle_{E^{\ast},E} \ge 0,\quad \forall x,y\in E^{\ast }.
\end{equation*}
Recall that $\mathcal{P}_{T}$ is the predictable sigma-algebra on $\Omega
\times [ 0,T]$. We say that a process $Q:\Omega\times [0,T] \rightarrow 
\mathcal{L}^{+}(E)$ is $\mathcal{P}_{T}$-measurable if $\langle
y,Q_tx\rangle_{E^{\ast},E}$ is $\mathcal{P}_{T}$-measurable for all $x,y\in
E^{\ast}.$

Assume that we are given a family of real-valued locally square integrable
martingales $M=(M_t{y})_{y\in E^{\ast}})_{t\ge 0}$ indexed by $E^{\ast}$ and
an increasing $\mathcal{P}_T$-measurable process $Q:\Omega\times [0,T]
\rightarrow \mathcal{L}^{+}(E)$ such that for all $x,y\in E^{\ast},$%
\begin{equation}
M_t(x)M_t(y)-\int_{0}^{t }\langle x,Q_{s}y\rangle_{E^{\ast},E} dV_{s},
\;\;t\ge 0,
\end{equation}
is a local martingale.

For each $(\omega,t)\in \Omega\times [0,T]$, let $\mathcal{H}_{t}=\mathcal{H}%
_{\omega,t}$ be the Hilbert subspace of $E$ defined as the completion of $%
Q_{\omega,t}E^{\ast}$ with respect to the inner product 
\begin{equation*}
\left( Q_{\omega,t}x,Q_{\omega,t}y\right) _{\mathcal{H}_{\omega,t}}:=\langle
x,Q_{\omega,t}y\rangle ,\;\;x,y\in E^{\ast}.
\end{equation*}
It can be shown that for all $(\omega,t)\in \Omega\times [0,T]$, $E^{\ast}$
is densely embedded into $\mathcal{H}_{t}^{\ast }$, the map $%
Q_{t}:E^{\ast}\rightarrow E$ can be extended to the Riesz isometry $Q_{t}:%
\mathcal{H}_{t}^{\ast }\rightarrow \mathcal{H}_{t}$ (still denoted $Q_{t}$),
and the bilinear form $\langle x,Q_{t}y\rangle_{E^{\ast},E}$, $x,y\in
E^{\ast},$ can be extended to $\langle x,Q_{t}y\rangle _{\mathcal{H}%
_{t}^{\ast },\mathcal{H}_{t}},\;x,y\in \mathcal{H}_{t}^{\ast }$. Note that
for all $x,y\in \mathcal{H}_{t}^{\ast } $, we have $\left( x,y\right) _{%
\mathcal{H}_{t}^{\ast }}=\langle x,Q_{t}y\rangle_{\mathcal{H}_{t}^{\ast },%
\mathcal{H}_{t}}. $

Let $\hat{L}_{loc}^{2}\left( Q\right) $ the space of all processes $f$ such
that $f_t\in \mathcal{H}_{t}^{\ast }$, $dV_td\mathbf{P}$-a.e., $\langle
f_{t},Q_{t}y\rangle _{\mathcal{H}_{t}^{\ast },\mathcal{H}_{t}}$ is $\mathcal{%
P}_{T}$-measurable for all $y\in E^{\ast}$, and $\mathbf{P}$-a.s. 
\begin{equation*}
\int_{0}^{T}| f_{t}| _{\mathcal{H}_{t}^{\ast
}}^{2}dV_{t}=\int_{0}^{T}\langle f_{t},Q_{t}f_{t}\rangle _{\mathcal{H}%
_{t}^{\ast },\mathcal{H}_{t}}dV_{t}<\infty .
\end{equation*}
In \cite{MiRo99}, the stochastic integral of $f\in \hat{L}_{loc}^{2}( Q)$
against $M$, denoted $\mathfrak{I}_{t}(f) =\int_{0}^{t}f_{s}dM_{s}$, $t\ge 0$%
, was constructed and has the following properties: $(\mathfrak{I}%
_t(f))_{t\ge0}$ is a locally square integrable martingale and $\mathbf{P}$%
-a.s.\ for all $t\in [0,T]$:

\begin{itemize}
\item for all $y\in E^{\ast}$, $\mathfrak{I}_t(y)
=\int_{0}^{t}ydM_{s}=M_{t}(y)$ (recall that $E^{\ast}$ is embedded into all $%
\mathcal{H}_{s}^{\ast }$);

\item for all $g\in \hat{L}_{loc}^{2}( Q)$. 
\begin{equation*}
\langle \mathfrak{I}(f) ,\mathfrak{I}(g) \rangle _{t}=\int_{0}^{t}\left(
f_{s},g_{s}\right) _{\mathcal{H}_{s}^{\ast }}dV_{s}=\int_{0}^{t}\langle
f_{s},Q_{s}g_{s}\rangle _{\mathcal{H}_{s}^{\ast },\mathcal{H}_{s}}dV_{s};
\end{equation*}

\item for all bounded $\mathcal{P}_{T}$-measurable processes $\phi:\Omega%
\times[0,T]\rightarrow\mathbf{R}$, 
\begin{equation*}
\int_{0}^{t}\phi _{s}d\mathfrak{I}_s(f) =\mathfrak{I}_t( \phi f)
=\int_{0}^{t}\phi _{s}f_{s}dM_{s}.
\end{equation*}
\end{itemize}

For a Hilbert space $H$ and $(\omega,t)\in \Omega\times [0,T]$, denote by $%
L_{2}(H,\mathcal{H}_{t}^{\ast })$ the space of all Hilbert-Schmidt operators 
$\Psi: H\rightarrow \mathcal{H}_{t}^{\ast }$ with norm and inner product
given by 
\begin{equation*}
| \Psi| _{L_{2}(H,\mathcal{H}_{t}^{\ast })}^{2}:=\sum_{n=1}^{\infty }|\Psi
h^{n}|_{\mathcal{H}_{t}^{\ast }}^{2}, \quad ( \Psi,\tilde{\Psi}) _{L_{2}(H,%
\mathcal{H}_{t}^{\ast })}= \sum_{n=1}^{\infty}( \Psi h^{n},\tilde{\Psi}%
h^{n}) _{\mathcal{H}_{t}^{\ast }},\;\; \tilde{\Psi}\in L_{2}(H,\mathcal{H}%
_{t}^{\ast }),
\end{equation*}
where $\left( h^{n}\right) _{n\in\mathbf{N}}$ is a complete orthogonal
system in $H$. Denote by $L_{loc}^{2}(H,Q)$ the space of all processes $\Psi$
such that $\Psi_{t}\in L_{2}( H,\mathcal{H}_{t}^{\ast }) $, $dV_td\mathbf{P}$%
-a.e., $\Psi_{t}h\in \hat{L}^{2}( Q) $, for each $h\in H,$, and $\mathbf{P}$%
-a.s. 
\begin{equation*}
\int_{0}^{T}| \Psi_{t}| _{L_{2}(H,\mathcal{H}_{t}^{\ast })}^{2}dV_{t}<\infty.
\end{equation*}%
For each $\Psi\in L_{loc}^{2}\left( H,Q\right) $, we define the stochastic
integral $\mathfrak{I}_{t}( \Psi) =\int_{0}^{t}\Psi_{s}dM_{s}$ as the unique 
$H$-valued c\`{a}dl\`{a}g locally square integrable martingale such that $%
\mathbf{P}$-a.s.\ for all $t\in [ 0,T]$ and $h\in H,$ 
\begin{equation*}
\left( \mathfrak{I}_t( \Psi) ,h\right) _{H}=\int_{0}^{t}\Psi_{s}hdM_{s}.
\end{equation*}%
For all $\Psi,\tilde{\Psi}\in L_{loc}^{2}\left(H,Q\right),$ we have that 
\begin{align*}
| \mathfrak{I}_{\cdot}\left( \Psi\right) | _{H}^{2}-\int_{0}^{\cdot}|
\Psi_{s}| _{L_{2}(H,\mathcal{H}_{s}^{\ast })}^{2}dV_{s}\;\; \mathnormal{and}
\;\; ( \mathfrak{I}_{\cdot}( \Psi) ,\mathfrak{I} _{\cdot}( \tilde{\Psi}) )
_{H}-\int_{0}^{\cdot}( \Psi_{s},\tilde{\Psi}_{s}) _{L_{2}(H, \mathcal{H}%
_{s}^{\ast })}dV_{s}
\end{align*}
are real-valued local martingales. Moreover, for all bounded $\mathcal{P}_T $%
-measurable $H$-valued processes $u:\Omega\times [0,T]\rightarrow H,$ $%
\mathbf{P}$-a.s.\ for all $t\in [ 0,T],$%
\begin{equation*}
\int_{0}^{t}u_{s}d\mathfrak{I}_s(\Psi)
=\int_{0}^{t}\{u_{s}\Psi_{s}\}_{H}dM_{s},
\end{equation*}%
where for a complete orthogonal system $\left( \tilde{e}_{s}^{n}\right)
_{n\in\mathbf{N}}$ in $\mathcal{H}_{s}^{\ast },$ 
\begin{equation*}
\{u_{s}\Psi_{s}\}_{H}:=\sum_{n=1}^{\infty}\left( \Psi_{s}u_{s},\tilde{e}%
_{s}^{n}\right) _{\mathcal{H}_{s}^{\ast }}\tilde{e}_{s}^{n}.
\end{equation*}
If $H$ and $Y$ are Hilbert spaces and $L : H\rightarrow Y$ is a bounded
linear operator and $\Psi\in L_{loc}^{2}( H,Q)$, then it follows that $L%
\mathfrak{I}_t(\Psi)= \mathfrak{I}_t(\Psi L^{\ast});$ indeed, for all $y\in
Y $, we have 
\begin{equation*}
(L\mathfrak{I}_t(\Psi),y)_Y=(\mathfrak{I}_t(\Psi),L^*y)_H=\int_0^t\Psi_sL^{%
\ast}ydM_s
\end{equation*}
and $\Psi L^{\ast}\in L_{loc}^{2}( Y,Q)$.

\subsection{Main results}

In this section, for $\mu \in (0,1],$ we consider the linear stochastic
evolution equation in the triple $(H^{-\mu },H^{0},H^{\mu })$ given by 
\begin{align}
du_{t}& =\left( \mathcal{L}_{t}u_{t}+f_{t}\right) dV_{t}+\left( \mathcal{M}%
_{t}u_{t-}+g_{t}\right) dM_{t},\;\;t\leq T,  \label{eq:SES} \\
u_{0}& =\varphi ,  \notag
\end{align}%
where $\varphi $ is an $\mathcal{F}_{0}$-measurable $H^{0}$-valued random
variable and $V_{t}$ is a continuous non-decreasing process such that $%
V_{t}\leq C$ for all $(\omega ,t)\in \Omega \times \lbrack 0,T],$ for some
positive constant $C$. Let $\alpha \geq \mu $ be given. We assume that:

\begin{enumerate}
\item the mapping $\mathcal{L}:\Omega \times [ 0,T]\times H^{\mu
}\rightarrow H^{-\mu }$ is linear in $H^{\mu },$ and for all $v\in H^{\mu }$%
, $\mathcal{L} v$ is $\mathcal{R}_{T}/\mathcal{B}(H^{-\mu })$-measurable; in
addition, $dV_{t}d\mathbf{P}$-a.e., $\mathcal{L}_{t}v\in H^{\alpha -\mu }$
for all $v\in H^{\alpha +\mu }$;

\item for $dV_{t}d\mathbf{P}$-almost-all $(\omega ,t)\in \Omega \times [
0,T] $, $\mathcal{M}_{\omega ,t}:H^{\mu }\rightarrow L_{2}(H^{0},\mathcal{H}%
_{\omega ,t}^{\ast })$ is linear, and for all $v\in H^{\mu }$, $\phi \in
H^{0}$, $y^{\prime }\in E^{\ast }$, $\langle (\mathcal{M}v)\phi
,Q_{t}y^{\prime }\rangle _{\mathcal{H}_{t}^{\ast },\mathcal{H}_{t}}$ is $%
\mathcal{P}_{T}$-measurable for all $y^{\prime }\in E^{\ast }$; in addition, 
$dV_{t}d\mathbf{P}$-a.e, $\mathcal{M}_{t}v\in L_{2}(H^{\alpha },\mathcal{H}%
_{t}^{\ast })$ for all $v\in H^{\alpha +\mu }.$

\item the process $f:\Omega \times [ 0,T]\rightarrow H^{\alpha -\mu }$ is $%
\mathcal{R}_{T}/\mathcal{B}(H^{\alpha -\mu })$-measurable and $g\in
L_{loc}^{2}\left( H^{\alpha },Q\right) \cap L_{loc}^{2}( H^{0},Q) $,
\end{enumerate}

Let us introduce the following assumption for $\lambda \in \{0,\alpha\}$.
Recall that $\left( u,v\right) _{\lambda }=\left( \Lambda ^{\lambda
}u,\Lambda ^{\lambda }v\right) _{0}$.

\begin{assumption}[$\protect\lambda ,\protect\mu $]
\label{asm:mainSES} There are positive constants $L$ and $K$ and an $%
\mathcal{R}_{T}$-measurable function $\bar{f}:\Omega \times [
0,T]\rightarrow \mathbf{R}$ such that the following conditions hold $dV_{t}d%
\mathbf{P}$-a.e.:

\begin{enumerate}
\item for all $v\in H^{\lambda +\mu }$, 
\begin{align*}
2(\Lambda ^{\mu }v,\Lambda ^{-\mu }\mathcal{L}_{t}v)_{\lambda }+|\mathcal{M}%
_{t}v|_{L_{2}(H^{\lambda },\mathcal{H}_{t}^{\ast })}^{2}& \leq L|v|_{\lambda
}^{2}; \\
2(\Lambda ^{\mu }v,\Lambda _{t}^{-\mu }\mathcal{L}_{t}v+f_{t})_{\lambda }+|%
\mathcal{M}_{t}v+g_{t}|_{L_{2}(H^{\lambda },\mathcal{H}_{t}^{\ast })}^{2}&
\leq L|v|_{\lambda }^{2}+\bar{f}_{t};
\end{align*}

\item for all $v\in H^{\lambda +\mu }$, 
\begin{gather*}
|\mathcal{L}_{t}v|_{\lambda -\mu } \le K|v|_{\lambda+\mu},\quad |\mathcal{M}%
_{t}v|_{L_{2}(H^ {\lambda },\mathcal{H}_{t}^{\ast})}\le K|v|_{\lambda+\mu};
\end{gather*}

\item 
\begin{equation*}
|f_{t}|_{\lambda-\mu}^{2}+|g_{t}|_{L_{2}(H^{\lambda },\mathcal{H}^{\ast}
_{t})}^{2} \le \bar{f}_{t},\quad \mathbf{E}\int_{0}^{T}\bar{f}%
_{t}dV_{t}<\infty.
\end{equation*}
\end{enumerate}
\end{assumption}

Let $\mathcal{O}_{T}$ be the optional sigma-algebra on $\Omega \times [ 0,T]$%
. For $\mu \in (0,1])$ and $\lambda \in \mathbf{R}_{+}$ with $\lambda\in \{
0,\alpha \}$, we denote by $\mathcal{W}^{\lambda ,\mu }$ the space of all $%
H^{\lambda }$-valued strongly c\`{a}dl\`{a}g processes $v:\Omega \times [
0,T]\rightarrow H^{\lambda }$ that belong to $L^{2}(\Omega \times [ 0,T],%
\mathcal{O}_{T},dV_{t}d\mathbf{P};H^{\lambda +\mu }).$ The following is our
definition of the solution of \eqref{eq:SES} and is standard in the
variational theory or $L^{2}$-theory of stochastic evolution equations.

\begin{definition}
\label{def:solution}A process $u\in \mathcal{W}^{0,\mu }$ is said to be a
solution of the stochastic evolution equation (\ref{eq:SES}) if $\mathbf{P}$%
-a.s.\ for all $t\in [ 0,T]$%
\begin{equation*}
u_{t}\overset{H^{-\mu}}{=}u_{0}+\int_{0}^{t}(\mathcal{L}%
_{s}u_{s}+f_{s})dV_{s}+\int_{0}^{t}(\mathcal{M}_{s}u_{s-}+g_{s})dM_{s},
\end{equation*}%
where $\overset{H^{-\mu}}{=}$ indicates that the equality holds in the $%
H^{-\mu}$. That is, $\mathbf{P}$-a.s.\ for all $t\in [ 0,T]$ and $v\in
H^{\mu }$, 
\begin{equation*}
(v,u_{t})_{0}=(v,u_{0})+\int_{0}^{t}\langle v,\mathcal{L}_{s}u_{s}+f_{s}%
\rangle _{\mu }dV_{s}+\int_{0}^{t}\{v(\mathcal{M}_{s}u_{s-}+g_{s})%
\}_{H^{0}}dM_{s}.  \label{eq1}
\end{equation*}
\end{definition}

\begin{remark}
In Definition \ref{def:solution}, it is implicitly assumed that the
integrals in (\ref{eq1}) are well-defined. Moreover, it is easy to check
that if Assumption \ref{asm:mainSES}$\left( 0,\mu \right) $ holds, then the
integrals in (\ref{eq1}) are well-defined.
\end{remark}

In order to obtain estimates of the second moments of the supremum in $t$ of
the solution of \eqref{eq:SES}, in the $H^{\alpha}$ norm, we will need to
impose the upcoming assumption. Before introducing this assumption, we
describe a few notational conventions. For two real-valued semimartingales $%
X_{t}$ and $Y_{t}$, we write $\mathbf{P}$-a.s. $dX_{t}\le dY_{t}$ if with
probability 1, $X_{t}-X_{s}\le Y_{t}-Y_{s}$ for any $0\le s\le t\le T$. For $%
v\in \mathcal{W}^{\lambda ,\mu },$ we define 
\begin{equation*}
\mathfrak{M}_{t}\left( v\right) :=\int_{0}^{t}\mathcal{M}_{s}v_{s}dM_{s},\;%
\;t\in [ 0,T],
\end{equation*}
and denote by $\left[ \mathfrak{M}(v)\right] _{\lambda ;t}$ the quadratic
variation process of $\mathfrak{M}_{t}\left( v\right) $ in $H^{\lambda }.$

\begin{assumption}[$\protect\lambda ,\protect\mu$]
\label{asm:specSES} There is a positive constant $L$, a $\mathcal{P}_{T}$%
-measurable function $\bar{g}:\Omega \times [ 0,T]\rightarrow \mathbf{R},$
and an increasing adapted processes $A,B:\Omega\times [0,T]\rightarrow 
\mathbf{R}$ with $dA_{t}d\mathbf{P\le }LdV_{t}d\mathbf{P}$, $dB_{t}d\mathbf{%
P\le }\bar{g}_{t}dV_{t}d\mathbf{P}$ on $\mathcal{P}_{T}$ such that the
following conditions hold $\mathbf{P}$-a.s.:

\begin{enumerate}
\item for all $v\in \mathcal{W}^{\lambda ,\mu }$, 
\begin{equation*}
(\Lambda^{\mu}v_{t},\Lambda^{-\mu}\mathcal{L}_{t}v_{t})_{\lambda}dV_t+d\left[
\mathfrak{M}(v)\right] _{\lambda ;t}+2\{v_{t-}\mathcal{M}_{t}v_{t-}\}_{H^{%
\lambda }}dM_{t} \newline
\le | v_{t-}| _{\lambda }^{2}dA_{t}+G_{t}(v)dM_{t},
\end{equation*}
where $G(v)\in \hat{L}_{loc}^{2}\left( Q\right) $ satisfies $| G_{t}(v)| _{%
\mathcal{H}_{t}^{\ast }}dV_t\le L| v_{t-}| _{\lambda }^{2}dV_t;$

\item for all $v\in \mathcal{W}^{\lambda ,\mu }$, 
\begin{equation*}
2d\left[ \mathfrak{M}\left( v\right) ,\mathcal{I}\left( g\right) \right]
_{\lambda ,t}+2\{v_{t-}g_{t}\}_{H^{\lambda }}dM_{t} \le | v_{t-}| _{\lambda
}dB_{t}+\tilde{G}_{t}\left( v\right) dM_{t},
\end{equation*}
where $\bar{G}(v)\in \hat{L}_{loc}^{2}\left( Q\right) $ satisfies $| \bar{G}%
_{t}(v)| _{\mathcal{H}_{t}^{\ast }}dV_t\le L| v_{t-}| _{\lambda }\bar{g}%
_{t}dV_t$, and 
\begin{equation*}
\mathbf{E}\int_{0}^{T}\bar{g}_{t}^{2}dV_{t}<\infty .
\end{equation*}
\end{enumerate}
\end{assumption}

Although Assumption \ref{asm:specSES}$\left( \lambda ,\mu \right) $ looks
rather technical, it is satisfied for a large class of parabolic stochastic
integro-differential equations (see Section \ref%
{sec:DegenerateLinearSIDESobolev}) under what we consider to be reasonable
assumptions.

Let $\mathcal{T}$ be the set of all stopping times $\tau \le T$ and $%
\mathcal{T}^p$ be the set of all predictable stopping times $\tau\le T$.

\begin{theorem}
\label{th:Degenerate} Let $\mu\in (0,1]$ and $\alpha\ge \mu$. Let Assumption %
\ref{asm:mainSES}$\left( \lambda ,\mu \right) $ hold for $\lambda\in
\{0,\alpha \}$ and assume that $\mathbf{E}\left[|\varphi |_{\alpha }^{2}%
\right]<\infty$.

\begin{enumerate}
\item Then there exists a unique solution $u=(u_t)_{t\le T}$ of %
\eqref{eq:SES} such that for any $\alpha ^{\prime }<\alpha$, $u$ is an $%
H^{\alpha ^{\prime }}$-valued strongly c\`{a}dl\`{a}g process and there is a
constant $N=N(L,K,C)$ such that 
\begin{equation*}
\mathbf{E}\left[ \underset{t\le T}{\sup }\left|u_{t}\right|_{\alpha -\mu
}^{2}\right] +\underset{\tau \in \mathcal{T}}{\sup }\mathbf{E}\left[
|u_{\tau }|_{\alpha }^{2}\right] +\mathbf{E}\int_{0}^{T}\left|u_{s}\right|_{%
\alpha }^{2}dV_{s}\le N \left(\mathbf{E}\left[|\varphi |_{\alpha }^{2}\right]%
+\mathbf{E}\int_0^T\bar{f}_{t}dV_{t} \right).
\end{equation*}
Moreover for each $p\in (0,2)$ and $\alpha ^{\prime }<\alpha ,$ there is a
constant $N=N(L,K,C,p,\alpha^{\prime })$ such that 
\begin{equation*}
\mathbf{E}\left[ \underset{t\le T}{\sup }\left|u_{t}\right|_{\alpha ^{\prime
}}^{p}\right] \le N\mathbf{E}\left[\left( |\varphi |_{\alpha }^{2}+\int_0^T%
\bar{f}_{t}dV_{t}\right) ^{\frac{p}{2}}\right].
\end{equation*}

\item If, in addition, Assumption \ref{asm:specSES}$\left( \lambda ,\mu
\right) $ holds for $\lambda\in \{0,\alpha\} $, then $u$ is an $H^{\alpha }$%
-valued weakly c\`{a}dl\`{a}g process and there is a constant $N=N(L,K,C)$
such that 
\begin{equation*}
\mathbf{E}\left[ \sup_{t\le T}|u_{t}|_{\alpha }^{2}\right] \le N\mathbf{E}%
\left[ |\varphi |_{\alpha }^{2}+\int_0^T(\bar{f}_{t}+\bar{g}_{t}^{2})dV_{t}%
\right].
\end{equation*}
\end{enumerate}
\end{theorem}

\begin{remark}
If $V$ is an arbitrary continuous increasing adapted process, then Theorem %
\ref{th:Degenerate} can be applied locally by considering $%
V_{t}^{C}=V_{t\wedge \tau _{C}},t\in [ 0,T],$ with $\tau _{C}=\inf \left( t%
\in[0,T]:V_{t}\ge C\right) \wedge T.$
\end{remark}

\subsection{Proof of Theorem \protect\ref{th:Degenerate}}

We will construct a sequence of approximations in $\mathcal{W}_{T}^{\alpha
,\mu } $ of the solution of \eqref{eq:SES} by solving in the triple $(
H^{-\mu },H^0,H^{\mu }) $ the equation%
\begin{align}
du_{t}& =\left( \mathcal{L}_{t}^{n}u_{t}+f_{t}\right) dV_{t}+\left( \mathcal{%
M}_{t}u_{t-}+g_{t}\right) dM_{t},\;\;t\le T,  \label{eq:SESvan} \\
u_{0}& =\varphi ,  \notag
\end{align}%
where $\mathcal{L}_{t}^{n}=\mathcal{L}_{t}-\frac{1}{n}(\Lambda ^{\mu })^{2}.$
In order to apply the foundational theorems on stochastic evolution
equations with jumps established in \cite{GyKr81} and \cite{Gy82}, it is
convenient for us first to consider the following equation in the triple $(
H^{-\mu },H^0,H^{\mu }) :$%
\begin{align}
dv_{t}& =\left( \Lambda ^{\alpha }\mathcal{L}_{t}\Lambda ^{-\alpha }v_{t}-%
\frac{1}{n}(\Lambda ^{\mu })^{2}v_{t}+\Lambda ^{\alpha }f_{t}\right)
dV_{t}+\left(\mathcal{M}_{t}\Lambda ^{-\alpha}v_{t-}(\Lambda
^{\alpha})^{\ast} +g_{t}(\Lambda^{\alpha})^{\ast}\right) dM_{t},\;\;t\le T,
\label{eq:SESvanlam} \\
v_{0}& =\Lambda ^{\alpha }\varphi .  \notag
\end{align}
The solutions of \eqref{eq:SESvan} and \eqref{eq:SESvanlam} are to be
understood following Definition \ref{def:solution}.

\begin{lemma}
\label{lem:aprioribound} Let $\mu\in (0,1]$ and $\alpha\ge \mu$. Let
Assumption \ref{asm:mainSES}$\left( \alpha ,\mu \right) $ hold and assume
that $\mathbf{E}\left[|\varphi |_{\alpha }^{2}\right]<\infty$.

\begin{enumerate}
\item For each $n\in \mathbf{N}$, there is a unique solution $%
v^{n}=(v_{t}^{n})_{t\le T}$ of \eqref{eq:SESvan}, and there is a constant $%
N=N(L,K,C)$ independent of $n$ such that 
\begin{equation}  \label{ineq:SupExpEst}
\underset{\tau \in \mathcal{T}}{\sup }\mathbf{E}\left[ |v_{\tau
}^{n}|_{0}^{2}\right] +\mathbf{E}\int_{0}^{T}|v_{t}^{n}|_{0}^{2}dV_{t}+\frac{%
1}{n}\mathbf{E}\int_{0}^{T}|v_{t}^{n}|_{\mu }^{2}dV_{t}\le N\mathbf{E}\left[
|\varphi |_{\alpha }^{2}+\int_{0}^{T}\bar{f}_{t}dV_{t}\right] .
\end{equation}%
Moreover, for each $p\in (0,2)$, there is a constant $N=N(L,K,T,p)$ 
\begin{equation}
\mathbf{E}\left[ \sup_{t\le T}|v_{t}^{n}|_{0}^{p}\right] \le N\mathbf{E}%
\left[ \left( |\varphi |_{\alpha }^{2}+\int_{0}^{T}\bar{f}_{t}dV_{t}\right)
^{\frac{p}{2}}\right] .  \label{ineq:paprest}
\end{equation}

\item If, in addition, Assumption \ref{asm:specSES}$\left( \alpha ,\mu
\right) $ holds, then there is a constant $N=N(L,K,C)$ such that%
\begin{equation}  \label{ineq:esupaprest}
\mathbf{E}\left[ \sup_{t\le T}|v_{t}^{n}|_{0}^{2}\right] \le N\mathbf{E}%
\left[ |\varphi |_{\alpha }^{2}+\int_0^T(\bar{f}_{t}+\bar{g}_{t}^{2})dV_{t}%
\right] .
\end{equation}
\end{enumerate}
\end{lemma}

\begin{proof}
$(i)$ For each $(\omega,t)\in\Omega\times [0,T]$ and $n\in \mathbf{N}$, let 
\begin{equation*}
\mathcal{L}_t^{\alpha }v=\Lambda ^{\alpha }\mathcal{L }_t\Lambda^{-\alpha }v
,\quad \mathcal{L}_{t}^{\alpha ,n}v=\mathcal{L}_t^{\alpha }v-\frac{1}{n}%
\left( \Lambda ^{\mu }\right) ^{2}v,\quad \mathcal{M}_{t}^{\alpha}v=\mathcal{%
M}_{t}\Lambda^{-\alpha}v(\Lambda ^{\alpha })^{\ast}.
\end{equation*}
Using basic properties of the spaces $(H^{\alpha} )_{\alpha\in\mathbf{R}}$
and the operators $(\Lambda^{\alpha})_{\alpha\in\mathbf{R}}$, $dV_{t}d%
\mathbf{P}$-a.e.\ for all $v\in H^{\mu }$, we have 
\begin{equation*}
2\langle v,\mathcal{L}_{t}^{\alpha }v\rangle _{\mu }=2(
\Lambda^{\mu}v,\Lambda^{-\mu}\Lambda^{\alpha}\mathcal{L}_{t}\Lambda^{-%
\alpha}v)_0=2(\Lambda^{\mu}\Lambda^{-\alpha}v,\Lambda^{-\mu}\mathcal{L}%
_{t}\Lambda^{-\alpha}v) _{\alpha },
\end{equation*}
\begin{equation*}
2\langle v,(\Lambda^{\mu})^2v\rangle_{\mu}=2(
\Lambda^{\mu}v,\Lambda^{\mu}v)_0=|v|^2_{\mu},
\end{equation*}
and 
\begin{align*}
\left|\mathcal{M}_{t}^{\alpha }v\right|_{L_{2}(H^{0},\mathcal{H}%
_{t}^{\ast})}^{2}&=\sum_{k=1}^{\infty }|\Lambda^{\alpha}(\mathcal{M}%
_{t}\Lambda^{-\alpha}v)^{\ast}\tilde{e}^{n}_t|_{H^0}^{2}=\sum_{k=1}^{\infty
}|(\mathcal{M}_{t}\Lambda^{-\alpha}v)^{\ast}\tilde{e}^{n}_t|_{H^{\alpha}}^{2}
\\
&=\sum_{k=1}^{\infty }|\mathcal{M}_{t}\Lambda^{-\alpha}v \bar{h}^k|_{%
\mathcal{H}_t^{\ast}}^{2}=\left|\mathcal{M}_{t}^{\alpha
}v\right|_{L_{2}(H^{\alpha},\mathcal{H}_{t}^{\ast})}^{2},
\end{align*}
where $( \tilde{e}_{t}^{k}) _{k\in\mathbf{N}}$,and $(\bar{h}^k)_{k\in 
\mathbf{N}}$ are orthonormal basis of $\mathcal{H}_t$ and $H^{\alpha}$,
respectively. It follows form Assumption \ref{asm:mainSES}$\left( \alpha
,\mu \right) $ that $dV_{t}d\mathbf{P}$-a.e.\ for all $v\in H^{\mu }$, we
have 
\begin{equation*}
|\mathcal{L}_{t}^{\alpha ,n}v|_{-\mu }\le \left( K+\frac{1}{n}\right)
|v|_{\mu }, \quad \left|\mathcal{M}_{t}^{\alpha }v\right|_{L_{2}(H^{0},%
\mathcal{H}_{t}^{\ast})}\le K|v|_{\mu},
\end{equation*}
and 
\begin{equation}
2\langle v,\mathcal{L}_{t}^{\alpha ,n}v+\Lambda ^{\alpha }f_{t}\rangle _{\mu
}+\left|\mathcal{M}_{t}^{\alpha }v+g_{t}(\Lambda ^{\alpha
})^{\ast}\right|_{L_{2}(H^{0},\mathcal{H}_{t}^{\ast})}^{2} \le -\frac{2}{n}%
|v|_{\mu }^{2}+L|v|_{0}^{2}+\bar{f}_{t}.  \label{ineq:newdis}
\end{equation}

In \cite{Gy82}, the variational theory for monotone stochastic evolution
equations driven by locally square integrable Hilbert-space-valued
martingales was derived; it is worth mentioning that the c\`adl\`ag version
of the variational solution in the pivot space and the uniqueness of the
solution was obtained using Theorem 2 in \cite{GyKr81}. The theorems and
proofs given in \cite{Gy82} continue to hold for equations driven by the
cylindrical martingales we consider in this paper. Therefore, by Theorems
2.9 and 2.10 in \cite{Gy82}, for every $n\in \mathbf{N}$, there exists a
unique solution $v^n=(v_t^n)_{t\le T}$ of the stochastic evolution equation
given by 
\begin{equation*}
v_{t}^{n}=\Lambda ^{\alpha }\varphi +\int_0^t(\mathcal{L}_{s}^{\alpha
,n}v_{s}^{n}+\Lambda ^{\alpha }f_{s})dV_{s}+\int_0^t(\mathcal{M}_{s}^{\alpha
}v_{s-}^{n}+\Lambda ^{\alpha }g_{s})dM_{s}, \;\;t\le T.
\end{equation*}
Furthermore, by virtue of Theorem 4.1 in \cite{Gy82}, there is a constant $%
N(n)=N(n,L,\allowbreak K,C)$ such that 
\begin{equation}  \label{ineq:estdepn}
\mathbf{E}\left[ \sup_{t\le T}|v_{t}^{n}|_{0}^{2}\right] +\mathbf{E}%
\int_0^T|v_{t}^{n}|_{\mu }^{2}dV_{t}\le N(n)\mathbf{E}\left[ |\varphi
|_{\alpha }^{2}+\int_0^T\bar{f}_{t}dV_{t}\right].
\end{equation}

We will now use our assumptions to obtain bounds of the solutions $v_t^n$, $%
n\in\mathbf{N}$, in the $H^0$-norm independent of $n$. Applying Theorem 2 in 
\cite{GyKr81}, $\mathbf{P}$-a.s.\ for all $t\in [0,T]$, we have 
\begin{equation}  \label{eq:itosquare}
|v_{t}^{n}|_{0}^{2}=|\varphi |_{\alpha }^{2}+\int_{0}^{t}2\langle v_{s}^{n},%
\mathcal{L}_{t}^{\alpha ,n}v_{s}^{n}+\Lambda ^{\alpha }f_{s}\rangle
_{\mu}dV_{s}+\left[ \tilde{\mathfrak{M}}\right] _{0,t}+m_{t},
\end{equation}%
where $(\tilde{\mathfrak{M}}_t)_{t\le T}$ and $(m_{t})_{t\le T}$ are local
martingales given by 
\begin{equation*}
\tilde{\mathfrak{M}}_t:=\int_{0}^{t}(\mathcal{M}_{s}^{\alpha
}v_{s-}^{n}+\Lambda ^{\alpha }g_{s})dM_{s},\quad
m_{t}:=2\int_{0}^{t}\{v_{s-}^{n}(\mathcal{M}_{s}^{\alpha }v_{s-}^{n}+\Lambda
^{\alpha }g_{s})\}_{H^0}dM_{s}.
\end{equation*}
Thus, taking the expectation of \eqref{eq:itosquare} and making use of
Assumption \ref{asm:mainSES}$\left( \alpha ,\mu \right)$, \eqref{ineq:newdis}%
, and \eqref{ineq:estdepn}, we find that for all $\tau \in \mathcal{T}$, 
\begin{align*}
\mathbf{E}\left[ |v_{\tau }^{n}|_{0}^{2}\right] &\le \mathbf{E} \left[%
|\varphi |_{\alpha }^{2}\right]+ \mathbf{E}\int_{0}^{\tau }\left(2\langle
v_{t}^{n},\mathcal{L}_{t}^{\alpha ,n}v_{t}^{n}+\Lambda ^{\alpha
}f_{t}\rangle _{\mu }+|\mathcal{M}_{t}^{\alpha }v+\Lambda ^{\alpha
}g_{t}|_{L_{2}(H^{0},\mathcal{H}_{t})}^{2}\right)dV_{t} \\
&\le \mathbf{E}\left[ |\varphi |_{\alpha }^{2}+\int_{0}^{\tau }\left(
L|v_{t}^{n}|_{0}^{2}-\frac{2}{n}|v_t^n|_{\mu }^{2}+\bar{f}_{t}\right) dV_{t}%
\right].
\end{align*}%
This implies that for any $\tau \in \mathcal{T}^{p}$, 
\begin{equation*}
\mathbf{E}\left[ |v_{\tau -}^{n}|_{0}^{2}\right] +\frac{2}{n}\mathbf{E}%
\int_{0}^{\tau }|v_t^n|_{\mu }^{2}dV_t \le \mathbf{E}\left[ |\varphi
|_{\alpha }^{2}+ \int_{0}^{\tau }\left( L|v_{s}^{n}|_{0}^{2}+\bar{f}%
_{s}\right)dV_{s}\right].
\end{equation*}%
By virtue of Lemmas 2 and 3 in \cite{GrMi83}, we deduce that there is a
constant $N=N(L,K,C)$ such that for any $\tau \in \mathcal{T},$%
\begin{equation*}
\mathbf{E}\left[ | v_{\tau }^{n}| _{0}^{2}\right]+\frac{1}{n} \mathbf{E}%
\int_{0}^{\tau }|v_t^n|_{\mu }^{2}dV_t \le N \mathbf{E} \left[|\varphi
|_{\alpha }^{2} +\int_{0}^{\tau }\bar{f}_s dV_{s}\right],
\end{equation*}%
which implies that \eqref{ineq:SupExpEst} holds since $V_t$ is uniformly
bounded by the constant $C$. Finally, owing to Corollary II in \cite{Le77},
we have \eqref{ineq:paprest}.

$(ii)$ Using Assumption \ref{asm:specSES}$\left( \alpha ,\mu \right) $ and
estimating \eqref{eq:itosquare}, we get that $\mathbf{P}$-a.s. 
\begin{align*}
d|v_{t}^{n}|_{0}^{2}&\le |\varphi |_{\alpha }^{2} +| v^n_{t-}| _{\lambda
}^{2}dA_{t}+| v^n_{t-}| _{\lambda }dB_{t}+( G_{t}(v^n)+\bar{G}_{t}\left(
v^n\right) )dM_{t}.
\end{align*}
Moreover, for any $\tau \in \mathcal{T}$, we obtain 
\begin{equation*}
\mathbf{E}\left[\sup_{t\le \tau }| v_{t}^{n}|_{0}^{2}\right]\le N\mathbf{E}%
\left[| \varphi | _{\alpha }^{2}+\int_{0}^{\tau
}|v_{t}^{n}|_{0}^{2}dV_{t}+\int_{0}^{\tau }(\bar{f}_{t}+\bar{g}%
_{t}^{2})dV_{t}+\sup_{t\le \tau }| l_{t}^{n}|\right],
\end{equation*}%
where $l^n_{t}:=\int_{0}^{t}(G_{s}(v^n)+\bar{G}_s(v^n))dM_{s}.$ Moreover, by
the Burkholder-Davis-Gundy inequality and Young's inequality, we have 
\begin{align*}
\mathbf{E}\sup_{t\le \tau }| l_{t}^{n}| &\le N\mathbf{E}\left[ \left(
\int_{0}^{\tau }(| v^n_{t-}| _{\alpha }^{4}+| v^n_{t-}| _{\alpha }^{2}\bar{g}%
^2_t)dV_{t}\right) ^{\frac{1}{2}}\right] \le N\mathbf{E}\left[ \sup_{t\le
\tau }| v_{t}^{n}| _{0}\left( \int_{0}^{\tau }(| v^n_{t-}| _{\alpha }^{2}+%
\bar{g}_t^2)dV_{t}\right) ^{\frac{1}{2}}\right] \\
&\le\frac{1}{4N}\mathbf{E}\left[\sup_{t\le \tau }|v_{t}^{n}|_{0}^{2}\right]+N%
\mathbf{E}\int_{0}^{\tau }(| v^n_{t}| _{\alpha }^{2}+\bar{g}_t^2)dV_{t},
\end{align*}%
from which estimate \eqref{ineq:esupaprest} follows.
\end{proof}

\begin{proof}[Proof of Theorem \protect\ref{th:Degenerate} ]
$(i)$ Let $v^{n}=(v_{t}^{n})_{t\le T}$ be the unique solution of %
\eqref{eq:SESvanlam} constructed in Lemma \ref{lem:aprioribound}. Since $%
v^{n}\in \mathcal{W}^{0,\mu }$, it follows that $u^{n}:=\Lambda ^{-\alpha
}v^{n}\in \mathcal{W}^{\alpha ,\mu }\subseteq \mathcal{W}^{0,\mu }$ is a
solution of \eqref{eq:SESvan} in the triple $(H^{-\mu },H^{0},H^{\mu })$.

We will first show that $(u^{n})_{n\in \mathbf{N}}$ is Cauchy in $H^{0}$.
Letting $u^{n,m}=u^{n}-u^{m}$, for each $n,m\in \mathbf{N}$, we have 
\begin{equation*}
u_{t}^{n,m}=\int_{0}^{t}(\mathcal{L}_{t}^{n}u_{t}^{n}-\mathcal{L}%
_{t}^{m}u_{t}^{m})dV_{t}+\int_{0}^{t}\mathcal{M}_{t}u^{n,m}dM_{t},\;\;t\le T.
\end{equation*}%
Applying Theorem 2 in \cite{GyKr81}, we obtain that $\mathbf{P}$-a.s. for
all $t\in [ 0,T]$, 
\begin{equation}
|u_{t}^{n}-u_{t}^{m}|_{0}^{2}=\int_{0}^{t}2\langle u_{s}^{n}-u_{s}^{m},%
\mathcal{L}_{s}^{n}u_{s}^{n}-\mathcal{L}_{s}^{m}u_{s}^{m}\rangle _{\mu
,0}dV_{s}+\left[ \mathfrak{M}^{n,m}\right] _{t}+\eta _{t}^{n,m}
\label{eq:itosquarecauchy}
\end{equation}%
where $(\mathfrak{M}_{t}^{n,m})_{t\le T}$ and $(\eta _{t}^{n,m})_{t\le T}$
are local martingales given by 
\begin{equation*}
\mathfrak{M}_{t}^{n,m}:=\int_{0}^{t}\mathcal{M}%
_{s}(u_{s-}^{n}-u_{s-}^{m})dM_{s},\quad \eta
_{t}^{m,n}:=2\int_{0}^{t}\{(u_{s-}^{n}-u_{s-}^{m})\mathcal{M}%
_{s}(u_{s-}^{n}-u_{s-}^{m})\}_{H^{0}}dM_{s}.
\end{equation*}%
Assumption \ref{asm:mainSES}$\left( 0,\mu \right) (i)$ implies that for any $%
\tau \in \mathcal{T}$, 
\begin{align*}
\mathbf{E}\left[ |u_{\tau }^{n,m}|_{0}^{2}\right] & \le \mathbf{E}%
\int_{0}^{\tau }\left( 2\langle u_{s}^{n,m},\mathcal{L}_{s}^{n}u_{s}^{n,m}%
\rangle _{\mu }+|\mathcal{M}_{t}u_{s}^{n,m}|_{L_{2}(H^{0},\mathcal{H}%
_{t}^{\ast })}^{2}\right) dV_{s} \\
& \le \mathbf{E}\int_{0}^{\tau }\left( L|u_{s}^{n,m}|_{0}^{2}+\frac{1}{n}%
|u_{s}^{n}|_{\mu }^{2}+\frac{1}{m}|u_{s}^{m}|_{\mu }^{2}\right) dV_{s},
\end{align*}%
and hence for any $\tau \in \mathcal{T}^{p}$, we have 
\begin{equation*}
\mathbf{E}\left[ |u_{\tau -}^{n.m}|_{0}^{2}\right] \le \mathbf{E}%
\int_{0}^{\tau }\left( L|u_{s}^{n.m}|_{0}^{2}+\frac{1}{n}|u_{s}^{n}|_{\mu
}^{2}+\frac{1}{m}|u_{s}^{m}|_{\mu }^{2}\right) dV_{s}.
\end{equation*}%
By virtue of Lemmas 2 and 4 in \cite{GrMi83} and \eqref{ineq:SupExpEst}
(noting that $|u_{t}^{n}|_{0}=|\Lambda ^{-\alpha }v_{t}^{n}|_{0}\le
N|v_{t}^{n}|_{0}$), there is a constant $N$ such that for any $\tau \in 
\mathcal{T},$%
\begin{equation}
\mathbf{E}\left[ |u_{\tau }^{n,m}|_{0}^{2}\right] \le N\mathbf{E}%
\int_{0}^{\tau }\left( \frac{1}{n}|u_{s}^{n}|_{\mu }^{2}+\frac{1}{m}%
|u_{s}^{m}|_{\mu }^{2}\right) dV_{s}\le N\left( \frac{1}{n}+\frac{1}{m}%
\right) \mathbf{E}\left[ |\varphi |_{\alpha }^{2}+\int_{0}^{T}\bar{f}%
_{t}dV_{t}\right] .  \label{ineq:squarecauchybound}
\end{equation}%
Using Corollary II in \cite{Le77}, we have that for each $p\in (0,2)$, there
is a constant $N$ such that 
\begin{equation*}
\mathbf{E}\left[ \sup_{t\le T}|u_{t}^{n,m}|_{0}^{p}\right] \le N\left( \frac{%
1}{n}+\frac{1}{m}\right) ^{\frac{p}{2}}\left[ \mathbf{E}\left[ |\varphi
|_{\alpha }^{2}+\int_{0}^{T}\bar{f}_{t}dV_{t}\right] \right] ^{\frac{p}{2}}.
\end{equation*}%
Therefore, 
\begin{equation}
\lim_{n,m\rightarrow \infty }\left[ \sup_{\tau \in \mathcal{T}}\mathbf{E}%
\left[ |u_{\tau }^{n,m}|_{0}^{2}\right] +\mathbf{E}\left[ \sup_{t\le
T}|u_{t}^{n,m}|_{0}^{p}\right] \right] =0,  \label{eq:limitnotnorm}
\end{equation}%
and hence there exists a strongly c\`{a}dl\`{a}g $H^{0}$-valued process $%
u=(u_{t})_{t\le T}$ such that 
\begin{equation}
\underset{n\rightarrow \infty }{d\mathbf{P}-\lim }\sup_{t\le
T}|u_{t}-u_{t}^{n}|_{0}=0.  \label{eq:convergeprobzero}
\end{equation}

Since for each $n$, $u^n$ is solution of \eqref{eq:SESvan}, we have that $%
\mathbf{P}$-a.s. for all $t\in [0,T]$ and $\phi \in H^{\mu}$, 
\begin{equation}  \label{eq:approxeq}
(\phi, u^n_{t})_0=(\phi, \varphi)_0 + \int_0^t\langle \phi,\mathcal{L}%
_{s}^{n}u^n_{s}+f_{s}\rangle_{\mu,0}dV_{s}+\int_0^t\{\phi(\mathcal{M}%
_{s}u^n_{s-}+g_{s})\}_{H^0}dM_{s}.
\end{equation}
Owing to \eqref{eq:convergeprobzero}, we know that the left-hand-side of %
\eqref{eq:approxeq} converges $\mathbf{P}$-a.s.\ for all $t\in [0,T]$ to $%
(\phi, u_{t})_0$ as $n$ tends to infinity. Our aim, of course, is to pass to
the limit as $n$ tends to infinity on the right-hand-side.

This can be done quite simply when $\alpha>\mu$. Indeed, owing to the
interpolation inequality \eqref{ineq:interpolationinequalityscale}, for each 
$\varepsilon >0$, $\alpha ^{\prime }<\alpha ,$ and $p\in (0,2)$, there is a
constant $N=N(\alpha,\alpha^{\prime },\varepsilon,p)$ such that%
\begin{equation}  \label{ineq:interpolationpalphaprime}
\sup_{t\le T}| u_{t}^{n,m}| _{\alpha ^{\prime }}^p\le \varepsilon \sup_{t\le
T}| u_{t}^{n,m}| _{\alpha }^p+N\sup_{t\le T}| u_{t}^{n,m}| _{0}^p.
\end{equation}
Since $| u_{t}^{n}| _{\alpha}=| \Lambda ^{-\alpha }v_{t}^{n}| _{\alpha}\le
N| v_{t}^{n}| _{0},$ by \eqref{ineq:esupaprest} and %
\eqref{ineq:interpolationpalphaprime}, we have that for all $\alpha ^{\prime
}<\alpha $ and $p\in (0,2),$ 
\begin{equation}  \label{ineq:interppasstolimit}
\mathbf{E}\sup_{t\le T}| u_{t}^{n,m}| _{\alpha ^{\prime }}^p\le \varepsilon N%
\mathbf{E}\left[ \left( |\varphi |_{\alpha }^{2}+\int_0^T\bar{f}%
_{t}dV_{t}\right) ^{\frac{p}{2}}\right] +N\mathbf{E}\sup_{t\le T}|
u_{t}^{n,m}| _{0}^p.
\end{equation}
Using \eqref{eq:limitnotnorm} and passing to the limit as $n$ and $m$ tends
to infinity on both sides of \eqref{ineq:interppasstolimit}, and then taking 
$\varepsilon\downarrow 0$, we get that for all $\alpha ^{\prime }<\alpha $
and $p\in (0,2),$ 
\begin{equation*}
\lim_{n,m\rightarrow\infty} \mathbf{E}\left[ \sup_{t\le
T}|u_{t}^{n,m}|_{\alpha ^{\prime }}^{p}\right]=0.
\end{equation*}
Combining the above results, we conclude that for any $\alpha^{\prime
}<\alpha$, $u$ is an $H^{\alpha ^{\prime }}$-valued strongly c\`{a}dl\`{a}g
process and 
\begin{equation}  \label{eq:convergeprobalphlmu}
\underset{n\rightarrow\infty}{d\mathbf{P}-\lim}\sup_{t\le
T}|u_t-u^n_t|_{\alpha^{\prime }}=0.
\end{equation}
In particular, if $\alpha>\mu$, then taking $\alpha^{\prime }>\mu$ in %
\eqref{eq:convergeprobalphlmu} and appealing to Assumption \ref{asm:mainSES}$%
\left( 0 ,\mu \right)(ii)$, \eqref{ineq:SupExpEst}, and the identity, 
\begin{equation*}
\langle \phi,\Lambda ^{2\mu}u_{s}^{n}\rangle _{\mu}= (
\Lambda^{\mu}\phi,\Lambda ^{\mu}u_{s}^{n}) _{0},
\end{equation*}
we can take the limit as $n$ tends to infinity on the right-hand-side of %
\eqref{eq:approxeq} by the dominated convergence theorem to conclude that $u$
is a solution of \eqref{eq:SES}.

The case $\alpha=\mu$ must be handled with weak convergence. Let 
\begin{equation*}
S(\mathcal{O}_T)=(\Omega \times [ 0,T],\mathcal{O}_{T} ,d\bar{V}_t d\mathbf{P%
})\quad \mathnormal{and}\quad S(\mathcal{P}_T)=(\Omega \times [ 0,T],%
\mathcal{P}_{T} ,d\bar{V}_t d\mathbf{P}),
\end{equation*}
where $\bar{V}_t=:V_t+t$. It follows from \eqref{ineq:SupExpEst} that there
exists a subsequence $(u^{n_{k}})_{k\in \mathbf{N}}$ of $(u^n)_{n\in \mathbf{%
N}}$ that converges weakly in $L^{2}(S(\mathcal{O}_T);H^{\mu })$ to some $%
\bar{u}\in L^{2}(S(\mathcal{O}_T);H^{\mu })$ which satisfies 
\begin{equation*}
\mathbf{E}\int_{]0,T]}|\bar{u}_{t}|_{\mu}^{2}d\bar{V}_{t}\le N\mathbf{E}%
\left[|\varphi |_{\mu }^{2}+\int_0^T\bar{f}_{t}dV_{t}\right].
\end{equation*}
For any $\phi\in H^{0}$ and bounded predictable process $\xi_t$, we have 
\begin{equation*}
\lim_{k\rightarrow\infty }\mathbf{E}\int_0^T \xi_t \langle\phi ,
u^{n_k}_t\rangle_{\mu} d\bar{V}_t =\lim_{k\rightarrow\infty} \mathbf{E}%
\int_0^T \xi_t ( u^{n_k}_t,\phi)_0 d\bar{V}_t = \mathbf{E}\int_0^T \xi_t (
u_t,\phi)_0 d\bar{V}_t ,
\end{equation*}
and hence $u = \bar{u}$ in $L^{2}(S(\mathcal{O}_T);H^{\mu })$ and $u^{n_k}$
converges to $u$ weakly in $L^{2}(S(\mathcal{O}_T);H^{\mu })$ as $k$ tends
to infinity. Define $u^{n_k}_{-}=(u^{n_k}_{t-})_{t\le T}$ and $%
u_{-}=(u_{t-})_{t\le T}$, where the limits are taken in $H^{0}$. By
repeating the argument above, we conclude that $u^{n_k}_{-}$ converges to $%
u_{-}$ weakly in $L^{2}(S(\mathcal{O}_T);H^{\mu })$ as $k$ tends to infinity.

Fix $\phi\in H^{\mu}$ and a $\mathcal{P}_{T}$-measurable process $(\xi
_{t})_{t\le T}$ bounded by the constant $K$. Define the linear functionals $%
\Phi ^{\mathcal{L}}:L^{2}(S(\mathcal{O}_T);H^{\mu})\rightarrow \mathbf{R}$
and $\Phi ^{\mathcal{M}}:L^{2}(S(\mathcal{P}_T);H^{\mu})\rightarrow \mathbf{R%
}$ by 
\begin{equation*}
\Phi ^{\mathcal{L}}(v)=\mathbf{E}\int_{]0,T]}\xi _{t}\int_{]0,t]}\langle
\phi,\mathcal{L}_{s}v_{s}\rangle_{\mu,0}dV_{s}d\bar{V}_{t},\quad \forall
v\in L^{2}(S(\mathcal{O}_T);H^{\mu})
\end{equation*}
and 
\begin{equation*}
\Phi ^{\mathcal{M}}(v)=\mathbf{E}\int_{]0,T]}\xi _{t}\int_{]0,t]}\{\phi%
\mathcal{M}_{s}v_{s}\}_{H^0}dM_{s}d\bar{V}_{t},\quad \forall v\in L^{2}(S(%
\mathcal{P}_T);H^{\mu}),
\end{equation*}%
respectively. Owing to Assumption \ref{asm:mainSES}$\left( 0 ,\mu \right)
(ii)$, the Burkholder-Davis-Gundy inequality, and the fact that $(\bar{V}%
_t)_{t\le T}$ is uniformly bounded by the constant $C$, there is a constant $%
N=N(K,C)$ such that 
\begin{equation*}
|\Phi ^{\mathcal{L}}(v)|\le N|\phi|_{\mu}\left(\mathbf{E}%
\int_{0}^{T}|v_{t}|_{\mu}^{2}d\bar V_{t}\right)^{\frac{1}{2}}, \;\;\forall
v\in L^{2}(S(\mathcal{O}_T);H^{\mu}),
\end{equation*}
and 
\begin{gather*}
|\Phi ^{\mathcal{M}}(v)| \le N|\phi|_{\mu}\left(\mathbf{E}%
\int_{]0,T]}|v_{s}|_{\mu }^2d\bar V_{s}\right)^{\frac{1}{2}}, \;\;\forall
v\in L^{2}(S(\mathcal{P}_T);H^{\mu}).
\end{gather*}
This implies that $\Phi ^{\mathcal{L}}$ is a continuous linear functional on 
$L^{2}(S(\mathcal{O}_T);H^{\mu})$ and $\Phi ^{\mathcal{M}}$ is a continuous
linear functional on $L^{2}(S(\mathcal{P}_T);H^{\mu})$, and hence that 
\begin{equation}
\lim_{k\rightarrow \infty }\Phi ^{\mathcal{L}}(u^{n_{k}})=\Phi ^{\mathcal{L}%
}(u), \quad \lim_{k\rightarrow \infty }\Phi ^{\mathcal{L}}(u^{n_{k}}_{-})=%
\Phi ^{\mathcal{M}}(u_{-}) .  \label{eq:convLM}
\end{equation}
For each $k$, we have that 
\begin{align}  \label{eq:expxieqn}
\mathbf{E}\int_0^T\xi_t(\phi, u^{n_k}_{t})_0d\bar{V}_t&=\mathbf{E}%
\int_0^T\xi_t(\phi, \varphi)_0d\bar{V}_t + \mathbf{E}\int_0^T\xi_t\int_0^t%
\langle \phi,\mathcal{L}_{s}^{n_k}u^{n_k}_{s}+f_{s}\rangle_{\mu,0}dV_{s}d%
\bar{V}_t \\
&\quad +\mathbf{E}\int_0^T\xi_t\int_0^t\{\phi(\mathcal{M}_{s}u^n_{s-}+g_{s})%
\}_{H^0}dM_{s}d\bar{V}_t.  \notag
\end{align}
Passing to the limit as $k$ tends to infinity on both sides of %
\eqref{eq:expxieqn} using \eqref{eq:convLM} and 
\begin{equation*}
\langle \phi,\Lambda ^{2\mu}u_{s}^{n_{k}}\rangle _{\mu}= (
\Lambda^{\mu}\phi,\Lambda ^{\mu}u_{s}^{n_{k}}) _{0},
\end{equation*}
we obtain that $d\bar{V}_t d\mathbf{P}$-a.e.\, 
\begin{equation}  \label{eq:solutionlamm1}
( \phi,u_{t}) _{0} =(\phi,\varphi )_{0}+\int_{]0,t]}\langle \phi,(\mathcal{L}%
_{s}u_{s}+f_{s})\rangle _{\mu}dV_{s}+\int_{]0,t]}\{\phi( \mathcal{M}%
_{s}u_{s-}+g_{s})\}_{H_0}dM_{s},\;\;t\le T.
\end{equation}
Therefore, for all $\alpha\ge \mu$, $u$ is a solution of \eqref{eq:SES}.

We will now show that 
\begin{equation}  \label{ineq:supexpalphaprime}
\underset{\tau \in \mathcal{T}}{\sup }\mathbf{E}\left[ |u_{\tau }|_{\alpha
}^{2}\right] \le N\mathbf{E}\left[ |\varphi |_{\alpha }^{2}+\int_{0}^{T}\bar{%
f}_{t}dV_{t}\right].
\end{equation}
Let $(h^{k})_{k\in \mathbf{N}}$ be a complete orthonormal basis in $%
H^{\alpha }$ such that the linear subspace spanned by $(h^{k})_{k\in \mathbf{%
N}}$ is dense in $H^{2\alpha }$. Owing to \eqref{ineq:SupExpEst}, for each $%
m\ge 1$ and $\tau \in \mathcal{T}$, 
\begin{equation*}
\mathbf{E}\left[ \sum_{k=1}^{m}|(u_{\tau }^{n},\Lambda ^{2\alpha
}h^{k})_{0}|^{2}\right]=\mathbf{E}\left[ \sum_{k=1}^{m}|(u_{\tau
}^{n},h^{k})_{\alpha }|^{2}\right] \le \mathbf{E}\left[ | u_{\tau }^{n}|
_{\alpha }^{2}\right] \le N\mathbf{E}\left[ |\varphi |_{\alpha
}^{2}+\int_{0}^{T}\bar{f}_{t}dV_{t}\right] .
\end{equation*}
Applying Fatou's lemma first in $n$ and then in $m$, we have that for each $%
\tau \in \mathcal{T} $,%
\begin{equation*}
\mathbf{E}\left[ \sum_{k=1}^{\infty }|(u_{\tau },\Lambda ^{2\alpha
}h^{k})_{0}|^{2}\right] \le N\mathbf{E}\left[ |\varphi |_{\alpha
}^{2}+\int_{0}^{T}\bar{f}_{t}dV_{t}\right] .
\end{equation*}
Hence, for each $t\in [0,T]$, $\mathbf{P}$-a.s.\ $v_t=\sum_{k}(u_{t
},\Lambda ^{2\alpha }h^{k})_{0}h^{k}\in $ $H^{\alpha }$. Since the linear
subspace spanned by $( \Lambda ^{2\alpha }h^{k})_{k\in\mathbf{N}} $ is dense
in $H^{0}$ and for each $t\in [0,T]$, $\mathbf{P}$-a.s., $( u_{t
}-v_t,\Lambda ^{2\alpha }h^{k})_{0}=0$, for all $k\in \mathbf{N}$, it
follows that $\mathbf{P}$-a.s. for each $\tau \in \mathcal{T} $, $u_{\tau
}=v $ and 
\begin{equation*}
\mathbf{E}\left[ |u_{\tau }|_{\alpha }^{2}\right] =\mathbf{E}%
\sum_{k=1}^{\infty }\left|(u_{\tau },\Lambda ^{2\alpha
}h^{k})_{0}\right|^{2}\le N\mathbf{E}\left[ |\varphi |_{\alpha
}^{2}+\int_{0}^{T}\bar{f}_{t}dV_{t}\right] ,
\end{equation*}
which proves \eqref{ineq:supexpalphaprime}.

Estimating \eqref{eq:SES} directly in the $H^{\alpha-\mu}$-norm, we easily
derive that 
\begin{equation*}
\mathbf{E}\left[ \sup_{t\le T}| u_{t}| _{\alpha-\mu }^{2}\right] \le N 
\mathbf{E}\left[ |\varphi |_{\alpha }^{2}+\int_{0}^{T}\bar{f}_sdV_{t}\right].
\end{equation*}%
If $(v_{t})_{t\le T}$ be another solution of \eqref{eq:SES}, then by Theorem
2 in \cite{GyKr81} and Assumption \ref{asm:mainSES}$(0,\mu)(i)$, $\mathbf{P}$%
-a.s. for all $t\in [ 0,T]$, we have 
\begin{equation*}
|u_{t}-v_t|_{0}^{2}\le L\int_0^t |u_s-v_s|_{0}^2dV_s+ m_{t},
\end{equation*}%
where $(m_{t})_{t\le T}$ is a local martingale with $m_{0}=0$, and hence
applying Lemmas 2 and 4 in \cite{GrMi83}, we get 
\begin{equation*}
\mathbf{P}\left( \sup_{t\le T}|u_{t}-v_t|_{0}>0\right) =0,
\end{equation*}
which implies that $(u_t)_{t\le T}$ is the unique solution of \eqref{eq:SES}%
. This completes the proof of part $(i)$.

$(ii)$ Estimating \eqref{eq:itosquarecauchy} using Assumption \ref%
{asm:specSES}$\left( 0 ,\mu \right) $, we get that $\mathbf{P}$-a.s. 
\begin{equation*}
d|u^{n,m}_{t}|_{0}^{2}\le |\varphi |_{0 }^{2} +\left(\frac{1}{n}+\frac{1}{m}%
\right)|u^{n,m}|_0^2dV_t+| u^{n,m}_{t-}| _{\lambda }^{2}dA_{t}+|
u^{n,m}_{t-}| _{\lambda }dB_{t}+( G_{t}(u^{n,m})+\bar{G}_{t}\left(
u^{n,m}\right) )dM_{t}.
\end{equation*}
Then estimating the stochastic integrand as in the proof of part $(ii)$ of
Lemma \ref{lem:aprioribound}, for any $\tau \in \mathcal{T},$ we get 
\begin{equation*}
\mathbf{E}\left[ \sup_{t\le \tau }|u_{t}^{n.m}|_{0}^{2}\right] \le N\mathbf{E%
}\int_{0}^{\tau }\left(|u_{s}^{n.m}|_{0}^{2}+\frac{1}{n}| u_{s}^{n}| _{\mu
}^{2}+\frac{1}{m}| u_{s}^{m}| _{\mu }^{2})\right) dV_{s},
\end{equation*}
and hence by Gronwall's lemma and Lemma \ref{lem:aprioribound}(ii), 
\begin{equation*}
\mathbf{E}\left[ \sup_{t\le \tau }|u_{t}^{n.m}|_{0}^{2}\right] \le N\left(%
\frac{1}{n}+\frac{1}{m}\right)\mathbf{E}\left[ |\varphi |_{\alpha
}^{2}+\int_{]0,T]}\left(\bar{f}_{t}+\bar{g}_{t}^{2}\right)dV_{t}\right]
\end{equation*}
Thus, 
\begin{equation}  \label{eq:limEsup2}
\lim_{n,m\rightarrow\infty}\mathbf{E}\left[ \sup_{t\le \tau
}|u_{t}^{n.m}|_{0}^{2}\right] =0.
\end{equation}
Let $(h^{k})_{k\in \mathbf{N}} $ be a complete orthonormal basis $H^{\alpha
} $ such that the linear subspace spanned by $(h^k)_{k\in \mathbf{N}}$ is
dense in $H^{2\alpha }$. Owing to \eqref{ineq:esupaprest}, for each $m\ge 1$
and $\tau \in \mathcal{T}$, 
\begin{align*}
\mathbf{E}\left[ \sup_{t\le T}\sum_{k=1}^{m}(u_{t}^{n},h^{k})_{\alpha }%
\right] &=\mathbf{E}\left[ \sup_{t\le T}\sum_{k=1}^{m}|(u_{t}^{n},\Lambda
^{2\alpha }h^{k})_{0}|^{2}\right] \le \mathbf{E}\left[ \sup_{t\le T}|
u_{t}^{n}| _{\alpha }^{2}\right] \\
&\le N\mathbf{E}\left[ |\varphi |_{\alpha }^{2}+\int_{]0,T]}( \bar{f}_{t}+%
\bar{g}_{t}^{2})dV_{t}\right] .
\end{align*}%
Applying Fatou's lemma first in $n$ and then in $m$, we have that 
\begin{equation*}
\mathbf{E}\left[ \sup_{t\le T}\sum_{k=1}^{\infty }|(u_{t},\Lambda ^{2\alpha
}h^{k})_{0}|^{2}\right] \le N\mathbf{E}\left[ |\varphi |_{\alpha
}^{2}+\int_{]0,T]}(\bar{f}_{t}+\bar{g}_{t}^{2})dV_{t}\right] .
\end{equation*}
Thus, $v=\sum_{k}(u,\Lambda ^{2\alpha }h^{k})_{0}h^{k}$ is an $H^{\alpha }$%
-valued weakly c\`{a}dl\`{a}g process. Since the linear subspace spanned on $%
(\Lambda ^{2\alpha }h^{k})_{k\in \mathbf{N}} $ is dense in $H^{0}$ and $%
\left( u_{t}-v_{t},\Lambda ^{2\alpha }h^k\right) _{0}=0$, for all $k\in 
\mathbf{N}$, it follows that $\mathbf{P}$-a.s. for all $t\in [0,T]$, $%
u_{t}=v_{t}$ and 
\begin{equation*}
\mathbf{E}\left[ \sup_{t\le T}|u_{t}|_{\alpha }^{2}\right] =\mathbf{E}\left[
\sup_{t\le T}\sum_{k=1}^{\infty }|(u_{t},\Lambda ^{2\alpha }h^{k})_{0}|^{2}%
\right] \le N\mathbf{E}\left[ |\varphi |_{\alpha }^{2}+\int_{]0,T]}(\bar{f}%
_{t}+\bar{g}_{t}^{2})dV_{t}\right] .
\end{equation*}
\end{proof}

\section{On the $L^2$-Sobolev theory for degenerate parabolic SIDEs}

\label{sec:DegenerateLinearSIDESobolev}

\subsection{Statement of main results}

In this section, we consider the $d_{2}$-dimensional system of SIDEs on $%
[0,T]\times \mathbf{R}^{d_{1}}$ given by 
\begin{align}
du_{t}^{l}& =\left( (\mathcal{L}_{t}^{1;l}+\mathcal{L}%
_{t}^{2;l})u_{t}+b_{t}^{i}\partial _{i}u_{t}^{l}+c_{t}^{l\bar{l}}u_{t}^{\bar{%
l}}+f_{t}^{l}\right) dV_{t}+\left( \mathcal{N}_{t}^{l\varrho
}u_{t}+g_{t}^{l\varrho }\right) dw_{t}^{\varrho }  \label{eq:SIDE} \\
& \quad +\int_{Z^{1}}\left( \mathcal{I}_{t,z}^{l}u_{t-}^{\bar{l}%
}+h_{t}^{l}(z)\right) \tilde{\eta}(dt,dz),  \notag \\
u_{0}^{l}& =\varphi ^{l},\;\;l\in \{1,\ldots ,d_{2}\},  \notag
\end{align}%
where for $k\in \{1,2\}$, $l\in \{1,\ldots ,d_{2}\}$, and $\phi \in
C_{c}^{\infty }(\mathbf{R}^{d_{1}};\mathbf{R}^{d_{2}})$, 
\begin{align*}
\mathcal{L}_{t}^{k;l}\phi (x)& :=\frac{1}{2}\sigma _{t}^{k;i\varrho
}(x)\sigma _{t}^{k;j\varrho }(x)\partial _{ij}\phi ^{l}(x)+\sigma
_{t}^{k;i\varrho }(x)\upsilon _{t}^{k;l\bar{l}\varrho }(x)\partial _{i}\phi
^{\bar{l}}(x) \\
& \quad +\int_{Z^{k}}\left( \left( \delta _{l\bar{l}}+\rho _{t}^{k;l\bar{l}%
}(x,z)\right) \left( \phi ^{\bar{l}}(x+\zeta _{t}^{k}(x,z))-\phi ^{\bar{l}%
}(x)\right) -\zeta _{t}^{k;i}(x,z)\partial _{i}\phi ^{l}(x)\right) \pi
_{t}^{k}(dz) \\
\mathcal{N}_{t}^{l\varrho }\phi (x)& :=\sigma _{t}^{1;i\varrho }(x)\partial
_{i}\phi ^{l}(x)+\upsilon _{t}^{1;l\bar{l}\varrho }(x)\phi ^{\bar{l}%
}(x),\;\varrho \in \mathbf{N}, \\
\mathcal{I}_{t,z}^{l}\phi (x)& :=\left( \delta _{l\bar{l}}+\rho _{t}^{1;l%
\bar{l}}(x,z)\right) \phi ^{\bar{l}}(x+\zeta _{t}^{1}(x,z))-\phi ^{l}(x).
\end{align*}%
We assume that 
\begin{equation*}
\sigma _{t}^{k}(x)=(\sigma _{\omega ,t}^{k;i\varrho }(x))_{1\leq i\leq
d_{1},\;\varrho \in \mathbf{N}},\quad b_{t}(x)=(b_{\omega ,t}^{i}(x))_{1\leq
i\leq d_{1}},\quad c_{\omega ,t}(x)=(c_{t}^{l\bar{l}}(x))_{1\leq l,\bar{l}%
\leq d_{2}},
\end{equation*}%
\begin{equation*}
\upsilon _{t}^{k}(x)=(\upsilon _{\omega ,t}^{k;l\bar{l}\varrho }(x))_{1\leq
l,\bar{l}\leq d_{2},\;\varrho \in \mathbf{N}},\quad f_{t}(x)=(f_{\omega
,t}^{i}(x))_{1\leq i\leq d_{2}},\quad g_{t}(x)=(g_{\omega ,t}^{i\varrho
}(x))_{1\leq i\leq d_{2},\;\varrho \in \mathbf{N}}
\end{equation*}%
are random fields on $\Omega \times \lbrack 0,T]\times \mathbf{R}^{d_{1}}$
that are $\mathcal{R}_{T}\otimes \mathcal{B}(\mathbf{R}^{d_{1}})$%
-measurable. Moreover, we assume that 
\begin{equation*}
\zeta _{t}^{1}(x,z)=(\zeta _{\omega ,t}^{1;i}(x,z))_{1\leq i\leq
d_{1}},\;\rho _{t}^{1}(x,z)=(\rho _{\omega ,t}^{1;l\bar{l}}(x,z))_{1\leq l,%
\bar{l}\leq d_{2}},\;h_{\omega ,t}^{i}(x,z))_{1\leq i\leq d_{2}},
\end{equation*}%
are random fields on $\Omega \times \lbrack 0,T]\times \mathbf{R}%
^{d_{1}}\times Z^{1}$ that are $\mathcal{P}_{T}\otimes \mathcal{B}(\mathbf{R}%
^{d_{1}})\otimes \mathcal{Z}^{1}$-measurable and 
\begin{equation*}
\zeta _{t}^{2}(x,z)=(\zeta _{\omega ,t}^{2;i}(x,z))_{1\leq i\leq
d_{1}},\;\rho _{t}^{2}(x,z)=(\rho _{\omega ,t}^{2;l\bar{l}}(x,z))_{1\leq l,%
\bar{l}\leq d_{2}},
\end{equation*}%
are random fields on $\Omega \times \lbrack 0,T]\times \mathbf{R}%
^{d_{1}}\times Z^{2}$ that are $\mathcal{R}_{T}\otimes \mathcal{B}(\mathbf{R}%
^{d_{1}})\otimes \mathcal{Z}^{2}$-measurable. We also assume that there is a
constant $C$ such that $V_{t}\leq C$ for all $(\omega ,t)\in \Omega \times
\lbrack 0,T].$

Let us describe our notation. Let $d\in \mathbf{N}$ be arbitrarily given.
For each integer $d\ge 1$, let $\mathbf{R}^{d}$ be the space of $d$%
-dimensional Euclidean points $x=(x_1,\ldots,\allowbreak x_{d})$. The dot
product of two elements $x,y\in \mathbf{R}^{d}$ is denoted by $x\cdot
y=\sum_{i=1}^{d}=x_iy_i$ and the norm of an element $x\in \mathbf{R}^{d}$ is
denoted by $|x|=\sqrt{x\cdot x}$. Let $\ell_2(\mathbf{R}^{d})$ be the space
of square-summable $\mathbf{R}^d$-valued sequences. The norm of an element $%
x\in \ell_2(\mathbf{R}^{d})$ is denoted by $|x|$. For a $d\times d$%
-dimensional matrix $A$ with real-valued entries, we denote by $\det A$ and $%
\func{tr}A$, the determinant and trace of $A$, respectively. The symmetric
part of a $d\times d$-dimensional matrix $A$ with real-valued entries is
denoted by $A_{\mathnormal{sym}} $. Let $I_{d}$ denote the $d\times d$%
-dimensional identity matrix.

For each $i\in \{1,\ldots,d_1\}$, let $\partial_i=\frac{\partial}{\partial
x_i}$ be the spatial derivative operator with respect to $x_i$ and write $%
\partial_{ij}=\partial_i\partial_j$ for each $i,j\in \{1,\ldots,d_1\}$. For
a once differentiable function $f=(f^{1}\ldots ,f^{d_{1}}):\mathbf{R}%
^{d_{1}}\rightarrow \mathbf{R}^{d}$, we denote the gradient of $f$ by $%
\nabla f=(\partial _{j}f^{i})_{1\le i\le d_1,1\le j\le d}$ and the
divergence of $f$ when $d=d_1$ by $\func{div}f =
\sum_{i=1}^{d_1}\partial_if^i$. For a once differentiable function $%
f=(f^{1\varrho},\ldots,f^{d\varrho})_{\varrho\in\mathbf{N}} : \mathbf{R}%
^{d_1}\rightarrow \ell_2(\mathbf{R}^{d})$, we denote the gradient of $f$ by $%
\nabla f=(\partial_jf^{i\varrho})_{1\le i\le d_1,1\le j\le d,\varrho\in%
\mathbf{N}} $ and the divergence of $f$ when $d=d_1$ by $\func{div}f =
(\sum_{i=1}^{d_1}\partial_if^{i\varrho})_{\varrho\in \mathbf{N}}.$

For a multi-index $\gamma =(\gamma _{1},\ldots ,\gamma _{{d}_1})\in (\mathbf{%
N}\cup \{0\})^{d_{1}}$ of length $|\gamma |:=\gamma _{1}+\cdots +\gamma
_{d_1}$, denote by $\partial ^{\gamma }$ the operator $\partial ^{\gamma
}=\partial _{1}^{\gamma _{1}}\cdots \partial _{d_1}^{\gamma _{d_1}}$, where $%
\partial _{i}^{0}$ is the identity operator for all $i\in \{1,\ldots
,d_{1}\} $, and let $x^{\gamma}=x^{\gamma_1}_1\cdots x_{d_1}^{\gamma_{d_1}}$%
, for each $x\in \mathbf{R}^{d_1}$.

For continuous functions $f:\mathbf{R}^{d_{1}}\rightarrow \mathbf{R}^{d}$,
we define 
\begin{equation*}
[f]_{0}=\sup_{x\in \mathbf{R}^{d_{1}}}|f(x)|
\end{equation*}%
and 
\begin{equation*}
[ f]_{\beta }=\sup_{x,y\in \mathbf{R}^{d_{1}},x\neq y}\frac{|f(x)-f(y)|}{%
|x-y|^{\beta }},\;\;\beta \in (0,1].
\end{equation*}
We denote by $C_{c}^{\infty }(\mathbf{R}^{d_{1}};\mathbf{R}^{d})$ the space
of infinitely differentiable $\mathbf{R}^{d}$-valued functions with compact
support in $\mathbf{R}^{d_1}$.

Let us introduce the following assumption for $m\in\mathbf{N}$ and a real
number $\beta \in [ 0,2]$.

\begin{assumption}[$m,d_{2}$]
\label{asm:coeffmain} Let $N_{0}$ be a positive constant. \ 

\begin{enumerate}
\item For all $(\omega ,t,x)\in \Omega \times \lbrack 0,T]\times \mathbf{R}%
^{d_{1}}$, the derivatives in $x$ of the random fields $b_{t}$, $%
c_{t},\sigma _{t}^{2},$ and $\upsilon _{t}^{2}$ up to order $m$ and $\sigma
_{t}^{k}$ and $\upsilon _{t}^{k}$, $k\in \{1,2\}$, up to order $m+1$ exist,
and for all $x\in \mathbf{R}^{d_{1}}$, 
\begin{equation*}
\max_{|\gamma |\leq m}\left( |\partial ^{\gamma }b_{t}(x)|+|\partial
^{\gamma }c_{t}(x)|+|\partial ^{\gamma }\nabla \sigma _{t}^{1}(x)|+|\partial
^{\gamma }\sigma _{t}^{2}(x)|+|\partial ^{\gamma }\nabla \upsilon
_{t}^{1}(x)|+|\partial ^{\gamma }\upsilon _{t}^{2}(x)|\right) \leq N_{0}.
\end{equation*}

\item For each $k\in \{1,2\}$ and all $(\omega ,t,x,z)\in \Omega \times
\lbrack 0,T]\times \mathbf{R}^{d_{1}}\times Z^{k}$, the derivatives in $x$
of the random field $\zeta_t^{k}(z)$ up to order $m$ exist, and for all $%
x\in \mathbf{R}^{d_{1}}$, 
\begin{align*}
\max_{|\gamma |\leq m}|\partial ^{\gamma }\zeta_t^{1}\left( x,z\right)
|+\max_{|\gamma |=m}\left[ \partial ^{\gamma }\zeta_t^{1}\left( \cdot
,z\right) \right] _{\frac{\beta }{2}}& \leq K_{t}^{1}\left( z\right)
,\;\;\forall z\in Z^{1}, \\
\max_{|\gamma |\leq m}|\partial ^{\gamma }\zeta_t^{2}\left( x,z\right) |&
\leq K_{t}^{2}(z),\;\;\forall z\in Z^{2},
\end{align*}%
where $K_{t}^{1}$ (resp. $K_{t}^{2})$ are $\mathcal{P}_{T}\otimes \mathcal{Z}%
^{1}\,$(resp. $\mathcal{P}_{T}\otimes \mathcal{Z}^{1})$ -measurable
processes satisfying 
\begin{equation*}
\sup_{z\in Z^{k}}K_{t}^{k}(z)+\int_{Z^{1}}K_{t}^{1}(z)^{\beta }\pi
_{t}^{1}(dz)+\int_{Z^{2}}K_{t}^{2}(z)^{2}\pi ^{2}(dz)\leq N_{0}.
\end{equation*}

\item There is a constant $\eta <1$ such that for each $k\in \{1,2\}$ on the
set all $(\omega ,t,x,z)\in \Omega \times \lbrack 0,T]\times \mathbf{R}%
^{d_{1}}\times Z^{k}$ in which $|\nabla \zeta_t^{k}(x,z)|>\eta $, 
\begin{equation}
\left| \left( I_{d_{1}}+\nabla \zeta_t^{k}(x,z)\right) ^{-1}\right| \leq
N_{0}.  \label{cond:thetagradientinverse}
\end{equation}

\item For each $k\in \{1,2\}$ and all $(\omega ,t,x,z)\in \Omega \times
\lbrack 0,T]\times \mathbf{R}^{d_{1}}\times Z^{k}$, the derivatives in $x$
of the random field $\rho _{t}^{k}(z)$ up to order $m$ exist, and for all $%
x\in \mathbf{R}^{d_{1}}$, 
\begin{align*}
\max_{|\gamma |\leq m}|\partial ^{\gamma }\rho _{t}^{1}(x,z)|+\max_{|\gamma
|=m}\left[ \rho _{t}^{1}\left( \cdot ,z\right) \right] _{\frac{\beta }{2}}&
\leq l_{t}^{1}(z), \\
\max_{|\gamma |\leq m}|\partial ^{\gamma }\rho _{t}^{2}(x,z)|& \leq
l_{t}^{2}\left( z\right) ,
\end{align*}%
where $l^{1}$ (resp. $l^{2}$) is $\mathcal{P}_{T}\otimes \mathcal{Z}^{1}$
(resp. $\mathcal{P}_{T}\otimes \mathcal{Z}^{2}$) -measurable function
satisfying 
\begin{equation*}
\int_{Z^{1}}l_{t}^{1}(z)^{2}\pi _{t}^{1}(dz)+\int_{Z^{2}}l_{t}^{2}(z)^{2}\pi
^{2}(dz)\leq N_{0}.
\end{equation*}
\end{enumerate}
\end{assumption}

Let $L^{2}=L^{2}(\mathbf{R}^{d_{1}},\mathcal{B}(\mathbf{R}^{d_{1}}),\nu ;%
\mathbf{R})$, where $\nu $ (differential is denoted by $dx$) is the Lebesgue
measure. Let $\mathcal{S}=\mathcal{S}(\mathbf{R}^{d_{1}})$ be the Schwartz
space of rapidly decreasing functions on $\mathbf{R}^{d_{1}}$. The Fourier
transform of an element $v\in \mathcal{S}$ is defined by 
\begin{equation*}
\hat{v}(\xi )=\mathcal{F}v(\xi )=\int_{\mathbf{R}^{d_{1}}}v(x)e^{-i2\pi \xi
\cdot x}dx,\;\;\xi \in \mathbf{R}^{d_{1}}.
\end{equation*}%
We denote by $\mathcal{F}^{-1}$ its inverse. Denote the space of tempered
distributions by $\mathcal{S}^{\prime }$, the dual of $\mathcal{S}$.

Let $\Delta :=\sum_{i=1}^{d_{1}}\partial _{i}^{2}$ be the Laplace operator
on $\mathbf{R}^{d_{1}}$. For $\alpha \in \mathbf{R}$, we define the Sobolev
scale 
\begin{align*}
H^{\alpha }(\mathbf{R}^{d_{1}};\mathbf{R}^{d})& =\left\{ v=(v^{i})_{1\leq
i\leq d}:v^{i}\in \mathcal{S}^{\prime }\;\;\mathnormal{and}\;\;\left( 1+4\pi
^{2}\left| \xi \right| ^{2}\right) ^{\alpha /2}\hat{v}^{i}\in L^{2}(\mathbf{R%
}^{d_{1}}),\;\;\forall i\in \{1,\ldots ,d\}\right\} \\
& =\left\{ v=(v^{i})_{1\leq i\leq d}:v^{i}\in \mathcal{S}^{\prime }\;\;%
\mathnormal{and}\;\;(I-\Delta )^{\frac{\alpha }{2}}v^{i}\in L^{2}(\mathbf{R}%
^{d_{1}}),\;\;\forall i\in \{1,\ldots ,d\}\right\}
\end{align*}%
with the norm and inner product given by 
\begin{equation*}
\| v\| _{\alpha ,d}=\left( \sum_{i=1}^{d}\left| \left( 1+4\pi ^{2}\left| \xi
\right| ^{2}\right) ^{\alpha /2}\hat{v}^{i}\right| _{L^{2}}^{2}\right)
^{1/2}=\left( \sum_{i=1}^{d}\left| \left( I-\Delta \right) ^{\alpha
/2}v^{i}\right| _{L^{2}}^{2}\right) ^{1/2}
\end{equation*}%
and%
\begin{equation*}
\left( v,u\right) _{\alpha ,d}=\sum_{i=1}^{d}\left( \left( I-\Delta \right)
^{\alpha /2}v^{i},\left( I-\Delta \right) ^{\alpha /2}u^{i}\right)
_{L^{2}},\;\;\forall u,v\in H^{\alpha }(\mathbf{R}^{d_{1}};\mathbf{R}^{d}),
\end{equation*}%
where 
\begin{equation*}
\left( I-\Delta \right) ^{\alpha /2}v^{i}=\mathcal{F}^{-1}\left( \left(
1+4\pi ^{2}\left| \xi \right| ^{2}\right) ^{\alpha /2}\hat{v}^{i}\right) .
\end{equation*}%
It is well-known that $C_{c}^{\infty }(\mathbf{R}^{d_{1}},\mathbf{R}^{d})$
is dense in $H^{\alpha }(\mathbf{R}^{d_{1}};\mathbf{R}^{d})$ for each $%
\alpha \in \mathbf{R}$. For $v\in H^1(\mathbf{R}^{d_1},\mathbf{R}^d)$ and $%
u\in H^{-1}(\mathbf{R}^{d_1};\mathbf{R}^{d})$, we let 
\begin{equation*}
\langle v, u\rangle_{1,d}=(\Lambda^1v, \Lambda^{-1}u)_{0,d},
\end{equation*}
and identify the dual of $H^1(\mathbf{R}^{d_1};\mathbf{R}^d)$ with $H^{-1}(%
\mathbf{R}^{d_1};\mathbf{R}^{d})$ through this bilinear form. Moreover, all
of the properties imposed in Section \ref{sec:DegenerateLinearStochasticEv}
for the abstract family of spaces $(H^{\alpha })_{\alpha \in \mathbf{R}}$
and operators $(\Lambda ^{\alpha })_{\alpha \in \mathbf{R}}$ holds for the
Sobolev scale. We refer the reader to \cite{Tr10} for more details about the
Sobolev scale (see the references therein as well).

For each $\alpha \in \mathbf{R}$, let $\mathbf{H}^{\alpha}( \mathbf{R}^{d_1};%
\mathbf{R}^{d};\mathcal{F}_{0}) $ be the space of all $\mathcal{F}_{0}$%
-measurable $H^{\alpha }(\mathbf{R}^{d_{1}};\mathbf{R}^{d})$-valued random
variables $\tilde{\varphi} $ satisfying $\mathbf{E}\left[\| \tilde{\varphi}%
\| _{\alpha }^{2}\right] <\infty .$

Let $\mathbf{H}^{\alpha }(\mathbf{R}^{d_{1}};\mathbf{R}^{d_{2}})$ be the
space of all $H^{\alpha }(\mathbf{R}^{d_{1}};\mathbf{R}^{d_{2}})$-valued $%
\mathcal{R}_{T}$-measurable processes $f:\Omega \times \lbrack
0,T]\rightarrow H^{\alpha }(\mathbf{R}^{d_{1}};\mathbf{R}^{d_{2}})$ such
that 
\begin{equation*}
\mathbf{E}\int_{0}^{T}\|f_{t}\|_{\alpha }^{2}dV_{t}<\infty .
\end{equation*}

Let $\mathbf{H}^{\alpha}(\mathbf{R}^{d_1};\ell_{2}(\mathbf{R}^{d}))$ be the
space of all sequences of $H^{\alpha }(\mathbf{R}^{d_{1}};\mathbf{R}^{d})$%
-valued $\mathcal{P}_{T}$-measurable processes $\tilde{g}=( \tilde{g}%
^{\varrho })_{\varrho\in \mathbf{N}}$, $\tilde{g}^{\varrho}: \Omega\times
[0,T]\rightarrow H^{\alpha }(\mathbf{R}^{d_{1}};\mathbf{R}^{d})$, satisfying

\begin{equation*}
\mathbf{E}\int_{0}^{T}\| \tilde{g}_{t}\| _{\alpha }^{2}dV_{t}=\mathbf{E}%
\int_{0}^{T}\sum_{\rho \in \mathbf{N}}\| \tilde{g}_{t}^{\varrho }\| _{\alpha
}^{2}dV_{t}<\infty \text{.}
\end{equation*}%
Let $\mathbf{H}^{\alpha }(\mathbf{R}^{d_{1}};\mathbf{R}^{d};\pi ^{1})$ be
the space of all $H^{\alpha }(\mathbf{R}^{d_{1}};\mathbf{R}^{d})$-valued $%
\mathcal{P}_{T}\otimes \mathcal{Z}^{1}$-measurable processes $\tilde{h}%
:\Omega \times \lbrack 0,T]\times \mathcal{Z}^{1}\rightarrow H^{\alpha }(%
\mathbf{R}^{d_{1}};\mathbf{R}^{d})$ such that 
\begin{equation*}
\mathbf{E}\int_{0}^{T}\int_{Z^{1}}\| \tilde{h}_{t}\left( z\right) \|
_{\alpha }^{2}\pi _{t}^{1}(dz)dV_{t}<\infty ..
\end{equation*}

For each $\alpha \in \mathbf{R}$, we set $H^{\alpha }=H^{\alpha }(\mathbf{R}%
^{d_{1}};\mathbf{R}^{d_{2}})$, $\mathbf{H}^{\alpha }(\mathcal{F}_{0})$, $%
\mathbf{H}^{\alpha }=\mathbf{H}^{\alpha }(\mathbf{R}^{d_{1}};\mathbf{R}%
^{d_{2}})$, $\mathbf{H}^{\alpha }(\mathbf{R}^{d_{1}};\ell _{2}(\mathbf{R}%
^{d_{2}}))=\mathbf{H}^{\alpha }(\ell _{2})$, $\mathbf{H}^{\alpha }(\pi ^{1})=%
\mathbf{H}^{\alpha }(\mathbf{R}^{d_{1}};\mathbf{R}^{d};\pi ^{1})$, and $\|
\cdot \| _{\alpha }=\| \cdot \| _{\alpha ,d_{2}}$, $(\cdot ,\cdot )_{\alpha
}=(\cdot ,\cdot )_{\alpha ,d_{2}},$ $\langle \cdot,\cdot\rangle_{1}=\langle
\cdot,\cdot\rangle_{1,d_2}.$ We also set $C_{c}^{\infty }=C_{c}^{\infty }(%
\mathbf{R}^{d_{1}};\mathbf{R}^{d})$.

\begin{definition}
\label{def:solutionside}Let $\varphi \in \mathbf{H}^{0}(\mathcal{F}_{0})$, $%
f\in \mathbf{H}^{-1},g\in \mathbf{H}^{0}(\ell _{2}),$ and $h\in \mathbf{H}%
^{0}(\pi ^{1}).$ An $H^{0}$-valued strongly c\`{a}dl\`{a}g process $%
u=(u_{t})_{t\leq T}$ is said to be a solution of the SIDE \eqref{eq:SIDE} if 
$u\in L^{2}(\Omega \times \lbrack 0,T],\mathcal{O}_{T},dV_{t}d\mathbf{P}%
;H^{1})$ and $\mathbf{P}$-a.s.\ for all $t\in \lbrack 0,T]$,%
\begin{align*}
u_{t}& \overset{H^{-1}}{=}\varphi +\int_{0}^{t}\left( (\mathcal{L}_{s}^{1;l}+%
\mathcal{L}_{s}^{2;l})u_{s}+b_{s}^{i}\partial _{i}u_{s}^{l}+c_{s}^{l\bar{l}%
}u_{s}^{\bar{l}}+f_{s}^{l}\right) dV_{s}+\int_{0}^{t}\left( \mathcal{N}%
_{s}^{l\varrho }u_{s}+g_{s}^{l\varrho }\right) dw_{s}^{\varrho } \\
& \quad +\int_{0}^{t}\int_{Z^{1}}\left( \mathcal{I}_{s,z}^{l}u_{s-}^{\bar{l}%
}+h_{s}^{l}(z)\right) \tilde{\eta}(ds,dz),
\end{align*}%
where $\overset{H^{-1}}{=}$ indicates that the equality holds in the $H^{-1}$%
. That is, $\mathbf{P}$-a.s.\ for all $t\in \lbrack 0,T]$ and $v\in H^{1}$, 
\begin{align*}
(v,u_{t})_{0}& =(v,u_{0})+\int_{0}^{t}\langle v,(\mathcal{L}_{s}^{1}+%
\mathcal{L}_{s}^{2})u_{s}+b_{s}^{i}\partial
_{i}u_{s}+c_{s}u_{s}+f_{s}\rangle _{1}dV_{s} \\
& \quad +\int_{0}^{t}\left( v,\left( \mathcal{N}_{s}^{l\varrho
}u_{s}+g_{s}^{l\varrho }\right) \right) _{0}dw_{s}^{\varrho
}+\int_{0}^{t}\int_{Z^{1}}\left( v,\left( \mathcal{I}_{s,z}^{l}u_{s-}^{\bar{l%
}}+h_{s}^{l}(z)\right) \right) \tilde{\eta}(ds,dz).
\end{align*}
\end{definition}

The main result of this section is the following statement.

\begin{theorem}
\label{thm:degeneratespideexist} Let Assumption \ref{asm:coeffmain}$(m,d_2 )$
hold for $m\in \mathbf{N}$ and a real number $\beta \in \lbrack 0,2] $. Then
for every $\varphi \in \mathbf{H}^{m}(\mathcal{F}_{0})$, $f\in \mathbf{H}^{m}
$, $g\in \mathbf{H}^{m+1}(\ell _{2}),h\in \mathbf{H}^{m+\frac{\beta }{2}%
}(\pi ^{1}),$ and there exists a unique solution $u=(u_{t})_{t\leq T}$ of %
\eqref{eq:SIDE} that is weakly c\`{a}dl\`{a}g as an $H^{m}$-valued process
and strongly c\`{a}dl\`{a}g as an $H^{\alpha ^{\prime }}$-valued process for
any $\alpha ^{\prime }<m$. Moreover, there is a constant $%
N=N(d_{1},d_{2},N_{0},m,\eta ,\beta )$ such that 
\begin{equation*}
\mathbf{E}\left[ \underset{t\leq T}{\sup }\| u_{t}\| _{m}^{2}\right] \leq N%
\mathbf{E}\left[ \| \varphi \| _{m}^{2}+\int_{0}^{T}\left( \| f_{t}\|
_{m}^{2}+\| g_{t}\| _{m+1}^{2}+\int_{Z^{1}}\| h_{t}(z)\| _{m+\frac{\beta }{2}%
}^{2}\pi _{t}^{1}(dz)\right) dV_{t}\right] .
\end{equation*}
\end{theorem}

\subsection{Proof of Theorem \protect\ref{thm:degeneratespideexist}}

By \cite{MiRo99} (see Examples 2.3-2.4), the stochastic integrals in %
\eqref{eq:SIDE} can be written as stochastic integrals with respect to a
cylindrical martingale. We will apply Theorem \ref{th:Degenerate} to %
\eqref{eq:SIDE} with $\alpha =m$ and $\mu =1$ by checking that Assumptions %
\ref{asm:mainSES}$(\lambda,1)$ and \ref{asm:specSES}$(\lambda,1)$ for $%
\lambda\in \{0,m\}$ are implied by Assumption \ref{asm:coeffmain}$(m,d_2)$.
We start with $\lambda =0$ as our base case and show that $\lambda =m$ can
be reduced to it.

We introduce our base assumption for $\beta \in \left[ 0,2\right] .$ {}

\begin{assumption}[$d_{2}$]
\label{asm:proofSIDEassumpzero}Let $N_{0}$ be a positive constant. \ 

\begin{enumerate}
\item For all $(\omega ,t,x)\in \Omega \times \lbrack 0,T]\times \mathbf{R}%
^{d_{1}}$, the derivatives in $x$ of the random fields $b_{t},\sigma
_{t}^{1},\sigma _{t}^{2},$ and $\func{div}\sigma _{t}^{1}$ exist, and for
all $x\in \mathbf{R}^{d_{1}}$, 
\begin{equation*}
|\nabla \func{div}\sigma _{t}^{1}\left( x\right) |+|\sigma _{t}^{k}\left(
x\right) |+|\mathnormal{\nabla }\sigma _{t}^{k}(x)|+|\mathnormal{\func{div}}%
b_{t}(x)|+|c_{t}(x)|+|\upsilon _{t,\mathnormal{sym}}^{2}\left( x\right)
|+|\nabla \upsilon _{t}^{1}\left( x\right) |\leq N_{0}.
\end{equation*}

\item For each $k\in \{1,2\}$ and all $(\omega ,t,x,z)\in \Omega \times
\lbrack 0,T]\times \mathbf{R}^{d_{1}}\times Z^{k}$, the derivatives in $x$
of the random fields $\zeta_t^{k}(z)$ exist, and for all $x\in \mathbf{R}%
^{d_{1}}$, 
\begin{gather*}
|\zeta_t^{1}(x,z)|\leq K_{t}^{1}\left( z\right) ,\quad |\nabla
\zeta_t^{1}\left( x,z\right) |\leq \bar{K}_{t}^{1}\left( z\right) ,\quad %
\left[ \func{div}\zeta_t^{1}\left( \cdot ,z\right) \right] _{\frac{\beta }{2}%
}\leq \tilde{K}_{t}^{1}\left( z\right) ,\;\;\forall z\in Z^{1}, \\
|\zeta_t^{2}(x,z)|\leq K_{t}^{2}\left( z\right) ,\quad |\nabla
\zeta_t^{2}\left( x,z\right) |\leq \bar{K}_{t}^{2}\left( z\right)
,\;\;\forall z\in Z^{2},
\end{gather*}%
where $K_{t}^{1},\bar{K}_{t}^{1},\tilde{K}_{t}^{1}$ (resp. $K_{t}^{2},\bar{K}%
_{t}^{2})$ are $\mathcal{P}_{T}\otimes \mathcal{Z}^{1}\,$(resp. $\mathcal{P}%
_{T}\otimes \mathcal{Z}^{2})$ -measurable processes satisfying 
\begin{align*}
\sup_{z\in Z^{1}}\left( K_{t}^{1}(z)+\bar{K}_{t}^{1}(z)+\tilde{K}%
_{t}^{1}\left( z\right) \right) +\int_{Z^{1}}\left( K_{t}^{1}(z)^{\beta }+%
\bar{K}_{t}^{1}\left( z\right) ^{2}+\tilde{K}_{t}^{1}\left( z\right)
^{2}\right) \pi _{t}^{1}(dz)& \leq N_{0}, \\
\sup_{z\in Z^{2}}\left( K_{t}^{2}(z)+\bar{K}_{t}^{2}(z)\right) +\int_{Z^{2}}%
\bar{K}_{t}^{2}(z)^{2}\pi ^{2}(dz)& \leq N_{0}.
\end{align*}

\item There is a constant $\eta <1$ such that for each $k\in \{1,2\}$ on the
set all $(\omega ,t,x,z)\in \Omega \times \lbrack 0,T]\times \mathbf{R}%
^{d_{1}}\times Z^{k}$ in which $|\nabla \zeta_t^{k}(x,z)|>\eta $, 
\begin{equation*}
\left| \left( I_{d_{1}}+\nabla \zeta_t^{k}(x,z)\right) ^{-1}\right| \leq
N_{0}.
\end{equation*}

\item For each $k\in \{1,2\}$ and all $(\omega ,t,x,z)\in \Omega \times
\lbrack 0,T]\times \mathbf{R}^{d_{1}}\times Z^{k}$, $|\rho
_{t}^{k}(x,z)|\leq l_{t}^{k}(z)$, and for all $(\omega ,t,z)\in \Omega
\times \lbrack 0,T]\times Z^{1}$, $\left[ \rho _{t,\mathnormal{sym}%
}^{1}(\cdot ,z)\right] _{\frac{\beta }{2}}\leq \tilde{l}_{t}^{1}\left(
z\right) ,$ where $l^{k}$ (resp., $\tilde{l}^{1}$) is a $\mathcal{P}%
_{T}\otimes \mathcal{Z}^{k}$-measurable (resp. $\mathcal{P}_{T}\otimes 
\mathcal{Z}^{1}$-measurable) functions satisfying 
\begin{equation*}
\int_{Z^{1}}\left( l_{t}^{1}(z)^{2}+\tilde{l}_{t}^{1}\left( z\right)
^{2}\right) \pi ^{1}(dz)+\int_{Z^{2}}l_{t}^{2}(z)^{2}\pi ^{2}(dz)\leq N_{0}.
\end{equation*}
\end{enumerate}
\end{assumption}

Note that Assumption \ref{asm:proofSIDEassumpzero}$(d_{2})$ is weaker than
Assumption \ref{asm:coeffmain}$(0,d_{2}).$

Let us make the following convention for the remainder of this section. If
we do not specify to which space the parameters $\omega ,t,x,y,$ and $z$
belong, then we mean $\omega \in \Omega $, $t\in \lbrack 0,T]$, $x\in 
\mathbf{R}^{d_{1}}$, and $z\in Z^{k}$. Moreover, unless otherwise specified,
all statements hold for all $\omega ,t,x,y,$ and $z$ independent of any
constant $N$ introduced is independent of $\omega ,t,x,y,$ and $z$. We will
also drop the dependence of processes $t,x,$ and $z$ when we feel it will
not obscure our argument. Lastly, all derivatives and H\"{o}lder norms are
taken with respect to $x\in \mathbf{R}^{d_{1}}$.

\begin{remark}
\label{rem:PropertiesofH} Let Assumption \ref{asm:proofSIDEassumpzero}$%
\left( d_{2}\right) $ hold. For each $k$ and $\theta \in \lbrack 0,1],$ on
the set all $\omega ,t,$ and $z$ in which $|K_{t}^{k}(z)|\leq \eta $, we
have 
\begin{equation*}
\left| (I_{d_{1}}+\theta \nabla \zeta_t^{k}(x,z))^{-1}\right| \leq \frac{1}{%
1-\theta \eta }.
\end{equation*}%
Moreover, for each $k$ and all $\omega ,t,$ and $z$, we have 
\begin{equation*}
\left| (I_{d_{1}}+\nabla \zeta_t^{k}(x,z))^{-1}\right| \leq \max \left( 
\frac{1}{1-\theta \eta },N_{0}\right) .
\end{equation*}%
Therefore, by Hadamard's theorem (see, e.g., Theorem 0.2 in \cite{DeGoZa94}
or 51.5 in \cite{Be77}):

\begin{itemize}
\item for each $k$ and $\theta \in \lbrack 0,1],$ on the set all $\omega ,t,$
and $z$ in which $|K_{t}^{k}(z)|\leq \eta $, the mapping 
\begin{equation*}
\tilde{\zeta}_{t,\theta }^{k}(x,z):=x+\theta \zeta_t^{k}(x,z)
\end{equation*}%
is a global diffeomorphism in $x$;

\item for each $k$ and all $\omega ,t,$ and $z$, the mapping 
\begin{equation*}
\tilde{\zeta}_{t}^{k}(x,z)=\tilde{\zeta}_{t,1}^{k}(x,z)=x+\zeta_t^{k}(x,z)
\end{equation*}%
is a global diffeomorphism in $x$.
\end{itemize}

When inverse of the mapping $x\mapsto \tilde{\zeta}_{t,\theta }^{k}(x,z)$
exists, we denote it by 
\begin{equation*}
\tilde{\zeta}_{t,\theta }^{k;-1}(x,z)=\left( \tilde{\zeta}_{t,\theta
}^{k;-1;j}(x,z)\right) _{1\leq j\leq d_{1}}
\end{equation*}%
and note that 
\begin{equation*}
\tilde{\zeta}_{t,\theta }^{k;-1}(x,z)=x-\theta \zeta_t^{k}(\tilde{\zeta}%
_{t,\theta }^{k;-1}(x,z),z).
\end{equation*}%
Furthermore, for each $k$ and $\theta \in \lbrack 0,1],$ on the set all $%
\omega ,t,$ and $z$ in which $|K_{t}^{k}(z)|\leq \eta $, there is a constant 
$N=N(d_{1},N_{0},\eta )$ such that 
\begin{equation}
|\nabla \tilde{\zeta}_{t,\theta }^{k;-1}(x,z)|\leq N
\label{ineq:BoundtildHtheta}
\end{equation}%
and for each $k$ and all $\omega ,t,$ and $z$, 
\begin{equation}
|\nabla \tilde{\zeta}_{t}^{k;-1}(x,z)|\leq N.  \label{ineq:BoundtildH}
\end{equation}%
Using simple properties of the determinant, we can easily show that there is
a constant $N=N(d_{1})$ such that for an arbitrary real-valued $d_{1}\times
d_{1}$ matrix $A$, 
\begin{equation*}
|\det (I_{d_{1}}+A)-1|\leq N|A|\quad \mathnormal{and}\quad |\det
(I_{d_{1}}+A)-1-\func{tr}A|\leq N|A|^{2}.
\end{equation*}%
Thus, there is a constant $N=N(d_{1},N_{0},\eta )$ such that 
\begin{equation}
|\det \nabla \tilde{\zeta}^{k;-1}-1|=\left| \det \left( I_{d}-\zeta_t^{k}(%
\tilde{\zeta}_{t}^{k;-1})\right) -1\right| \leq N|\nabla \zeta^{k}(\tilde{%
\zeta}^{k;-1})|  \label{ineq:detHminus1}
\end{equation}%
and 
\begin{equation*}
\left| \det \nabla \tilde{\zeta}^{k;-1}-1+\func{div}\left( \zeta^{k}(\tilde{%
\zeta}^{k;-1})\right) \right| \leq N|\nabla \zeta^{k}(\tilde{\zeta}%
^{k;-1})|^{2}.
\end{equation*}%
Since $\partial _{l}\tilde{\zeta}^{k;-1;j}=\delta _{lj}-\partial
_{m}\zeta^{k;j}(\tilde{\zeta}^{k;-1}\partial _{l}\tilde{\zeta}^{k;-1;m}),$
we have 
\begin{equation*}
|\func{div}F^{k}-\func{div}\zeta^{k}(\tilde{\zeta}^{k;-1})|=|\partial
_{j}\zeta^{k;l}(\tilde{\zeta}_{t}^{k;-1}(\partial _{l}\tilde{\zeta}%
^{k;-1;j}-\delta _{lj})|\leq N|\nabla \zeta^{k}(\tilde{\zeta}^{k;-1})|^{2},
\end{equation*}%
and thus 
\begin{equation}
\left| \det \nabla \tilde{\zeta}^{k;-1}-1+\func{div}\zeta^{k}(\tilde{\zeta}%
^{k;-1})\right| \leq N|\nabla \zeta^{k}(\tilde{\zeta}^{k;-1})|^{2}.
\label{ineq:detHminus1minustr}
\end{equation}
\end{remark}

In the following three lemmas, we will show that Assumptions \ref%
{asm:mainSES}$(\lambda ,1)$ and \ref{asm:specSES}$(\lambda ,1)$ for $\lambda
\in \{0,m\}$ hold under Assumption \ref{asm:proofSIDEassumpzero}$(\beta )$
for any $\beta \in \left[ 0,2\right] $. For each $l\in \{1,\ldots ,d_{2}\}$
and all $\phi \in C_{c}^{\infty }$, let 
\begin{gather*}
\mathcal{L}_{t}^{l}\phi =(\mathcal{L}_{t}^{1;l}+\mathcal{L}_{t}^{2;l})\phi
+b_{t}^{i}\partial _{i}\phi ^{l}+c_{t}^{l\bar{l}}\phi ^{\bar{l}}, \\
\mathcal{A}_{t}^{1;l}\phi =\frac{1}{2}\sigma _{t}^{1;i\varrho }\sigma
_{t}^{1;j\varrho }\partial _{ij}\phi ^{l}+\sigma _{t}^{1;i\varrho }\upsilon
_{t}^{1;l\bar{l}\varrho }\partial _{i}\phi ^{\bar{l}},\quad \mathnormal{and}%
\quad \mathcal{J}_{t}^{1}\phi =\mathcal{L}_{t}^{1}\phi -\mathcal{A}%
_{t}^{1}\phi .
\end{gather*}

\begin{lemma}
\label{lem:GrowthSIDE}Let Assumption \ref{asm:proofSIDEassumpzero}$(d_{2})$
hold. Then there is a constant $N=N(d_{1},d_{2},N_{0},\eta )$ such that for
all $(\omega ,t)\in \Omega \times \lbrack 0,T]$ and $v\in H^{1},$ 
\begin{gather*}
\| \mathcal{L}_{t}v\| _{-1}\leq N\| v\| _{1},\quad \| \mathcal{A}_{t}v\|
_{-1}\leq N\| v\| _{1},\quad \| \mathcal{J}_{t}^{1}v\| _{-1}\leq N\| v\|
_{1}, \\
\quad \| \mathcal{N}_{t}v\| _{0}\leq N\| \phi \| _{1},\quad \text{%
\textnormal{and}}\quad \int_{Z^{1}}\| \mathcal{I}_{t,z}v\| _{0}^{2}\pi
_{t}^{1}(dz)\leq N\| v\| _{1}^{2}.
\end{gather*}
\end{lemma}

\begin{proof}
First we will show that there is a constant $N$ such that 
\begin{equation*}
(\psi ,\mathcal{L}_{t}\phi )_{0}\leq N\| \psi \| _{1}\| \phi \| _{1},\quad
\forall \phi \in C_{c}^{\infty }.
\end{equation*}%
Once this is established, we know that $\mathcal{L}$ extends to a linear
operator from $H^{1}$ to $H^{-1}$ (still denoted by $\mathcal{L}$) and $\| 
\mathcal{L}_{t}v\| _{-1}\leq N\| v\| _{1}$, for all $v\in H^{1}$. Using
Taylor's formula and the divergence theorem, we get that for all and all $%
\phi ,\psi \in C_{c}^{\infty }$ 
\begin{equation*}
(\psi ,\mathcal{L}\phi )_{0}=\sum_{k=1}^{2}\left( (\psi ,\mathfrak{L}%
_{t}^{k}\phi )_{0}+(\partial _{i}\psi ,\mathfrak{Y}_{t}^{k;i}\phi
)_{0}+(\psi ,b_{t}\partial _{i}\phi )_{0}+(\psi ,c_{t}\phi )_{0}\right) ,
\end{equation*}%
where for each $k\in \{1,2\}$, $l\in \{1,\ldots ,d_{2}\}$, and $i\in
\{1,2,\ldots ,d_{1}\}$, 
\begin{align*}
\mathfrak{L}^{k;l}\phi :& =-\int_{\bar{K}^{k}<\eta }\int_{0}^{1}\left( \phi
^{l}(\tilde{\zeta}_{\theta }^{k})-\phi ^{l}\right) \partial
_{i}\zeta^{k;i}d\theta \pi ^{k}(dz) \\
& \quad -\int_{\bar{K}^{k}<\eta }\int_{0}^{1}\theta \partial _{j}\phi ^{l}(%
\tilde{\zeta}_{\theta }^{k})\partial _{i}\zeta^{k;j}\zeta^{k;i}d\theta \pi
^{k}(dz) \\
& \quad +\int_{Z^{k}}\rho ^{k;l\bar{l}}\left( \phi ^{l}(\tilde{\zeta}%
^{k})-\phi ^{l}\right) \pi ^{k}(dz)+\sigma ^{k;i\varrho }\upsilon ^{k;l\bar{l%
}\varrho }\partial _{i}\phi ^{\bar{l}} \\
& \quad +\int_{\bar{K}^{k}>\eta }\left( \phi ^{l}(\tilde{\zeta}^{k})-\phi
^{l}-\zeta^{k;i}\partial _{i}\phi ^{l}\right) \pi ^{k}(dz), \\
\mathfrak{Y}^{k;li}\phi :& =-\int_{\bar{K}^{k}<\eta }\int_{0}^{1}\left( \phi
^{l}(\tilde{\zeta}_{\theta }^{k})-\phi ^{l}\right) \zeta^{k;i}d\theta \pi
^{k}(dz)-\frac{1}{2}\partial _{i}\left( \sigma ^{k;i\varrho }\sigma
^{k;j\varrho }\right) \partial _{j}\phi ^{l}.
\end{align*}%
For the remainder of the proof, we make the convention that statements hold
for all $\phi ,\psi \in C_{c}^{\infty }$ and that all constants $N$ are
independent of $\phi $. By Minkowski's integral inequality and Holder's
inequality, we have (using the notation of Remark \ref{rem:PropertiesofH})%
\begin{gather*}
||\mathfrak{L}^{k}\phi ||_{0}\leq \left( \int_{\bar{K}^{k}<\eta
}(K^{k})^{2}\pi ^{k}(dz)\right) ^{\frac{1}{2}}\int_{0}^{1}\left( \int_{%
\mathbf{R}^{d_{1}}}\int_{\bar{K}^{k}<\eta }|\phi (\tilde{\zeta}_{\theta
}^{k}(z))-\phi |^{2}\pi ^{k}(dz)dx\right) ^{\frac{1}{2}}d\theta \\
+\int_{\bar{K}^{k}<\eta }K^{k}(z)^{2}\int_{0}^{1}\left( \int_{\mathbf{R}%
^{d_{1}}}|\nabla \phi (\tilde{\zeta}_{\theta }^{k})|^{2}dx\right) ^{\frac{1}{%
2}}\pi ^{k}(dz)\theta d\theta \\
+\left( \int_{Z^{1}}(l^{k})^{2}\pi ^{k}(dz)\right) ^{\frac{1}{2}}\left(
\int_{\mathbf{R}^{d_{1}}}\int_{Z^{1}}|\phi (\tilde{\zeta}^{k})-\phi |^{2}\pi
^{k}(dz)dx\right) ^{\frac{1}{2}} \\
+N\| \nabla \phi \| _{0}+\int_{\bar{K}^{k}\geq \eta }\int_{0}^{1}\left(
\int_{\mathbf{R}^{d_{1}}}|\phi (\tilde{\zeta}^{k})-\phi -\zeta^{k;i}\partial
_{i}\phi |^{2}dx\right) ^{\frac{1}{2}}d\theta \pi ^{k}(dz),
\end{gather*}%
and for all $i\in \{1,\ldots ,d_{1}\},$ 
\begin{equation*}
||\mathfrak{Y}^{k;i}\phi ||_{0}\leq \left( \int_{\bar{K}^{k}<\eta
}(K^{k})^{2}\pi ^{k}(dz)\right) ^{\frac{1}{2}}\int_{0}^{1}\left( \int_{%
\mathbf{R}^{d_{1}}}\int_{\bar{K}^{k}<\eta }|\phi (\tilde{\zeta}_{\theta
}^{k})-\phi |^{2}\pi ^{k}(dz)dx\right) ^{\frac{1}{2}}d\theta +N\| \nabla
\phi \| _{0}.
\end{equation*}%
Applying the change of variable formula and appealing to %
\eqref{ineq:BoundtildHtheta}, we find that 
\begin{gather*}
\int_{\mathbf{R}^{d_{1}}}\int_{\bar{K}^{k}<\eta }|\phi (\tilde{\zeta}%
_{\theta }^{k})-\phi |^{2}\pi ^{k}(dz)dx\leq \theta \int_{\bar{K}^{k}<\eta
}\int_{0}^{1}\int_{\mathbf{R}^{d_{1}}}|\nabla \phi (\tilde{\zeta}_{\theta 
\bar{\theta}}^{k})|^{2}|\zeta^{k}|^{2}dxd\bar{\theta} \\
\leq \theta \int_{\bar{K}^{k}<\eta }(K^{k})^{2}\int_{0}^{1}\int_{\mathbf{R}%
^{d_{1}}}|\nabla \phi |^{2}\left| \det \nabla \tilde{\zeta}_{\theta \bar{%
\theta}}^{k;-1}\right| dxd\bar{\theta}\leq N\theta \| \nabla \phi \|
_{0}^{2}.
\end{gather*}%
Similarly, since $\pi ^{k}(\{z\in Z^{k}:\bar{K}^{k}\geq \eta \})\leq N_{0}$,
we have 
\begin{gather*}
\int_{\mathbf{R}^{d_{1}}}\int_{\bar{K}^{k}\geq \eta }|\phi (\tilde{\zeta}%
^{k})-\phi |^{2}\pi ^{k}(dz)dx\leq 2\int_{\mathbf{R}^{d_{1}}}\int_{\bar{K}%
^{k}\geq \eta }\left( |\phi (\tilde{\zeta}^{k})|^{2}+|\phi |^{2}\right) \pi
^{k}(dz)dx \\
\leq \int_{\bar{K}^{k}\geq \eta }\int_{\mathbf{R}^{d_{1}}}|\phi |^{2}\left(
1+\left| \det \nabla \tilde{\zeta}^{k;-1}\right| \right) dx\pi ^{k}(dz)\leq
N\| \phi \| _{0}^{2}
\end{gather*}%
and 
\begin{gather*}
\int_{\bar{K}^{k}\geq \eta }\left( \int_{\mathbf{R}^{d_{1}}}|\phi (\tilde{%
\zeta}^{k})-\phi -\zeta^{k;i}\partial _{i}\phi |^{2}dx\right) ^{\frac{1}{2}%
}\pi ^{k}(dz) \\
\leq N\int_{\bar{K}^{k}\geq \eta }\left( \int_{\mathbf{R}^{d_{1}}}\left(
|\phi |^{2}\left( 1+\left| \det \nabla \tilde{\zeta}^{k;-1}\right| \right)
+(K^{k})^{2}|\nabla \phi |^{2}\right) dx\right) ^{\frac{1}{2}}\pi
^{k}(dz)\leq N\| \phi \| _{1},
\end{gather*}%
where in the last inequality we used \eqref{ineq:BoundtildH}. Moreover, 
\begin{gather*}
\int_{\bar{K}^{k}<\eta }(K^{k})^{2}\left( \int_{\mathbf{R}^{d_{1}}}|\nabla
\phi (\tilde{\zeta}_{\theta }^{k})|^{2}dx\right) ^{\frac{1}{2}}\pi ^{k}(dz)
\\
\leq \int_{\bar{K}^{k}<\eta }(K^{k})^{2}\left( \int_{\mathbf{R}%
^{d_{1}}}|\nabla \phi |^{2}\det \nabla \tilde{\zeta}_{\theta }^{k;-1}\,
dx\right) ^{\frac{1}{2}}\pi ^{k}(dz)\leq N\| \nabla \phi \| _{0}.
\end{gather*}%
Combining the above estimates, we get that $(\psi ,\mathcal{L}\phi )_{0}\leq
N\| \psi \| _{1}\| \phi \| _{1}.$ It is clear from the above computation
that 
\begin{equation*}
\| \mathcal{A}_{t}^{1}\phi \| _{-1}\leq N\| \phi \| _{1},\quad \| \mathcal{J}%
_{t}^{1}\phi \| _{-1}\leq N\| \phi \| _{1},
\end{equation*}%
where actually $\mathcal{A}^{1}$ and $\mathcal{J}^{1}$ are actually
extensions of the operators defined above. The inequality $\| \mathcal{N}%
\phi \| _{0}\leq N\|\phi\|_{1}$ can easily be obtained. Following similar
calculations to ones we derived above (using \eqref{ineq:BoundtildH} and %
\eqref{ineq:BoundtildHtheta}), we obtain 
\begin{equation*}
\int_{Z^{1}}||\mathcal{I}\phi ||_{0}^{2}\pi ^{1}(dz)\leq N(A_{1}+A_{2}),
\end{equation*}%
where%
\begin{align*}
A_{1}& :=\int_{\mathbf{R}^{d_{1}}}\int_{\bar{K}^{1}\leq \eta
}\int_{0}^{1}|\nabla \phi (\tilde{\zeta}_{\theta }^{1})|^{2}|\xi^{1}|^{2}\pi
^{1}(dz)d\theta dx \\
& \quad +\int_{\mathbf{R}^{d_{1}}}\int_{K^{1}>\eta }\left( |\phi (\tilde{%
\zeta}^{1})|^{2}+|\phi |^{2}\right) \pi ^{1}(dz)dx\leq N\|\phi\|_{1}^{2}
\end{align*}%
and 
\begin{equation*}
A_{2}:=\int_{\mathbf{R}^{d_{1}}}\int_{Z^{1}}\rho ^{1;l\bar{l}}\phi ^{\bar{l}%
}(\tilde{\zeta}^{1})\pi ^{1}\left( dz\right) dx\leq N\|\phi\|_{0}^{2}.
\end{equation*}
\end{proof}

\begin{lemma}
\label{lem:CoercvitySIDE} Let Assumption \ref{asm:proofSIDEassumpzero}$%
(d_{2})$ hold. Then there is a constant $N=N(d_{1},d_{2},N_{0},\eta ,\beta )$
such that for all $(\omega ,t)\in \Omega \times \lbrack 0,T]$ and all $v\in
H^{1},$ 
\begin{align}
2\langle v,\mathcal{L}_{t}^{2}v+b_{t}^{i}\partial _{i}v_{t}+c_{t}^{\cdot 
\bar{l}}v_{t}^{\bar{l}}\rangle _{1}& +\frac{1}{4}(\sigma _{t}^{2;i\varrho
}\partial _{i}v,\sigma _{t}^{2;j\varrho }\partial _{j}v)_{0}  \notag \\
& +\frac{1}{4}\int_{Z^{2}}\| v(\tilde{\zeta}_{t}^{2}(z))-v\| _{0}^{2}\pi
_{t}^{2}(dz)\leq N||v||_{0}^{2},  \label{ineq:coercivityuncorrelated}
\end{align}%
\begin{equation}
2\langle v,\mathcal{A}_{t}^{1}v\rangle _{1}+||\mathcal{N}_{t}v||_{0}^{2}\leq
N||v||_{0}^{2},\quad 2\langle v,\mathcal{J}_{t}^{1}v\rangle
_{1}+\int_{Z^{1}}||\mathcal{I}_{t,z}v||_{0}^{2}\pi _{t}^{1}(dz)\leq
N||v||_{0}^{2},  \label{ineq:coercivitygronwall}
\end{equation}%
and 
\begin{align}
2\langle v,\mathcal{L}_{t}v+f_{t}\rangle _{1}& +||\mathcal{N}%
_{t}v+g_{t}||_{0}^{2}+\int_{Z_{1}}||\mathcal{I}_{t,z}v+h_{t}\left( z\right)
||_{0}^{2}\pi ^{1}\left( dz\right)  \notag \\
& +\frac{1}{4}(\sigma _{t}^{2;i\varrho }\partial _{i}v,\sigma
_{t}^{2;j\varrho }\partial _{j}v)_{0}+\frac{1}{4}\int_{Z^{2}}\| v(\tilde{%
\zeta}_{t}^{2}(z))-v\| _{0}^{2}\pi _{t}^{2}(dz)  \notag \\
& \leq N\left( \| v\|
_{0}^{2}+||f_{t}||_{0}^{2}+||g_{t}||_{1}^{2}+\int_{Z^{1}}||h_{t}\left(
z\right) ||_{\frac{\beta }{2}}^{2}\pi _{t}^{1}(dz)\right) .
\label{ineq:coercivityfull}
\end{align}
\end{lemma}

\begin{proof}
For the remainder of the proof, we make the convention that statements hold
for all $\phi \in C_{c}^{\infty }$ and that all constants $N$ are
independent of $\phi $. Using the divergence theorem, we get 
\begin{align*}
2\langle \phi ,\mathcal{A}^{1}\phi \rangle _{1}& +||\mathcal{N}\phi
||_{0}^{2}=\frac{1}{2}\int_{\mathbf{R}^{d_{1}}}\left( |\func{div}\sigma
^{1}|^{2}+2\sigma ^{1;i}\partial _{i}\func{div}\sigma ^{1}+\partial
_{j}\sigma ^{1;i}\partial _{i}\sigma ^{1;j}\right) |\phi |^{2}dx \\
& \quad +\int_{\mathbf{R}^{d_{1}}}\left( |\upsilon ^{1}\phi |^{2}-2\phi
^{l}\left( \upsilon _{\mathnormal{sym}}^{1;l\bar{l}}\func{div}\sigma
^{1}+\sigma ^{1;i}\partial _{i}\upsilon _{\mathnormal{sym}}^{1;l\bar{l}%
}\right) \phi ^{\bar{l}}\right) dx\leq N\|\phi\|_{0}^{2}.
\end{align*}%
Rearranging terms and using the identity $2a(b-a)=-|b-a|^{2}+|b|^{2}-|a|^{2}$%
, $a,b\in \mathbf{R}$, we obtain 
\begin{equation*}
2\langle \phi ,\mathcal{J}^{1}\phi \rangle _{1}+\int_{Z_{1}}||\mathcal{I}%
\phi ||_{0}^{2}\pi _{t}^{1}(dz)=A_{1}+A_{2},
\end{equation*}%
where%
\begin{align*}
A_{1}& :=\int_{\mathbf{R}^{d_{1}}}\int_{Z^{1}}\left( |\phi (\tilde{\zeta}%
^{1})|^{2}-|\phi |^{2}-2\phi \zeta^{1;i}\partial _{i}\phi \right) \pi
^{1}(dz)dx \\
A_{2}& :=2\int_{\mathbf{R}^{d_{1}}}\int_{Z^{1}}\left( \phi ^{l}(\tilde{\zeta}%
^{1})\rho ^{1;l\bar{l}}\phi ^{\bar{l}}(\tilde{\zeta}^{1})-\phi ^{l}\rho ^{1;%
\bar{l}}\phi ^{l}\right) \pi ^{1}(dz)dx +\int_{\mathbf{R}^{d_{1}}}%
\int_{Z^{1}}|\rho ^{1}\phi (\tilde{\zeta}^{1})|^{2}\pi ^{1}(dz)dx,
\end{align*}%
Since 
\begin{equation*}
|\func{div}\zeta^{1}(\tilde{\zeta}^{1;-1})-\func{div}\zeta^{1}|\leq \lbrack 
\func{div}\zeta^{1}]_{\frac{\beta }{2}}(K^{1})^{\frac{\beta }{2}}\leq (%
\tilde{K}^{1})^{2}+(K^{1})^{\beta },
\end{equation*}%
changing the variable of integration and making use of the estimate %
\eqref{ineq:detHminus1minustr}, we obtain%
\begin{equation*}
A_{1}\leq \int_{\mathbf{R}^{d_{1}}}|\phi |^{2}\int_{Z^{1}}|\det \nabla 
\tilde{\zeta}^{1;-1}-1+\func{div}\zeta^{1}|\pi ^{1}(dz)dx\leq
N\|\phi\|_{0}^{2}
\end{equation*}%
and%
\begin{align*}
A_{2}& =2\int_{\mathbf{R}^{d_{1}}}\int_{Z^{1}}\phi ^{l}\left( \rho ^{1;l\bar{%
l}}(\tilde{\zeta}^{1;-1})-\rho ^{1;l\bar{l}}\right) \phi ^{\bar{l}}\pi
^{1}(dz)dx \\
& \quad +\int_{\mathbf{R}^{d_{1}}}\int_{Z^{1}}2\phi ^{l}\rho ^{1;l\bar{l}}(%
\tilde{\zeta}^{1;-1})\phi ^{\bar{l}}\left( \det \nabla \tilde{\zeta}%
^{1;-1}-1\right) \pi ^{1}(dz)dx \\
& \quad +\int_{\mathbf{R}^{d_{1}}}\int_{Z^{1}}|\rho ^{1}(\tilde{\zeta}%
^{1;-1})\phi |^{2}\det \nabla \tilde{\zeta}^{1;-1}\pi
^{1}(dz)dx=:A_{21}+A_{22}+A_{23}.
\end{align*}%
Owing to \eqref{ineq:detHminus1} and Holder's inequality, we have 
\begin{equation*}
A_{22}+A_{23}\leq N\int_{Z^{1}}((l^{1})^{2}+(K^{1})^{2})\pi ^{1}(dz)||\phi
||_{0}^{2}.
\end{equation*}%
For $\beta >0,$ we have 
\begin{equation*}
A_{21}\leq N\int_{Z^{1}}\left[ \rho _{\mathnormal{sym}}^{1}\right] _{\frac{%
\beta }{2}}(K^{1})^{\frac{\beta }{2}}\pi ^{1}(dz)\|\phi\|_{0}^{2}\leq
N\int_{Z^{1}}\left( (\tilde{l}^{1})^{2}+(K^{1})^{\beta }\right) \pi
^{1}(dz)\|\phi\|_{0}^{2}\leq N\|\phi\|_{0}^{2}
\end{equation*}%
and for $\beta =0$, using Holder's inequality, we get 
\begin{equation*}
A_{21}\leq N\|\phi\|_{0}^{2}\int_{Z^{1}}(l^{1})^{2}\pi ^{1}(dz).
\end{equation*}%
By the divergence theorem, we have 
\begin{equation*}
2\langle \phi ,\mathcal{L}^{2}\phi \rangle _{0}=B_{1}+B_{2}+B_{3}\text{, }
\end{equation*}%
where 
\begin{align*}
B_{1}& :=\int_{\mathbf{R}^{d_{1}}}\left( \phi ^{l}\sigma ^{2;i\varrho
}\sigma ^{2;j\varrho }\partial _{ij}\phi ^{l}+2\sigma ^{2;i\varrho }\upsilon
^{2;l\bar{l}\varrho }\partial _{i}\phi ^{\bar{l}}\right) dx, \\
B_{2}& :=2\int_{\mathbf{R}^{d_{1}}}\int_{Z^{2}}\phi ^{l}\left( \phi ^{l}(%
\tilde{\zeta}^{2})-\phi ^{l}-\zeta^{2;i}\partial _{i}\phi ^{l}\right) \pi
^{2}(dz)dx, \\
B_{3}& :=2\int_{\mathbf{R}^{d_{1}}}\int_{Z^{2}}\phi ^{l}\rho ^{2;l\bar{l}%
}\left( \phi ^{\bar{l}}(\tilde{\zeta}^{2})-\phi ^{\bar{l}}\right) \pi
^{2}(dz)dx.
\end{align*}%
Owing to the divergence theorem, we have 
\begin{equation*}
(\phi ,\sigma ^{2;i\varrho }\sigma ^{2;j\varrho }\partial _{ij}\phi
)_{0}=-\left( \left( \sigma ^{2;i\varrho }\sigma ^{2;j\varrho }\partial
_{i}\phi +\phi \left( \sigma ^{2;j\varrho }\func{div}\sigma ^{2;\varrho
}+\sigma ^{2;i\varrho }\partial _{i}\sigma ^{2;j\varrho }\right) \right)
,\partial _{j}\phi \right) _{0}
\end{equation*}%
\begin{equation*}
(\phi \sigma ^{2;i\varrho }\partial _{i}\sigma ^{2;j\varrho },\partial
_{j}\phi )_{0}=-\frac{1}{2}\left( \phi \left( \partial _{j}\sigma
^{2;i\varrho }\partial _{i}\sigma ^{2;j\varrho }+\sigma ^{2;j\varrho
}\partial _{j}\func{div}\sigma ^{2;\varrho }\right) ,\phi \right) _{0},
\end{equation*}%
\begin{equation*}
(\phi \sigma ^{2;j\varrho }\partial _{j}\func{div}\sigma ^{2;\varrho },\phi
)_{0}=-(\phi |\func{div}\sigma ^{2}|^{2},\phi )_{0}+2(\phi \sigma
^{2;i\varrho }\func{div}\sigma ^{2;\varrho },\partial _{j}\phi )_{0},
\end{equation*}%
and hence, 
\begin{align*}
(\phi ,\sigma _{t}^{2;i\varrho }\sigma _{t}^{2;j\varrho }\partial _{ij}\phi
)_{0}& =-\left( \sigma ^{2;i\varrho }\sigma ^{2;j\varrho }\partial _{i}\phi
^{l}\partial _{j}\phi ^{l}+2\phi \sigma ^{2;j\varrho }\func{div}\sigma
^{2;\varrho },\partial _{j}\phi \right) _{0} \\
& \quad +\frac{1}{2}\left( \phi \left( \partial _{j}\sigma ^{2;i\varrho
}\partial _{i}\sigma ^{2;j\varrho }-|\func{div}\sigma ^{2}|^{2}\right) ,\phi
\right) _{0}.
\end{align*}%
Thus, by Young's inequality,%
\begin{equation*}
B_{1}\leq -\frac{1}{2}\int \partial _{i}\phi ^{l}\sigma ^{2;i\varrho }\sigma
^{2;j\varrho }\partial _{j}\phi ^{l}dx+N\|\phi\|_{0}^{2}.
\end{equation*}%
Once again making use of the identity $2a(b-a)=-|b-a|^{2}+|b|^{2}-|a|^{2}$, $%
a,b\in \mathbf{R}$, we get 
\begin{equation*}
2\phi ^{l}(\phi ^{l}(\tilde{\zeta}^{2})-\phi ^{l}-\zeta^{2;i}\partial
_{i}\phi ^{l})=-|\phi (\tilde{\zeta}^{2})-\phi |^{2}+|\phi (\tilde{\zeta}%
^{2})|^{2}-|\phi |^{2}-\zeta_t^{2;i}\partial _{i}|\phi |^{2}.
\end{equation*}%
Changing the variable of integration and applying the divergence theorem, we
obtain 
\begin{align*}
B_{2}& =-\int_{\mathbf{R}^{d_{1}}}\int_{Z^{2}}|\phi (\tilde{\zeta}^{2})-\phi
|^{2}\pi ^{2}(dz)dx \\
& \quad +\int_{\mathbf{R}^{d_{1}}}\int_{Z^{2}}|\phi |^{2}\left( \det \nabla 
\tilde{\zeta}^{2;-1}-1+\func{div}\zeta^{2}(\tilde{\zeta}^{2;-1})\right) \pi
^{2}(dz)dx \\
& \quad +\int_{\mathbf{R}^{d_{1}}}\int_{Z^{2}}|\phi |^{2}\left( \func{div}
\zeta^{2}-\func{div}\zeta^{2}(\tilde{\zeta}^{2;-1})\right) \pi ^{2}(dz)dx.
\end{align*}%
Changing the variable of integration in the last term of $B_{2}$, we get 
\begin{align*}
& \int_{\mathbf{R}^{d_{1}}}\int_{Z^{2}}|\phi |^{2}\left( \func{div}\zeta^{2}-%
\func{div}\zeta^{2}(\tilde{\zeta}^{2;-1})\right) \pi ^{2}(dz)dx \\
& =\int_{\mathbf{R}^{d_{1}}}\int_{Z^{2}}\left( |\phi |^{2}-|\phi (\tilde{%
\zeta})|^{2}\det \nabla \tilde{\zeta}^{2}\right) \func{div}\zeta^{2}\pi
^{2}(dz)dx \\
& =\int_{\mathbf{R}^{d_{1}}}\int_{Z^{2}}|\phi (\tilde{\zeta}^{2})|^{2}\left(
1-\det \nabla \tilde{\zeta}^{2}\right) \func{div}\zeta^{2}\pi ^{2}(dz)dx \\
& +\int_{\mathbf{R}^{d_{1}}}\int_{Z^{2}}(\phi ^{l}-\phi ^{l}(\tilde{\zeta}%
^{2}))(\phi ^{l}+\phi ^{l}(\tilde{\zeta}^{2}))\func{div}\zeta^{2}\pi
^{2}(dz)dx=:B_{21}+B_{22}.
\end{align*}%
Clearly,%
\begin{equation*}
B_{21}\leq N\int_{Z^{2}}(\bar{K}^{2})^{2}\pi ^{2}(dz)\|\phi\|_{0}^{2},
\end{equation*}%
and applying Holder's inequality,%
\begin{equation*}
B_{22}\leq N\int_{\mathbf{R}^{d_{1}}}\left( \int_{Z^{2}}|\phi (\tilde{\zeta}%
^{2})-\phi |^{2}\pi ^{2}(dz)\right) ^{\frac{1}{2}}\left( \int_{Z^{2}}\left(
|\phi |^{2}+|\phi (\tilde{\zeta}^{2})|^{2}\right) (\bar{K}^{2})^{2}\pi
^{2}(dz)\right) ^{\frac{1}{2}}dx.
\end{equation*}%
Hence, by Remark \ref{rem:PropertiesofH} and Young's inequality,%
\begin{equation*}
B_{2}\leq -\frac{1}{2}\int_{\mathbf{R}^{d_{1}}}\int_{Z^{2}}|\phi (\tilde{%
\zeta}^{2})-\phi |^{2}\pi ^{2}(dz)dx+N\|\phi\|_{0}^{2}.
\end{equation*}%
By Holder's inequality,%
\begin{equation*}
B_{3}\leq N\int_{\mathbf{R}^{d_{1}}}\left( \int_{Z^{2}}|\phi (\tilde{\zeta}%
^{2})-\phi |^{2}\pi ^{2}(dz)\right) ^{\frac{1}{2}}\left(
\int_{Z^{2}}(l^{2})^{2}\pi ^{2}(dz)\right) ^{\frac{1}{2}}|\phi |dx.
\end{equation*}%
Applying Young's inequality again and combining $B_{2}$ and $B_{3},$ we
derive 
\begin{equation}
2\langle \phi ,\mathcal{L}^{2}\phi \rangle _{1}\leq N\|\phi\|_{0}^{2}-\frac{1%
}{4}\int \partial _{i}\phi ^{l}\sigma ^{2;i\varrho }\sigma ^{2;j\varrho
}\partial _{j}\phi ^{l}dx-\frac{1}{4}\int \int_{Z^{2}}|\phi (\tilde{\zeta}%
^{2})-\phi |^{2}\pi ^{2}(dz)dx.  \label{ineq:proofuncorrcoerc}
\end{equation}%
By the divergence theorem, we have 
\begin{equation}
2\langle \phi ,b^{i}\partial _{i}\phi +c^{\cdot \bar{l}}\phi ^{\bar{l}%
}+f\rangle _{0}=2\left( \phi ,f\right) _{0}+(\phi ,\phi \func{div}%
b)_{0}+2(\phi ,c\phi )_{0}\leq N(\| \phi \| _{0}^{2}+\| f\| _{0}^{2}).
\label{ineq:coercfirstzero}
\end{equation}%
Combining \eqref{ineq:proofuncorrcoerc} and \eqref{ineq:coercfirstzero}, we
obtain \eqref{ineq:coercivityuncorrelated}. To obtain the estimate %
\eqref{ineq:coercivitygronwall}, we use \eqref{ineq:coercivityuncorrelated}
and \eqref{ineq:coercivitygronwall}, and estimate the additional terms: 
\begin{equation*}
D:=\left( \sigma ^{1;i\varrho }\partial _{i}\phi +\upsilon ^{1;\cdot \bar{l}%
\varrho }\phi ^{\bar{l}},g^{\varrho }\right) _{0}
\end{equation*}%
and%
\begin{equation*}
2\int_{Z^{1}}\left( \left( \phi (\tilde{\zeta}^{1})-\phi ,h\right)
_{0}+(\rho ^{1}\phi (\tilde{\zeta}^{1}),h)_{0}\right) \pi
^{1}(dz)=:E_{1}+E_{2}.
\end{equation*}%
By the divergence theorem and Holder's inequality,$|D|\leq N\left( ||\phi
||_{0}^{2}+||g||_{1}^{2}\right) .$ Applying Holder's inequality and changing
the variable of integration, we get 
\begin{equation*}
E_{2}\leq \int_{\mathbf{R}^{d_{1}}}\int_{Z^{1}}\left( |\rho ^{1}\phi (\tilde{%
\zeta}^{1})|^{2}+|h|^{2}\right) \pi ^{1}(dz)dx\leq N\left( ||\phi
||_{0}^{2}+\int_{Z^{1}}||h\left( z\right) ||_{0}^{2}\pi ^{1}(dz)\right) .
\end{equation*}%
Then by \eqref{ineq:detHminus1}, Holder's inequality, and Lemma \ref%
{lem:fractionalderviativeestimategrowth}, 
\begin{align*}
E_{1}& =2\int_{Z^{1}}\int_{\mathbf{R}^{d_{1}}}\phi ^{l}\left( h^{l}(\tilde{%
\zeta}^{1;-1})(\det \nabla \tilde{\zeta}^{1;-1}-1)+h^{l}(\tilde{\zeta}%
^{1;-1})-h\right) dx\pi ^{1}(dz) \\
& \leq N\left( \|\phi\|_{0}^{2}+\int_{\mathbf{R}^{d_{1}}}\int_{Z^{1}}\left(
|h|^{2}+|h(\tilde{\zeta}^{1;-1})-h|^{2}\right) \pi ^{1}(dz)dx\right) \\
& \leq N\left( \|\phi\|_{0}^{2}+\int_{Z^{1}}||h(z)||_{\frac{\beta }{2}%
}^{2}\pi ^{1}(dz)\right) .
\end{align*}%
This completes the proof.
\end{proof}

In the following lemma, we verify that Assumption \ref{asm:specSES}$(0,1)$
holds for \eqref{eq:SIDE}. Recall that $\mathcal{W}^{0,1}$ is the space of
all $H^{0}$-valued strongly c\`{a}dl\`{a}g processes $v:\Omega \times
\lbrack 0,T]\rightarrow H^{0}$ that belong to $L^{2}(\Omega \times \lbrack
0,T],\mathcal{O}_{T},dV_{t}d\mathbf{P};H^{1}).$

\begin{lemma}
\label{lem:SpecialCoercivitySIDE}Let Assumption \ref{asm:proofSIDEassumpzero}%
$(d_{2})$ hold. Then there is a constant $N=N(d_{1},d_{2},N_{0},\eta ,\beta
) $ such that for all $v\in \mathcal{W}^{0,1},$ $\mathbf{P}$-a.s.:

\begin{enumerate}
\item 
\begin{gather*}
2\langle v_{t},\mathcal{L}_{t}v_{t}\rangle _{1}dV_{t}+\| \mathcal{N}%
_{t}v_{t}\| _{0}^{2}dV_{t}+\int_{Z^{1}}\| \mathcal{I}_{t,z}v_{t-}\|
_{0}^{2}\eta (dt,dz) \\
+2(v_{t},\mathcal{N}_{t}^{\varrho }v_{t})_{0}dw_{t}^{\varrho
}+2\int_{Z^{1}}(v_{t-},\mathcal{I}_{t,z}v_{t-})_{0}\tilde{\eta}\left(
dt,dz\right) \\
\leq \left(N||v_{t}||_{0}^{2}dV_{t}+\int_{Z^{1}}N\kappa
_{t}(z)||v_{t-}||_{0}^{2}\eta ( dt,dz) +2(v_{t},\mathcal{N}_{t}^{\rho
}v_{t})_{0}dw_{t}^{\varrho }+\int_{Z^{1}}G_{t,z}(v)\tilde{\eta}%
(dt,dz)\right) ,
\end{gather*}
where 
\begin{gather*}
|G_{t,z}(v)|dV_t\leq \bar{\kappa}_{t}(z)||v_{t-}||_{0}^{2}dV_t,\; \forall
z\in Z^{1},\quad |(v_{t},\mathcal{N}_{t}v_{t})_{0}|dV_t\leq N\left| \left|
v_{t}\right| \right| _{0}^{2}dV_t,
\end{gather*}
and $\kappa_t$ and $\bar{\kappa}_t$ are $\mathcal{P}_{T}\times \mathcal{Z}%
^{1}$-measurable processes such that for all $t\in [0,T]$, 
\begin{equation*}
\int_{Z^{1}}\left( \kappa _{t}(z)+\tilde{\kappa}_{t}(z)^{2}\right) \pi
_{t}^{1}(dz)\leq N;
\end{equation*}

\item 
\begin{gather*}
2(\mathcal{N}_{t}^{\varrho }v_{t},g_{t}^{\varrho })_{0}dV_{t}+2\int_{Z^{1}}(%
\mathcal{I}_{t,z}v_{t-},h_{t}(z))_{0}\eta (dt,dz)+2(v_{t},g_{t}^{\varrho
})_{0}dw_{t}^{\varrho }+2\int_{Z^{1}}(v_{t-},h_{t}(z))_{0}\tilde{\eta}(dt,dz)
\\
\leq \left(N \| v_{t-}\| _{0}r_{t}dV_{t}+\int_{Z^{1}}N\| v_{t-}\| _{0}\|
h_{t}(z)\| _{0}\hat{\kappa}_{t}(z)\eta (dt,dz)+2(v_{t},g_{t}^{\varrho
})_{0}dw_{t}^{\varrho }+2\int_{Z^{1}}\bar{G}_{t,z}(v)\tilde{\eta}%
(dt,dz)\right) ,
\end{gather*}%
where 
\begin{gather*}
r_{t}:=\|g_t\|_1+\left| \left| \int_{Z^{1}}\left( h_{t}(\tilde{\zeta}%
_{t}^{1;-1}(z),z)-h_{t}(z)\right) \pi _{t}^{1}(dz)\right| \right| _{0},
\;t\in [0,T], \\
|(v_{t},g_{t})_{0}|dV_t\leq N\|v_{t}\|_{0}\| g_{t}\|_{0}dV_t, \\
\bar{G}_{t,z}(v)dV_t\leq N||v_{t-}||_{0}||h_{t}(z)||_{0},dV_t, \; \forall
z\in Z^{1},
\end{gather*}%
and $\hat{\kappa}_t$ is a $\mathcal{P}_{T}\times \mathcal{Z}^{1}$-measurable
process such that for all $t\in [0,T]$, 
\begin{equation*}
\int_{Z^{1}}\hat{\kappa}_{t}(z)^{2}\pi _{t}^{1}(dz)\leq N.
\end{equation*}
\end{enumerate}
\end{lemma}

\begin{proof}
$(i)$ Owing to the divergence theorem, we have 
\begin{equation*}
2(v_{t},\mathcal{N}_{t}^{\varrho }v_{t})_{0}=(v_{t},u_{t}\func{div}\sigma
_{t}^{1\varrho })_{0}+2(v_{t},\upsilon _{t}^{1\varrho }v_{t})_{0}, \;\;
\forall \varrho\in \mathbf{N},
\end{equation*}
and hence $\mathbf{P}$-a.s., 
\begin{equation*}
|2(v_{t},\mathcal{N}_{t}v_{t})_{0}|dV_t\leq N||v_{t}||^{2}dV_t.
\end{equation*}
By virtue of Lemma \ref{lem:CoercvitySIDE}$(i)$, it suffices estimate 
\begin{equation*}
Q:=2\langle v_{t},\mathcal{J}^{1}_{t,z}v_{t}\rangle
_{1}dV_{t}+\int_{Z^{1}}\| \mathcal{I}_{t,z}v_{t-}\| _{0}^{2}\eta (dt,dz)
+2\int_{Z^{1}}(v_{t-},\mathcal{I}_{t,z}v_{t-})_{0}\tilde{\eta}\left(
dt,dz\right) .
\end{equation*}
An application of divergence theorem shows that 
\begin{equation*}
Q =\int_{Z^{1}}P_{t,z}(u)\eta (dt,dz)+\int_{Z^{1}}G_{t,z}(v)\tilde{\eta}%
(dt,dz),
\end{equation*}
where%
\begin{equation*}
G_{t,z}(v):=2(v_{t-},\rho _{t}^{1}(z)v_{t-})_{0}-(v_{t-},v_{t-}\func{div}%
\zeta_t^{1}(z))_{0},
\end{equation*}
and $P_{t,z}(v):=D_{1}+D_{2}+D_{3}$ with 
\begin{align*}
D_{1}& :=2(v_{t-}(\tilde{\zeta}_{t}^{1}(z)),\rho _{t}^{1}(z)v_{t-}(\tilde{%
\zeta}_{t}^{1}))_{0}-2(v_{t-},\rho _{t}^{1}(z)v_{t-})_{0}, \\
D_{2}& :=\| v_{t-}(\tilde{\zeta}_{t}^{1}(z))\| _{0}^{2}-\| v_{t-}\|
_{0}^{2}+(v_{t-},v_{t-}\func{div}\zeta_t^{1}(z))_{0}, \quad D_{3}:=\| \rho
_{t}^{1}(z)v_{t-}(\tilde{\zeta}_t^1(z))\| _{0}^{2}.
\end{align*}
Given our assumptions, it is clear that $\mathbf{P}$-a.s., 
\begin{equation*}
G_{t,z}(v)dV_t\leq N\left(l_{t}^{1}(z)+\bar{K}_{t}^{1}(z)\right)\| v_{t-}\|
_{0}^{2}dV_t\;\; \mathnormal{\ and }\;\; D_{3}dV_t\leq l_{t}^{1}\left(
z\right) ^{2}||v_{t-}||_{0}^{2}dV_t,
\end{equation*}
where in the last inequality we used the change of variable formula.
Changing the variable of integration and using \eqref{ineq:detHminus1} and %
\eqref{ineq:detHminus1minustr}, we find that $d\mathbf{P}$ -a.s.,%
\begin{align*}
D_{1}dV_t& \leq N\left( v_{t-},v_{t-}\left|\rho _{t}^{1}(\tilde{\zeta}%
_{t}^{1;-1}(z),z)\det \nabla \tilde{\zeta}_{t}^{1;-1}(z)-\rho
_{t}^{1}(z)\right|\right) _{0}dV_t \\
& \leq N\left( l_{t}^{1}(z)\bar{K}_{t}^{1}(z)+\tilde{l}%
_{t}^{1}(z)K_{t}^{1}(z)^{\frac{\beta }{2}}\right) \| v_{t-}\| _{0}^{2}dV_t,
\end{align*}%
and%
\begin{align*}
D_{2}dV_t& =\left( v_{t-},v_{t-}|\det \nabla \tilde{\zeta}_{t}^{1;-1}(z)-1+%
\func{div}\zeta_t^{1}(z)|\right) _{0}dV_t \leq \left( \bar{K}_{t}^{1}(z)^{2}+%
\tilde{K}_{t}^{1}(z)K_{t}^{1}(z)^{\frac{\beta }{2}}\right)N\| v_{t-}\|
_{0}^{2}dV_t.
\end{align*}
Setting 
\begin{equation*}
\kappa_t(z)=l_t^1(z)^2+l_t^1\bar{K}_t^1(z)+\tilde{l}_t^1(z)K_t^1(z)^{\frac{%
\beta}{2}}+\bar{K}_{t}^{1}(z)^{2}+\tilde{K}_{t}^{1}(z)K_{t}^{1}(z)^{\frac{%
\beta }{2}}, \quad \bar{\kappa}_t(z)=l_t^1(z)+\bar{K}_t^1(z), \;z\in Z^1,
\end{equation*}
and appealing to our assumptions, we complete the proof $(i)$.

$(ii)$ By the divergence theorem, we have 
\begin{equation*}
(g_{t}^{\varrho },\mathcal{N}_{t}^{\varrho }v_{t})_{0}= (g_{t}^{\varrho },%
\func{div}\sigma _{t}^{1\varrho }v_{t})_{0}+(\sigma ^{1;i\varrho }\partial
_{i}g_{t}^{\varrho },v_{t})_{0}, \;\;\forall \rho\in \mathbf{N},
\end{equation*}
and thus by the Cauchy-Schwartz inequality, 
\begin{equation*}
|(g_{t},\mathcal{N}_{t}v_{t})_{0}|dV_t\leq N\| v_{t}\| _{0}\| g_{t}\|
_{1}dV_t.
\end{equation*}
Changing the variable of integration, we obtain 
\begin{align*}
(\mathcal{I}_{t,z}v_{t-},h_{t}(z))_{0}& =\left(h_{t}(z),\left(v_{t-}(\tilde{%
\zeta}_{t}^{1}(z))-v_{t-}+\rho _{t}^{1}(z)v_{t-}(\tilde{\zeta}%
_{t}^{1}(z)\right)\right)_{0} \\
& =(h_{t}(\tilde{\zeta}_{t}^{1;-1}(z),z)-h_{t}(z),v_{t-})_{0}+(h_{t}(\tilde{%
\zeta}_{t}^{1;-1}(z),z),(\det \nabla \tilde{\zeta}%
_{t}^{1;-1}(z)-1)v_{t-})_{0} \\
& \quad +(h_{t}(z),\rho _{t}^{1}(z)v_{t-}(\tilde{\zeta}_{t}^{1}(z)))_{0}.
\end{align*}
A simple calculation shows that $\mathbf{P}$-a.s., 
\begin{gather*}
2\int_{Z^{1}}(\mathcal{I}_{t,z}v_{t-},h_{t}(z))_{0}\eta
(dt,dz)+2\int_{Z^{1}}(v_{t-},h_{t}(z))_{0}\tilde{\eta}(dt,dz) \\
\le 2\| v_{t-}\| _{0}r^1_{t}dV_{t}+\int_{Z^{1}}\bar{P}_{t,z}( v) \eta (
dt,dz)+\int_{Z^1}\bar{G}_{t,z}(v)\tilde{\eta}(dt,dz),
\end{gather*}%
where 
\begin{equation*}
r^1_{t}:=\left| \left| \int_{Z^{1}}(h_{t}(\tilde{\zeta}%
_{t}^{1;-1}(z),z)-h_{t}(z))\pi _{t}^{1}(dz)\right| \right| _{0}, \quad \bar{G%
}_{t,z}(v)=(h_{t}(\tilde{\zeta}_{t}^{1;-1}(z),z),v_t),
\end{equation*}%
and%
\begin{equation*}
P_{t,z}(v):=(h_{t}(\tilde{\zeta}_{t}^{1;-1}(z)),v_{t-}(\det \nabla \tilde{%
\zeta}_{t}^{1;-1}(z)-1))_{0}+(h_{t}(z),\rho _{t}^{1}(z)v_{t-}(\tilde{\zeta}%
_{t}^{1}(z)))_{0}.
\end{equation*}
Applying the change of variable formula and Holder's inequality, $\mathbf{P}$%
-a.s. we obtain 
\begin{equation*}
\int_{Z^{1}}\tilde{P}_{t,z}(u)\eta (dt,dz)\leq N\| v_{t-}\| _{0}\int_{Z^{1}}(%
\bar{K}_{t}^{1}(z)+l_{t}^{1}(z))\| h_{t}(z)\| _{0}\eta (dt,dz)
\end{equation*}
and 
\begin{align*}
|\tilde{G}_{t,z}(v)| dV_t\leq N\| v_{t-}\| _{0}\| h_{t}(z)\| _{0}dV_t.
\end{align*}%
This completes the proof.
\end{proof}

Let $d\in \mathbf{N}$. For a function $v\in H^{m}(\mathbf{R}^{d_{1}},\mathbf{%
R}^{d})$, define the linear operator $\mathcal{D}v\in H^{m-1}(\mathbf{R}%
^{d_{1}};\allowbreak \mathbf{R}^{d(d_{1}+1)})$ by 
\begin{equation*}
\mathcal{D}v=\left( \partial _{0}v,\partial _{1}v,\ldots ,\partial
_{d_{1}}v\right) =\tilde{v}
\end{equation*}%
with $\tilde{v}^{l0}=v^{l}$ and $\tilde{v}^{lj}=\partial _{j}v^{l},1\leq
l\leq d,0\leq j\leq d_{1}$ (recall $\partial _{0}v=v$). We define $\mathcal{D%
}^{n}v$ for $n\in \mathbf{N}$ by iteratively applying $\mathcal{D}$ $n$%
-times. Recall that $\Lambda =(I-\Delta )^{\frac{1}{2}}$. It is easy to
check that for each $n\in \mathbf{N}$ and all $u,v\in H^{n+1}(\mathbf{R}%
^{d_{1}},\mathbf{R}^{d})$, 
\begin{align}
(u,v)_{n,d}& =(\Lambda ^{n}u,\Lambda ^{n}v)_{0,d}=(\mathcal{D}^{n}u,\mathcal{%
D}^{n}v)_{0,d\bar{d}_{1}^{n}},  \label{eq:IPDn} \\
(\Lambda u,\Lambda ^{-1}v)_{n,d}& =(\Lambda ^{n+1}u,\Lambda ^{n-1}v)_{0,d}=(%
\mathcal{D}^{n}\Lambda u,\mathcal{D}^{n}\Lambda ^{-1}v)_{0,d\bar{d}_{1}^{n}},
\notag \\
(\mathcal{D}^{n}u,\mathcal{D}^{n}v)_{-1,d\bar{d}_{1}^{n}}& =(\mathcal{D}%
^{n}\Lambda ^{-1}u,\mathcal{D}^{n}\Lambda ^{-1}v)_{0,d\bar{d}%
_{1}^{n}}=(u,v)_{n-1,d},  \notag
\end{align}%
where $\bar{d}_{1}=d_{1}+1.$ Let us introduce the operators $\mathcal{E}(%
\mathcal{L})$, $\mathcal{E}(\mathcal{N}),$ and $\mathcal{E}(\mathcal{I}_{z})$
acting on $\phi =(\phi ^{lj})_{1\leq l\leq d_{2},1\leq j\leq \bar{d}_{1}}\in
C_{c}^{\infty }(\mathbf{R}^{d_{1}},\mathbf{R}^{d_{2}\bar{d}_{1}})$ that are
defined as $\mathcal{L},\mathcal{N}$, and $\mathcal{I},$ respectively, but
with the $d_{2}\times d_{2}$-dimensional coefficients $\upsilon _{t}^{k}$, $%
\rho ^{k}$, and $c$ replaced by the $d_{2}\bar{d_{1}}\times d_{2}\bar{d_{1}}$%
-dimensional coefficients given by 
\begin{align}
\upsilon ^{k;lj,\bar{l}\bar{j}\varrho }& =\upsilon ^{k;l\bar{l}\varrho
}\delta _{j\bar{j}}+1_{j\geq 1}(\partial _{j}\sigma ^{k;\bar{j}\varrho
}\delta _{l\bar{l}}+\partial _{j}\upsilon ^{k;l\bar{l}\varrho }\delta _{\bar{%
j}0}),  \label{eq:upsextended} \\
\rho ^{k;lj,\bar{l}\bar{j}}& =\rho _{t}^{k;l\bar{l}}\delta _{\bar{j}%
j}+1_{j\geq 1}(\partial _{j}\rho ^{k;l\bar{l}}\delta _{\bar{j}0}+(\delta _{l%
\bar{l}}+\rho ^{k;l\bar{l}})\partial _{j}\zeta^{k;\bar{j}}),
\label{eq:rhoextended}
\end{align}%
and%
\begin{equation}
c^{lj,\bar{l}\bar{j}}=c^{l\bar{l}}\delta _{\bar{j}j}+\partial _{j}b^{\bar{j}%
}\delta _{l\bar{l}}+\partial _{j}c^{l\bar{l}}\delta _{\bar{j}%
0}+\sum_{k=1}^{2}\left( \upsilon ^{k;l\bar{l}\varrho }\partial _{j}\sigma
^{k;\bar{j}\varrho }+\int_{Z^{k}}\rho ^{k;l\bar{l}}\partial _{j}\zeta^{k;%
\bar{j}}\pi _{t}^{k}(dz)\right) ,  \label{eq:cextended}
\end{equation}%
for $1\leq l,\bar{l}\leq d_{2}$ and $0\leq j,\bar{j}\leq d_{1}$. The
coefficients $\sigma ^{k},b$, and functions $\zeta^{k},k\in \{1,2\},$ remain
unchanged in the definition of $\mathcal{E}(\mathcal{L})$, $\mathcal{E}(%
\mathcal{N}),$ and $\mathcal{E}(\mathcal{I})$. We define $\mathcal{E}^{n}(%
\mathcal{L})$, $\mathcal{E}^{n}(\mathcal{N})$, and $\mathcal{E}^{n}(\mathcal{%
I})$, for $n\in \mathbf{N}$ by iteratively applying $\mathcal{E}$ $n$-times
by the rules (\ref{eq:upsextended})-(\ref{eq:cextended}) above with $\sigma
^{k},b$, and $\zeta^{k},k\in \{1,2\},$ unchanged. A simple calculation shows
that for all $v\in H^{2}( \mathbf{R}^{d_{1}};\mathbf{R}^{d_{2}}) $,%
\begin{equation*}
\mathcal{D}[\mathcal{L}v] =\mathcal{E}(\mathcal{L})\mathcal{D}v,\quad 
\mathcal{D}[\mathcal{N}^{\varrho }v]=\mathcal{E}(\mathcal{N}^{\varrho })%
\mathcal{D}v,\;\varrho \in \mathbf{N},\quad \mathcal{D}[\mathcal{I}_{z}v]=%
\mathcal{E}(\mathcal{I}_{z})\mathcal{D}v. 
\end{equation*}
Continuing, for all $v\in H^{n+1}(\mathbf{R}^{d_1};\mathbf{R}^d)$ we have 
\begin{equation}
\mathcal{D}^{n}[\mathcal{L}v]=\mathcal{E}^{n}(\mathcal{L})\mathcal{D}%
^{n}v,\quad \mathcal{D}^{n}[\mathcal{N}^{\varrho }v]=\mathcal{E}^{n}(%
\mathcal{N}^{\varrho })\mathcal{D}^{n}v,\;\varrho \in \mathbf{N},\quad 
\mathcal{D}^{n}[\mathcal{I}_{z}v]=\mathcal{E}^{n}(\mathcal{I}_{z})\mathcal{D}%
^{n}v.  \label{eq:DnofOp}
\end{equation}%
If Assumption \ref{asm:coeffmain}$(m,d_{2})$ holds, it can readily be
verified by induction and the definitions (\ref{eq:upsextended})-(\ref%
{eq:cextended}) that Assumption \ref{asm:proofSIDEassumpzero}$(0,d_{2}\bar{d}%
_{1}^{m})$ holds for the coefficients of the operators $\mathcal{E}^{m}(%
\mathcal{L}),$ $\mathcal{E}^{m}(\mathcal{N})$ and $\mathcal{E}^{m}(\mathcal{I%
})$. Moreover, owing to our assumptions on the input data, we have 
\begin{gather*}
\mathcal{D}^{m}\phi \in \mathbf{H}^{0}(\mathbf{R}^{d_{1}};\mathbf{R}^{d_{2}%
\bar{d}_{1}^{m}};\mathcal{F}_{0}),\quad \mathcal{D}^{m}f\in \mathbf{H}^{0}(%
\mathbf{R}^{d_{1}};\mathbf{R}^{d_{2}\bar{d}_{1}^{m}}) \\
\mathcal{D}^{m}g\in \mathbf{\zeta}^{1}(\mathbf{R}^{d_{1}};\ell _{2}(\mathbf{R%
}^{d_{2}\bar{d}_{1}^{m}})),\quad \mathcal{D}^{m}h\in \mathbf{H}^{\frac{\beta 
}{2}}(\mathbf{R}^{d_{1}};\mathbf{R}^{d_{2}\bar{d}_{1}^{m}};\pi ^{1}).
\end{gather*}%
Making use of \eqref{eq:IPDn}, \eqref{eq:DnofOp} and applying Lemma \ref%
{lem:GrowthSIDE} to $\mathcal{E}^{m}(\mathcal{L})$, for all $v\in H^{m+1},$
we obtain 
\begin{equation*}
\Vert \mathcal{L}v\Vert _{m-1}=\Vert \mathcal{D}^{m}[\mathcal{L}v]\Vert
_{-1}=\Vert \mathcal{E}^{m}(\mathcal{L})\mathcal{D}^{m}v\Vert _{-1,d_{2}\bar{%
d}_{1}^{m}}\leq N\Vert \mathcal{D}^{m}v\Vert _{1,d_{2}\bar{d}%
_{1}^{m}}=N\Vert v\Vert _{m+1}.
\end{equation*}%
Likewise, for all $v\in H^{m+1},$ we derive 
\begin{equation*}
\Vert \mathcal{N}v\Vert _{m}\leq N\Vert v\Vert _{m+1},\quad \int_{Z^{1}}||%
\mathcal{I}v||_{m}^{2}\pi ^{1}(dz)\leq N\Vert v\Vert _{m+1}^{2}.
\end{equation*}%
By virtue of Lemma \ref{lem:CoercvitySIDE}, we have that for all $v\in
H^{m+1},$%
\begin{gather*}
2(\Lambda v,\Lambda ^{-1}\mathcal{L}_{t}v)_{m}+||\mathcal{N}%
_{t}v||_{m}^{2}+\int_{Z^{1}}||\mathcal{I}_{t,z}v||_{m}^{2}\pi ^{1}\left(
dz\right) \\
=2(\mathcal{D}^{m}\Lambda v,\mathcal{D}^{m}\Lambda ^{-1}[\mathcal{L}%
_{t}v])_{0}+||\mathcal{D}^{n}[\mathcal{N}_{t}v]||_{0}^{2}+\int_{Z^{1}}||%
\mathcal{D}^{m}[\mathcal{I}_{t,z}v]||_{0}^{2}\pi ^{1}(dz) \\
2\langle \mathcal{D}^{m}v,\mathcal{E}^{m}(\mathcal{L}_{t})\mathcal{D}%
^{m}v\rangle _{1}+||\mathcal{E}^{m}(\mathcal{N}_{t})\mathcal{D}%
^{m}v||_{0}^{2}+\int_{Z^{1}}||\mathcal{E}^{m}(\mathcal{I}_{t,z})\mathcal{D}%
^{n}v||_{0}^{2}\pi ^{1}(dz)\leq N\left\vert \left\vert \mathcal{D}%
^{m}v\right\vert \right\vert _{0,d_{2}\bar{d}_{1}^{m}}^{2}=N\left\vert
\left\vert v\right\vert \right\vert _{m}^{2}
\end{gather*}%
Using a similar argument, we find that for all $v\in H^{m+1},$ 
\begin{equation*}
2(\Lambda v,\Lambda ^{-1}(\mathcal{L}_{t}v+f_{t}))_{m}+||\mathcal{N}%
_{t}v+g_{t}||_{m}^{2}+\int_{Z^{1}}||\mathcal{I}_{t,z}v+h_{t}(z)||_{m}^{2}\pi
^{1}(dz)\leq N||v||_{m}^{2}+N\bar{f}_{t},
\end{equation*}%
where 
\begin{equation*}
\bar{f}_{t}=\Vert f_{t}\Vert _{m}^{2}+\Vert g_{t}\Vert
_{m+1}^{2}+\int_{Z^{1}}\Vert h_{t}(z)\Vert _{m+\frac{\beta }{2}}^{2}\pi
_{t}^{1}(dz).
\end{equation*}%
Therefore, Assumption \ref{asm:mainSES}($m,d_{2}$) holds for the equation (%
\ref{eq:SIDE}). Similarly, using Lemmas \ref{lem:SpecialCoercivitySIDE} and %
\ref{lem:fractionalderviativeestimategrowth}, we find that that Assumption %
\ref{asm:specSES}$(m,d_{2})$ holds for equation (\ref{eq:SIDE}) as well. The
statement of the theorem then follows directly from Theorem \ref%
{th:Degenerate}.

\subsection{Appendix}

For each $\kappa $ $\in (0,1)$ and tempered distribution $f$ on $\mathbf{R}%
^{d_{1}}$, we define 
\begin{equation*}
\partial ^{\kappa }f=\mathcal{F}^{-1}[|\cdot |^{\kappa }\mathcal{F}f(\cdot
)],
\end{equation*}%
where $\mathcal{F}$ denotes the Fourier transform and $\mathcal{F}^{-1}$
denotes the inverse Fourier transform.

\begin{lemma}[cf. Lemma 2.1 in \protect\cite{Ko84}]
\label{lem:FractionaKernel} Let $f:\mathbf{R}^{d_1}\rightarrow \mathbf{R}$
be smooth and bounded. Then for each $\kappa \in (0,1)$, there are constants 
$N_{1}=N_{1}(d_{1},\kappa )$, $N_{2}=N_{2}(d_{1},\kappa )$, and $%
N_{3}=N_{2}(d_{1},\kappa )$ such that for all $x,y,z\in \mathbf{R}^{d_1}$, 
\begin{equation}
\partial ^{\kappa }f(x)=N_{1}\int_{\mathbf{R}^{d}}\left( f(x+z)-f(x)\right) 
\frac{dz}{|z|^{d+\delta }}  \label{eq:integraldefoffrac}
\end{equation}%
and 
\begin{equation}
f(x+y)-f(x)=N_{2}\int_{\mathbf{R}^{d_{1}}}\partial ^{\kappa
}f(x-z)k^{(\kappa )}(y,z)dz,  \label{eq:Kerneldifferencefractional}
\end{equation}%
where 
\begin{equation}
k^{(\kappa )}(y,z)=|y+z|^{\kappa -d}-|z|^{\kappa -d}\quad \mathnormal{and}%
\quad \int_{\mathbf{R}^{d_{1}}}|k^{(\kappa )}(y,z)|dz=N_{3}|z|^{\kappa }.
\label{eq:integralofKernelfractional}
\end{equation}
\end{lemma}

\begin{lemma}
\label{lem:fractionalderviativeestimategrowth} Let $(Z,\mathcal{Z},\pi )$ be
a sigma-finite measure space. Let $H:\mathbf{R}^{d_{1}}\times Z\rightarrow 
\mathbf{R}^{d_{1}}$ be $\mathcal{B}(\mathbf{R}^{d})\otimes \mathcal{Z}$%
-measurable and assume that for all $(x,z)\in \mathbf{R}^{d_{1}}\times Z$, 
\begin{equation*}
|\zeta(x,z)|\le K(z)\quad \mathnormal{and} \quad |\nabla \zeta(x,z)|\le \bar{%
K}(z)
\end{equation*}%
where $K,\bar{K}:Z\rightarrow \mathbf{R}_{+}$ is a $\mathcal{Z}$-measurable
function for which there is a positive constant $N_{0}$ such that for some
fixed $\beta \in (0,2]$, 
\begin{equation*}
\sup_{z\in Z}K(z)+\sup_{z\in Z}\bar{K}(z)+\int_{Z}\left( K(z)^{\beta }+\bar{K%
}(z)^{2}\right) \pi (dz)<N_{0}
\end{equation*}%
Assume that there is a constant $\eta <1$ such that $(x,z)\in \{(x,z)\in 
\mathbf{R}^{d_{1}}\times Z:|\nabla \zeta(x,z)|>\eta \},$ 
\begin{equation*}
| \left( I_{d_{1}}+\nabla \zeta_t(x,z)\right) ^{-1}| \le N_{0}.
\end{equation*}%
Then there is a constant $N=N(d_{1},N_{0},\beta ,\eta )$ such that for all $%
\mathcal{B}(\mathbf{R}^{d_{1}})\otimes \mathcal{Z}$-measurable $h:\mathbf{R}%
^{d_{1}}\times Z\rightarrow \mathbf{R}^{d_{2}}$ with $h\in L^{2}(Z,\mathcal{Z%
},\pi ;H^{\frac{\beta }{2}}(\mathbf{R}^{d_{1}};\mathbf{R}^{d_{2}})),$ 
\begin{equation*}
\int_{\mathbf{R}^{d_{1}}}\left| \int_{Z}\left(
h(x+\zeta(x,z),z)-h(x,z)\right) \pi (dz)\right| ^{2}dx\le N\int_{Z}\| h(z)\|
_{\frac{\beta }{2}}^{2}\pi (dz).
\end{equation*}
\end{lemma}

\begin{proof}
It is easy to see that for any $\mathcal{B}(\mathbf{R}^{d_{1}})\otimes 
\mathcal{Z}$-measurable $h:\mathbf{R}^{d_{1}}\times Z\rightarrow \mathbf{R}%
^{d_{2}}$ such that 
\begin{equation}
\int_{Z}\sup_{x\in \mathbf{R}^{d_{1}}}|\nabla h(x,z)|^{2}\pi (dz)<\infty ,
\label{asm:gradienthsup}
\end{equation}%
the integral $\int_{Z}(h(x+\zeta(x,z),z)-h(x,z))\pi (dz)$ is well-defined.
Moreover, for any $\mathcal{B}(\mathbf{R}^{d_{1}})\otimes \mathcal{Z}$%
-measurable $h:\mathbf{R}^{d_{1}}\times Z\rightarrow \mathbf{R}$ with $h\in
L^{2}(Z,\mathcal{Z},\pi ;H^{\frac{\beta }{2}}(\mathbf{R}^{d_{1}};\mathbf{R}%
^{d_{2}})),$ we can always find a sequence $(h^{n})_{n\in \mathbf{N}}$ of $%
\mathcal{B}(\mathbf{R}^{d_{1}})\otimes \mathcal{Z}$-measurable processes
such that each element of the sequence is smooth with compact support in $x$
and satisfies \eqref{asm:gradienthsup} and 
\begin{equation*}
\lim_{n\rightarrow \infty }\int_{Z}\Vert h(z)-h^{n}(z)\Vert _{\frac{\beta }{2%
}}^{2}\pi (dz)=0.
\end{equation*}%
Thus, if we prove this lemma for $h$ that is smooth with compact support in $%
x$ and satisfies \eqref{asm:gradienthsup}, then we can conclude that the
sequence 
\begin{equation*}
\int_{Z}\left( h^{n}(x+\zeta(x,z),z)-h^{n}(x,z)\right) \pi (dz),\;\;n\in 
\mathbf{N},
\end{equation*}%
is Cauchy in $H^{0}(\mathbf{R}^{d_{1}};\mathbf{R}^{d_{2}})$. We then define 
\begin{equation*}
\int_{Z}\left( h(x+\zeta(x,z),z)-h(x,z)\right) \pi (dz)
\end{equation*}%
for any $\mathcal{B}(\mathbf{R}^{d_{1}})\otimes \mathcal{Z}$-measurable $h:%
\mathbf{R}^{d_{1}}\times Z\rightarrow \mathbf{R}$ with $h\in L^{2}(Z,%
\mathcal{Z},\pi ;H^{\frac{\beta }{2}}(\mathbf{R}^{d_{1}};\mathbf{R}%
^{d_{2}})) $ to be the unique $H^{0}(\mathbf{R}^{d_{1}};\mathbf{R}^{d_{2}})$
limit of the Cauchy sequence. Hence, it suffices to consider $h$ that is
smooth with compact support in $x$ and satisfies \eqref{asm:gradienthsup}.
First, let us consider the case $\beta \in (0,2)$. By Lemma \ref%
{lem:FractionaKernel}, we have 
\begin{gather*}
\int_{\mathbf{R}^{d_{1}}}\left\vert \int_{Z}\left( h(\tilde{\zeta}%
(x,z),z)-h(x,z)\right) \pi (dz)\right\vert ^{2}dx \\
=N_{2}^{2}\int_{\mathbf{R}^{d_{1}}}\left\vert \int_{Z}\int_{\mathbf{R}%
^{d_{1}}}\partial ^{\frac{\beta }{2}}h(x-y,z)k^{(\frac{\beta }{2}%
)}(\zeta(x,z),y))\, dy\pi (dz)\right\vert ^{2}dx \\
\ =:N_{2}^{2}\int_{\mathbf{R}^{d_{1}}}|\int_{Z}A(x,z)\pi (dz)|^{2}dx.
\end{gather*}%
Applying H{\"{o}}lder's inequality and Lemma \ref{lem:FractionaKernel}, for
all $x$ and $z$, we have 
\begin{align*}
A(x,z)& \leq \left( \int_{\mathbf{R}^{d_{1}}}|\partial ^{\beta
/2}h(x-y,z)|^{2}k^{(\frac{\beta }{2})}(\zeta(x,z),y))\, dy\right) ^{\frac{1}{%
2}}\left( \int_{\mathbf{R}^{d_{1}}}k^{(\frac{\beta }{2})}(\zeta(x,z),y))\,
dy\right) ^{\frac{1}{2}} \\
& =\sqrt{N_{3}}\left( \int_{\mathbf{R}^{d_{1}}}|\partial ^{\beta
/2}h(x-y,z)|^{2}k^{(\frac{\beta }{2})}(\zeta(x,z),y))\, dy\right) ^{\frac{1}{%
2}}|\zeta_t(x,z)|^{\frac{\beta }{4}} \\
& \leq K(z)^{\frac{\beta }{2}}\sqrt{N_{3}}\left( \int_{\mathbf{R}%
^{d_{1}}}|\partial ^{\frac{\beta }{2}}h(x-y,z)|^{2}k^{(\frac{\beta }{2}%
)}(\zeta(x,z),y))\, dy K(z)^{-\frac{\beta }{2}}\right) ^{\frac{1}{2}}.
\end{align*}%
Using H{\"{o}}lder's inequality again, for all $x$, we get 
\begin{equation*}
\left\vert \int_{Z}A(x,z)\pi (dz)\right\vert ^{2}\leq
N_{3}N_{0}\int_{Z}\int_{\mathbf{R}^{d_{1}}}|\partial ^{\beta
/2}h(x-y,z)|^{2}k^{(\frac{\beta }{2})}(\zeta(x,z),y))\, dyK(z)^{-\frac{\beta 
}{2}}\pi (dz).
\end{equation*}%
For each $x$ and $z$, we set 
\begin{gather*}
B(x,z)=\int_{|y|\leq 2K(z)}|\partial ^{\frac{\beta }{2}}h(x-y,z)|^{2}k^{(%
\frac{\beta }{2})}(\zeta(x,z),y))\, dy \\
C(x,z)=\int_{|y|>2K(z)}|\partial ^{\frac{\beta }{2}}h(x-y,z)|^{2}k^{(\frac{%
\beta }{2})}(\zeta(x,z),y))\, dy.
\end{gather*}%
Changing the variable integration, for all $x$ and $z$, we find 
\begin{align*}
B(x,z)& \leq \int_{|y+\zeta(x,z)|\leq 3K(z)}|\partial ^{\frac{\beta }{2}%
}h(x-y,z)|^{2}\frac{dy}{|y+\zeta(x,z)|^{d_{1}-\frac{\beta }{2}}} \\
& \quad +\int_{|y|\leq 2K(z)}|\partial ^{\frac{\beta }{2}}h(x-y,z)|^{2}\frac{%
dy}{|y|^{d_{1}-\frac{\beta }{2}}}=:B_{1}(x,z)+B_{2}(x,z),
\end{align*}%
and 
\begin{align*}
B_{1}(x,z)& \leq \int_{|y|\leq 3K(z)}|\partial ^{\frac{\beta }{2}}h((\tilde{%
\zeta}(x,z)-y),z)|^{2}|\frac{dy}{|y|^{d_{1}-\frac{\beta }{2}}} \\
& \leq K(z)^{\frac{\beta }{2}}\int_{|y|\leq 3}|\partial ^{\frac{\beta }{2}%
}h((\tilde{\zeta}(x,z)-yK(z)),z)|^{2}\frac{dy}{|y|^{d_{1}-\frac{\beta }{2}}},
\\
B_{2}(x,z)& \leq K(z)^{\frac{\beta }{2}}\int_{|y|\leq 2}|\partial ^{\beta
/2}h(x-yK(z),z)|^{2}\frac{dy}{|y|^{d_{1}-\frac{\beta }{2}}}.
\end{align*}%
Owing to Remark \ref{rem:PropertiesofH}, for all $z$, the map $x\mapsto
x+\zeta(x,z)=\tilde{\zeta}(x,z)$ is a global diffeomorphism and 
\begin{equation*}
\det \nabla \tilde{\zeta}^{-1}(x,z)\leq N.
\end{equation*}%
for some constant $N=N(N_{0},d_{1},\eta ).$ Thus, by the change of variable
formula, there is a constant $N=N(d_{1},N_{0},\beta ,\eta )$ such that 
\begin{gather*}
\int_{\mathbf{R}^{d_{1}}}\int_{Z}B_{1}(x,z)K(z)^{-\frac{\beta }{2}}\pi (dz)dx
\\
\leq \int_{Z}\int_{|y|\leq 3}\int_{\mathbf{R}^{d_{1}}}|\partial ^{\beta
/2}h((\tilde{\zeta}(x,z)-yK(z)),z)|^{2}dx\frac{dy}{|y|^{d_{1}-\frac{\beta }{2%
}}}\pi (dz) \\
\leq \int_{Z}\int_{|y|\leq 3}\int_{\mathbf{R}^{d_{1}}}|\partial ^{\beta
/2}h((x-yK(z)),z)|^{2}|\det \nabla \tilde{\zeta}^{-1}(x,z)|dx\frac{dy}{%
|y|^{d_{1}-\frac{\beta }{2}}}\pi (dz) \\
\leq N\int_{Z}\int_{\mathbf{R}^{d_{1}}}|\partial ^{\frac{\beta }{2}%
}h(x,z)|^{2}dx\pi (dz)
\end{gather*}%
and 
\begin{equation*}
\int_{\mathbf{R}^{d_{1}}}\int_{Z}K(z)^{\frac{\beta }{2}}B_{2}(x,z)dx\pi
(dz)\leq N\int_{Z}\int_{\mathbf{R}^{d_{1}}}|\partial ^{\frac{\beta }{2}%
}h(x,z)|^{2}dx\pi (dz).
\end{equation*}%
For all $x,y,$ and $z$ such that $|\zeta(x,z)|\leq K(z)\leq \frac{1}{2}|y|$,
we have 
\begin{gather*}
\left\vert \frac{1}{|y+\zeta(x,z)|^{d_{1}-\frac{\beta }{2}}}-\frac{1}{%
|y|^{d_{1}-\beta /2}}\right\vert \\
\leq \left\vert d_{1}-\frac{\beta }{2}\right\vert \left\vert \left( \frac{1}{%
|y+\zeta(x,z)|^{1+d_{1}-\frac{\beta }{2}}}+\frac{1}{|y|^{1+d_{1}-\frac{\beta 
}{2}}}\right) \right\vert |\zeta(x,z)|\leq 3\left\vert d_{1}-\frac{\beta }{2}%
\right\vert \frac{|\zeta(x,z)|}{|y|^{1+d_{1}-\frac{\beta }{2}}},
\end{gather*}%
and hence for all $x$ and $z$, 
\begin{align*}
C(x,z)& =\int_{|y|>2K(z)}|\partial ^{\frac{\beta }{2}}h(x-y,z)|^{2}\left%
\vert \frac{1}{|y+\zeta(x,z)|^{d-\frac{\beta }{2}}}-\frac{1}{|y|^{d-\frac{%
\beta }{2}}}\right\vert\, dy \\
& \leq N\int_{|y|>2K(z)}|\partial ^{\frac{\beta }{2}}h(x-y,z)|^{2}\frac{%
|K(z)|}{|y|^{1+d_{1}-\frac{\beta }{2}}}\, dy \\
& \leq NK(z)^{\frac{\beta }{2}}\int_{|y|>2}|\partial ^{\frac{\beta }{2}%
}h((x-K(z)y),z)|^{2}\frac{dy}{|y|^{1+d_{1}-\frac{\beta }{2}}}.
\end{align*}%
Estimating as above, we find that there is a constant $N=N(d_{1},N_{0},\beta
)$ 
\begin{equation*}
\int_{Z}\int_{\mathbf{R}^{d_{1}}}K(z)^{\frac{\beta }{2}}C(x,z)dx\pi (dz)\leq
N\int_{Z}\int_{\mathbf{R}^{d_{1}}}|\partial ^{\frac{\beta }{2}%
}h(x,z)|^{2}dx\pi (dz).
\end{equation*}%
Combining the above estimates, we obtain the desired estimate for $\beta \in
(0,2)$. Let us now consider the case $\beta =2$. It follows from Remark \ref%
{rem:PropertiesofH} that for each $\theta \in \lbrack 0,1]$, on the set of $%
z\in \{z:\bar{K}(z)<\frac{1}{2}\},$ the map $x\mapsto x+\theta \zeta(x,z)=%
\tilde{\zeta}_{\theta }(x,z)$ is a global diffeomorphism and 
\begin{equation*}
\det \nabla \tilde{\zeta}_{\theta }^{-1}(x,z)\leq N,
\end{equation*}%
for some constant $N=N(N_{0},d_{1}).$ Hence, making use of Taylor's theorem
and the change of variable formula, we find 
\begin{gather*}
\int_{\mathbf{R}^{d_{1}}}\left\vert \int_{Z}\left(
h(x+\zeta(x,z),z)-h(x,z)\right) \pi (dz)\right\vert ^{2}dx \\
\leq \int_{\mathbf{R}^{d_{1}}}\left\vert \int_{\bar{K}(z)\geq \frac{1}{2}%
}\left( h(x+\zeta(x,z),z)-h(x,z)\right) \pi (dz)\right\vert ^{2}dx \\
+\int_{\mathbf{R}^{d_{1}}}\left\vert \int_{\bar{K}(z)<\frac{1}{2}%
}\int_{0}^{1}|\nabla h(x+\theta \zeta(x,z),z)\right\vert d\theta K(z)\pi
(dz)|^{2}dx \\
\leq \pi \left\{\bar{K}(z)\geq \frac{1}{2}\right\}\int_{\bar{K}(z)\geq \eta
}\int_{\mathbf{R}^{d_{1}}}|h(x,z)|^{2}|\det \tilde{\zeta}^{-1}(x,z)+1|dx\pi
(dz) \\
+N_{0}\int_{\bar{K}(z)<\frac{1}{2}}\int_{\mathbf{R}^{d_{1}}}\int_{0}^{1}|%
\nabla h(x,z)|^{2}|\det \nabla \tilde{\zeta}_{\theta }^{-1}(x,z)|d\theta
dx\pi (dz)\leq N\int_{Z}\Vert h(z)\Vert _{1}^{2}\pi (dz).
\end{gather*}%
This completes the proof.
\end{proof}

\textbf{Acknowledgment.} We express our gratitude to Istv\'an Gy\"{o}ngy for
useful discussions we have had on subject of stochastic evolution equations
driven by jump processes.

\bibliographystyle{alpha}
\bibliography{../bibliography}

\begin{thebibliography}{DMGZ94}

\bibitem[Ber77]{Be77}
Melvin~S. Berger.
\newblock {\em Nonlinearity and functional analysis}.
\newblock Academic Press [Harcourt Brace Jovanovich, Publishers], New
  York-London, 1977.
\newblock Lectures on nonlinear problems in mathematical analysis, Pure and
  Applied Mathematics.

\bibitem[DG14]{DaGy14}
Konstantinos~Anastasios Dareiotis and Istv{{\'a}}n Gy{{\"o}}ngy.
\newblock A comparison principle for stochastic integro-differential equations.
\newblock {\em Potential Anal.}, 41(4):1203--1222, 2014.

\bibitem[DMGZ94]{DeGoZa94}
Giuseppe De~Marco, Gianluca Gorni, and Gaetano Zampieri.
\newblock Global inversion of functions: an introduction.
\newblock {\em NoDEA Nonlinear Differential Equations Appl.}, 1(3):229--248,
  1994.

\bibitem[GGK14]{GeGyKr14}
M{\'a}t{\'e} Gerencs{\'e}r, Istv{\'a}n Gy{\"o}ngy, and N.~V. Krylov.
\newblock On the solvability of degenerate stochastic partial differential
  equations in {S}obolev spaces.
\newblock {\em arXiv preprint arXiv:1404.4401}, 2014.

\bibitem[GK81]{GyKr81}
Istv{\'a}n Gy{\"o}ngy and N.~V. Krylov.
\newblock On stochastics equations with respect to semimartingales. {II}.
  {I}t\^o formula in {B}anach spaces.
\newblock {\em Stochastics}, 6(3-4):153--173, 1981.

\bibitem[GK81]{GyKr80}
Istv{\'a}n Gy{\"o}ngy and N.~V. Krylov.
\newblock On stochastic equations with respect to semimartingales. {I}.
\newblock {\em Stochastics}, 4(1):1--21, 1980/81.

\bibitem[GM83]{GrMi83}
B.~Grigelionis and R.~Mikulevi{\v{c}}ius.
\newblock Stochastic evolution equations and densities of the conditional
  distributions.
\newblock In {\em Theory and {A}pplication of {R}andom {F}ields}, volume~49 of
  {\em Lecture Notes in Control and Inform. Sci.}, pages 49--88. Springer,
  Berlin, 1983.

\bibitem[Gy{\"o}82]{Gy82}
I.~Gy{\"o}ngy.
\newblock On stochastic equations with respect to semimartingales. {III}.
\newblock {\em Stochastics}, 7(4):231--254, 1982.

\bibitem[Kom84]{Ko84}
Takashi Komatsu.
\newblock On the martingale problem for generators of stable processes with
  perturbations.
\newblock {\em Osaka J. Math.}, 21(1):113--132, 1984.

\bibitem[KR77]{KrRo77}
N.~V. Krylov and B.~L. Rozovski{\u\i}.
\newblock The {C}auchy problem for linear stochastic partial differential
  equations.
\newblock {\em Izv. Akad. Nauk SSSR Ser. Mat.}, 41(6):1329--1347, 1448, 1977.

\bibitem[KR79]{KrRo79}
N.~V. Krylov and B.~L. Rozovski{\u\i}.
\newblock Stochastic evolution equations.
\newblock In {\em Current {P}roblems in {M}athematics, {V}ol. 14 ({R}ussian)},
  pages 71--147, 256. Akad. Nauk SSSR, Vsesoyuz. Inst. Nauchn. i Tekhn.
  Informatsii, Moscow, 1979.

\bibitem[KR82]{KrRo82a}
N.~V. Krylov and B.~L. Rozovski{\u\i}.
\newblock Characteristics of second-order degenerate parabolic {I}t\^o
  equations.
\newblock {\em Trudy Sem. Petrovsk.}, 8:153--168, 1982.

\bibitem[Len77]{Le77}
E.~Lenglart.
\newblock Relation de domination entre deux processus.
\newblock {\em Ann. Inst. H. Poincar{\'e} Sect. B (N.S.)}, 13(2):171--179,
  1977.

\bibitem[LM14a]{LeMi14}
James-Michael Leahy and Remigijus Mikulevicius.
\newblock On classical solutions of linear stochastic integro-differential
  equations.
\newblock {\em arXiv preprint arXiv:1404.0345}, 2014.

\bibitem[LM14b]{LeMi14b}
James-Michael Leahy and Remigijus Mikulevicius.
\newblock On some properties of space inverses of stochastic flows.
\newblock {\em arXiv preprint arXiv:1411.6277}, 2014.

\bibitem[MR99]{MiRo99}
R.~Mikulevicius and B.~L. Rozovskii.
\newblock Martingale problems for stochastic {PDE}'s.
\newblock In {\em Stochastic partial differential equations: six perspectives},
  volume~64 of {\em Math. Surveys Monogr.}, pages 243--325. Amer. Math. Soc.,
  Providence, RI, 1999.

\bibitem[Ole65]{Ol65}
O.~A. Ole{\u\i}nik.
\newblock On the smoothness of solutions of degenerating elliptic and parabolic
  equations.
\newblock {\em Dokl. Akad. Nauk SSSR}, 163:577--580, 1965.

\bibitem[OR71]{OlRa71}
O.~A. Ole{\u\i}nik and E.~V. Radkevi{\v{c}}.
\newblock Second order equations with nonnegative characteristic form.
\newblock In {\em Mathematical {A}nalysis, 1969 ({R}ussian)}, pages 7--252.
  Akad. Nauk SSSR Vsesojuzn. Inst. Nau\v cn. i Tehn. Informacii, Moscow, 1971.

\bibitem[Par75]{Pa75}
{{\'E}}tienne Pardoux.
\newblock \'{E}quations aux d{\'e}riv{\'e}es partielles stochastiques de type
  monotone.
\newblock In {\em S{\'e}minaire sur les \'{E}quations aux {D}{\'e}riv{\'e}es
  {P}artielles (1974--1975), {III}, {E}xp. {N}o. 2}, page~10. Coll{\`e}ge de
  France, Paris, 1975.

\bibitem[PZ07]{PeZa07}
S.~Peszat and J.~Zabczyk.
\newblock {\em Stochastic {P}artial {D}ifferential {E}quations with {L}{\'e}vy
  {N}oise: {A}n {E}volution {E}quation {A}pproach}, volume 113 of {\em
  Encyclopedia of Mathematics and its Applications}.
\newblock Cambridge University Press, Cambridge, 2007.

\bibitem[PZ13]{PeZa13}
S.~Peszat and J.~Zabczyk.
\newblock Time regularity of solutions to linear equations with {L}{\'e}vy
  noise in infinite dimensions.
\newblock {\em Stochastic Process. Appl.}, 123(3):719--751, 2013.

\bibitem[Roz90]{Ro90}
B.~L. Rozovski{\u\i}.
\newblock {\em Stochastic {E}volution {S}ystems: {L}inear {T}heory and
  {A}pplications to {N}onlinear {F}iltering}, volume~35 of {\em Mathematics and
  its Applications (Soviet Series)}.
\newblock Kluwer Academic Publishers Group, Dordrecht, 1990.
\newblock Translated from the Russian by A. Yarkho.

\bibitem[Tri10]{Tr10}
Hans Triebel.
\newblock {\em Theory of function spaces}.
\newblock Modern Birkh{\"a}user Classics. Birkh{\"a}user/Springer Basel AG,
  Basel, 2010.
\newblock Reprint of 1983 edition [MR0730762], Also published in 1983 by
  Birkh{{\"a}}user Verlag [MR0781540].

\end{thebibliography}

\end{document}